\documentclass{article}

\title{Notes on stochastic integration theory with respect to \cadlag semimartingales and a brief introduction to L\'evy processes} 
\author{M. van den Bosch$^{\rm a, }$\footnote{Email address: \url{vandenboschm@math.leidenuniv.nl}.}}
\usepackage{marvosym}
 
\usepackage{geometry}
\usepackage[english]{babel}
\usepackage[utf8]{inputenc}
\usepackage{dsfont}
\usepackage{mathrsfs}  
\usepackage{color,soul} 
\usepackage[bottom]{footmisc}
  
\usepackage{xargs}
\usepackage[colorinlistoftodos,prependcaption,textsize=scriptsize,disable]{todonotes} 
\renewcommandx{\mark}[2][1=]{\todo[linecolor=black,backgroundcolor=white!25,bordercolor=black,#1]{#2}}
 
\usepackage{mathtools} 
\usepackage{amsmath} 
\usepackage{amssymb} 
\usepackage{amsthm}
\usepackage{xspace}
\usepackage{enumerate} 

\usepackage[font=small,skip=5pt]{caption}
\usepackage{graphicx}
\usepackage[export]{adjustbox}

\usepackage{animate}
\usepackage{media9}

 \usepackage[
 backend=biber,
 citestyle=numeric,
 bibstyle=authoryear,
 dashed=false,
 doi=false,
 isbn=false,
 url=false,
 eprint=false
 ]{biblatex}
 \usepackage{csquotes}
 
 \makeatletter
 \input{numeric.bbx}
 \makeatother
 
 \DeclareFieldFormat[article,inbook,incollection,inproceedings,patent,thesis,unpublished, periodical]{title}{{#1\isdot}}
 \DeclareFieldFormat[article,inbook,incollection,inproceedings,patent,thesis,unpublished, periodical]{volume}{\mkbibbold{#1}}
 
 \renewbibmacro{in:}{}
 \DeclareFieldFormat{pages}{#1}
 
 \usepackage{listings}
 \usepackage{color} 
 \definecolor{mygreen}{RGB}{28,172,0} 
 \definecolor{mylilas}{RGB}{170,55,241}
\usepackage[hidelinks]{hyperref}
\numberwithin{equation}{section}

\newcommand{\citeb}[2]{\cite[\textcolor{black}{#1}]{#2}}

\usepackage{afterpage}
 
\usepackage{changepage,ifthen}

\usepackage{listings}
\usepackage{color} 
\definecolor{mygreen}{RGB}{28,172,0} 
\definecolor{mylilas}{RGB}{170,55,241}
\usepackage{titlesec}
\titlespacing{\section}{0pt}{3.75ex}{1ex}
\titlespacing{\subsection}{0pt}{3ex}{1ex}
\titlespacing{\subsubsection}{0pt}{3ex}{1ex}

\usepackage{imakeidx}\makeindex[columns=2, title=Index, intoc]
\usepackage{ifthen}

\makeatletter
\def\name#1{\def\tempa{#1}\futurelet\next\name@i}
\def\name@i{\ifx\next\bgroup\expandafter\name@ii\else\expandafter\name@end\fi}
\def\name@ii#1{\textnormal{\small\textbf{\tempa}}\index{#1}}
\def\name@end{\textnormal{\small\textbf{\tempa}}\index{\tempa}}
\makeatother
\newcommand{\namethrm}[1]{\textnormal{\small\textbf{(#1)}}\index{#1}}
\newcommand{\namethrms}[2]{\textnormal{\small\textbf{(#1)}}\index{#2}}
\newcommand{\namemath}[1]{#1\index{#1}}

\newtheorem{theorem}{Theorem}[section]
\newtheorem{lemma}[theorem]{Lemma}
\newtheorem{proposition}[theorem]{Proposition}
\newtheorem{corollary}[theorem]{Corollary}

\theoremstyle{definition}
\newtheorem{definition}[theorem]{Definition}
\newtheorem{remark}[theorem]{Remark}

\newtheorem{examplex}[theorem]{Example}
\newenvironment{example}
{\pushQED{\qed}\examplex}
{\popQED\endexamplex}

\newcommand{\R}{\mathbb{R}}
\newcommand{\Rplus}{[0,\infty)}

\newcommand{\N}{\mathbb{N}}
\renewcommand{\epsilon}{\varepsilon}
\renewcommand{\phi}{\varphi}
\newcommand{\dint}[3]{\int_{#1}^{#2}{#3} \, \mathrm{d}}
\newcommand{\cadlag}{càdlàg\xspace}
\newcommand{\caglad}{càglàd\xspace}
\renewcommand{\geq}{\geqslant}
\renewcommand{\leq}{\leqslant}
\renewcommand{\subset}{\subseteq}
\renewcommand{\emptyset}{\varnothing}
\newcommand{\regulated}{$\mathscr{L}_{\textnormal{reg}}$\xspace}

\begin{document}

\maketitle

\begin{center}\small
    \textsc{
    $^{\mathrm a}$Mathematical Institute,  Leiden University,\\ P.O. Box 9512, 2300 RA Leiden, The Netherlands}
    \end{center}

\ 

\begin{abstract}
The purpose of these notes is to distribute, mostly without proofs,    fundamental  definitions and  results concerning the theory of semimartingales and stochastic integration.
The material serves as a foundational guide for those interested in applying these concepts,  particularly in the study of stochastic (functional) differential equations driven by Lévy processes. These notes are adapted from the preliminary chapter of the author's master's thesis (with only minor changes) and are intended to introduce newcomers to the essentials of càdlàg semimartingale theory while also discussing the advantages, limitations, and subtleties as compared to stochastic integration  in the continuous setting.



\end{abstract}



\section{Stochastic integration}\label{Sec1.1}
\noindent This  section is mainly based on \cite{book:protter}, an excellent   source for a lesser succinct overview as provided below. We take a slightly different route in {\S}\ref{Sec1.1.4},  restricting   to the $L^2$-case  as   in \cite{book:chung,book:jacod,book:kallenberg,unpublished:timo}. 
   The reader is  assumed to be a bit familiar with the theoretical background of stochastic integration with respect to a \textit{continuous} martingale, or in a less general setting, with respect to a Brownian motion. We refer for instance to, but not limited to,  \cite{da2014stochastic,book:Kurtz,evans2012stoch,book:kallenberg,book:karatzas,book:kloeden,book:mao,book:revuz,unpublished:peter,twardowska1996wong} for  a complete course on this subject. The latter also motivates the reason why we   emphasise often  on the differences between the continuous and the discontinuous  setting.

\subsection{Prerequisites regarding
the definition of a  semimartingale}\label{Sec1.1.1}
\noindent Let us consider   a probability space $(\Omega, \mathcal{F}, \mathbb{P})$  and a filtration $\mathbb F=(\mathcal{F}_t)_{t \geqslant 0}$, i.e. $\mathcal F_s\subset \mathcal F_t$ holds for $0\leq s\leq t$. We simply say $(\Omega,\mathcal F,\mathbb F, \mathbb P)$ is a \name{filtered probability space}. Recall that $\mathbb F$ is said to be \name{right-continuous}{right-continuous filtration} when $\mathcal F_t=\bigcap_{s>t}\mathcal F_s$ holds for all $t\geq 0.$ From now on, we will always assume  the   filtered probability space satisfies the \name{usual conditions}, i.e., we require the $\sigma$-algebra $\mathcal F_0$ to contain all $\mathbb P$-null sets and the filtration to be right-continuous.   Take $\mathcal F_{\infty}:=\sigma(\bigcup_{t\geq 0}\mathcal F_t)$.

Recall   a stochastic process $X= (X_{t} )_{t \in I}$ on $(\Omega, \mathcal{F},\mathbb F, \mathbb{P})$ is a family of random variables, indexed by $t$, on some index set $I$.  Within this chapter we consider $\mathbb R^d$-valued   processes, thus random variables of the form $X_t:\Omega\to \mathbb R^d$ for all $t\in I$, and    take $I=[0,\infty)$, unless specified otherwise. The   process $X$ is said to be  \name{adapted}\index{process!adapted} when $X_t$ is $\mathcal F_t$-measurable for all $t\geq 0.$ 


A \name{stopping time} $T$ with respect to the filtration $\mathbb F$ is a random variable $T:\Omega\to [0,\infty]$ such that $\{T \leq t\}=\{\omega \in \Omega: T(\omega) \leq t\} \in \mathcal{F}_{t}$ for every $t \geq 0.$ Define for any random variable $T$ and stochastic process $X$ the function $X_T$ on the measurable set $\{T<\infty\}$ by $X_T(\omega)=X_{T(\omega)}(\omega)$. Subsequently, we call $X^T$ defined by
$
	X^T_t=X_{t\wedge T}=X_t\mathds{1}_{\{t<T\}}+X_T\mathds{1}_{\{t\geq T\}}
$
the \name{stopped process} of $X$ at $T$. Consider $\mathcal F_T=\{F\in \mathcal F_\infty:F\cap \{T\leq t\}\in \mathcal F_t\quad \forall t\geq 0\},
$ which is the $\sigma$-algebra associated to  the stopping time $T$.  For $S,T$ stopping times,   $S\leq T$ $\mathbb P$-a.s., one has $\mathcal F_S\subset \mathcal F_T$.

Processes $X$ and $Y$ are called \name{indistinguishable} if the set $\{\omega\in\Omega:X_t(\omega)=Y_t(\omega)\quad\forall t\in I\}$ contains a set of probability one. Hence, sample paths of indistinguishable processes are $\mathbb P$-a.s. equal, and one should think of indistinguishable processes as   the same process. Moreover, a process $Y$ is said to be a \name{version} of $X$ if    $X_t=Y_t$ holds $\mathbb P$-a.s. for all $t\in I$. This basically means     two versions of one another are stochastically equivalent. In practice, it will suffice to consider a particular version; see Remark \ref{remark:1}. More precisely, we will regard  a representative of each equivalence class under the version-relation.  This is   justified by  the following crucial fact: if $X$ and $Y$ are indistinguishable, then  $Y$ is a version of $X$, and the converse implication holds when processes $X$ and $Y$ have sufficiently regular paths.
 
\begin{proposition}[Theorem I.2 of \cite{book:protter}]\label{prop:regulatity}
	Suppose $X=(X_t)_{t\geq 0}$ and $Y=(Y_t)_{t\geq 0}$ are two stochastic processes, with $X$ a version of $Y.$ If $X$ and $Y$ have right-continuous paths $\mathbb P$-a.s., or if $X$ and $Y$ have left-continuous paths $\mathbb P$-a.s., then $X$ and $Y$ are indistinguishable.  
\end{proposition}

%
%
 A   stochastic process $X=(X_t)_{t\geq 0}$ is said to be \name{continuous}{process!continuous}, if for   every $\omega\in \Omega$ the sample path $t\mapsto X_t(\omega)$ is continuous.\footnote{Importantly, note   that \cite{book:protter}, and also  \cite{book:mao} for instance, defines a continuous process   with $\mathbb P$-a.s. continuous paths. Despite the difference with our definition, conform \cite{book:karatzas}, \cite{book:kallenberg} and \citeb{p. 9}{book:chung} for example, such a process would be indistinguishable
 	from one with all paths continuous. Indeed, it is trivial to define a version on the $\mathbb P$-null set to make it
 	continuous on all of $\Omega$. Such a version is unique up to indistinguishability, due to Proposition \ref{prop:regulatity}.
 	The same is true if we replace continuous by, e.g., left- and right-continuous processes. Another valid reason of why we take regularity everywhere, per definition, is to be found in Example \ref{ex:regularity}. } Analogously, we define left- and right-continuous processes. 
 
 \begin{definition}\label{def:first}
 	A function $f:E\subset \mathbb R \rightarrow \mathbb{R}^{d}$ is called \name{càdlàg}  (resp., \name{càglàd}) when $f$ is right (resp., left)-continuous  and attains   left (resp., right) limits everywhere. Likewise, a stochastic process $X$ is called \name{\cadlag}{process!càdlàg} (resp., \name{\caglad}{process!càglàd}) if   every sample path is \cadlag (resp., \caglad). 
 \end{definition} 

Recall, any left- or right-continuous function $f:E\to\R^d$ is Borel-measurable.
 For a \cadlag function $f$ we     denote its left limit at $t\in E$ as  $f(t-)=\lim _{s \uparrow t} f(s)$. For $X$ \cadlag one similarly defines the \caglad process $X_-=(X_{t-})_{t\geq 0}$, with $(X_-)_0=X_{0-}=0$ by convention. Finally, we have  the   process  $\Delta X=(\Delta X_t)_{t\geq 0}  $,  where $\Delta X_t:=X_t-X_{t-}$ is called the \name{jump}{jump process} at $t$. We say that $X$ has \name{bounded jumps} if $\sup_{t\geq 0}\|\Delta X_t\|\leq C<\infty$ holds $\mathbb P$-a.s. for some $C>0,$ where we write $\|\cdot\|$  for the Euclidean norm.
 
 %
 %
 %
 %
 
If $f:E\to\R^d$ is continuous, then $f$ is bounded when $E$ is compact; $\sup_{t\in E}\|f(t)\|\leq M$ for some $M\geq 0$.  Now, suppose  $f$ is   \cadlag.  Then it is easy to   see, assuming $E$ is   compact, that $f$ is bounded too and   the total amount of jumps larger than  $\epsilon>0$, i.e.,  $t\in E:\|f(t)-f(t-)\|\geq \epsilon$, is finite  \citeb{p. 122}{book:billingsley}.
 Therefore,    sample paths of a \cadlag   process $X$ have at most
 countable discontinuities. 
For   \caglad $f$ and $X$ similar results hold.
 
Ultimately, let us recall the definition of martingales, submartingales and supermartingales in continuous time.
\begin{definition}
	A real-valued process $M=(M_t)_{t\geq 0}$ on $(\Omega,\mathcal F,\mathbb F, \mathbb P)$ is   a \name{submartingale}\index{martingale!submartingale} if it is adapted, when we have that the $M_t$ are integrable for all $t\geq 0$, and
	\begin{equation}
		\mathbb E[M_t|\mathcal F_s] \geq M_s \quad\mathbb P\text{-a.s.}
	\end{equation}
holds for all $t\geq s\geq 0.$ If $-M$ is a submartingale, we call $M$ a \name{supermartingale}\index{martingale!supermartingale}. When $M$ is both a submartingale and supermartingale,  it is called a \name{martingale}{martingale!definition of}.
\end{definition}
Note, integrability of the random variable $M_t$ means $\mathbb E|M_t|<\infty.$ We say $M$ is a \name{square integrable martingale}{martingale!square integrable}, an $L^2$-martingale in short, whenever $\mathbb E M_t^2<\infty$ holds for all $t\geq 0$.

\begin{remark}\label{remark:1}
Due to the fact   we assume our filtered probability space satisfies the usual conditions, we may assume without loss of generality that a martingale is \cadlag. \textit{That is, we may---and therefore we will always---take  the unique (up to indistinguishability) \cadlag version without special mention.} It is a     consequence of Doob's regularisation principle  \citeb{p. 134}{book:kallenberg}.
\end{remark}

Before we are able to define a semimartingale, we need to introduce     a local martingale and a finite variation process first. We want to address the fact  that there are several---not necessarily equivalent---definitions of a local martingale, as is also remarked in \citeb{p. 25}{unpublished:peter}.

\begin{definition}
	A \name{fundamental sequence}  is an increasing sequence $(T_{n})_{n \in \mathbb{N}}$\index{$\mathbb N$, natural numbers}  of stopping times, i.e., for any $n\in \N$ we have $T_{n} \leq T_{n+1}$ almost surely, such that $\lim _{n \rightarrow \infty} T_{n}=\infty$ holds almost surely. A stochastic process $M$ is a \name{local martingale}\index{martingale!local} if there exists a fundamental sequence such that for every $n\in \mathbb N,$ the  stopped process $M^{T_{n}}\mathds{1}_{\{T_n>0\}}$ is a martingale.
\end{definition}
A local martingale as defined above is not conform \citeb{p. 330}{book:kallenberg} nor \citeb{p. 12}{book:mao} for example; there one requires $M^{T_n}-M_0$ to be a martingale. Sometimes, within the definition of a local martingale, the    $M^{T_n}$ is not multiplied by $\mathds{1}_{\{T_n>0\}}$; see \citeb{p. 25}{unpublished:peter} for instance. Doing this multiplication however relaxes the integrability condition on $M_0$. This is, in particular, useful in the consideration of stochastic   differential equations with a non-integrable initial condition \citeb{p. 37}{book:protter}. Nonetheless, we want to point out that  one is  mainly concerned with local martingales satisfying $M_0 = 0$ almost surely (see Definition \ref{def:semimart}), and for that case
the differences between the several definitions  disappear. Finally note, the uniform integrability assumption within the definition of a local martingale is actually not necessary \citeb{p. 123}{book:revuz}.  

We say $M$ is a \name{locally square integrable martingale}{martingale!locally square integrable}  when there is an increasing sequence  of stopping times $(T_{n})_{n \in \mathbb{N}}$ tending to infinity almost surely such that $M^{T_{n}}\mathds{1}_{\{T_n>0\}}$ are all square integrable martingales. Obviously a locally square integrable martingale is a local martingale, and the converse holds if and only if $M$ is continuous \citeb{p. 26}{unpublished:peter}.

In these notes we will omit further detail on local martingales and it suffices to know that any martingale is  also a local martingale, but not conversely.

\begin{definition}\label{def:increasing}
	Let $A=(A_t)_{t\geq 0}$ be a real-valued \cadlag process on  $(\Omega, \mathcal{F},\mathbb F, \mathbb{P})$. We call $A$   an \name{increasing process} if it is adapted, and when  $\mathbb P$-almost every sample path $t\mapsto A_t(\omega)$ is non-decreasing. The stochastic process $A$ is said to be a  \name{finite variation process}{finite variation!process} if $A$ can be written as the difference of two increasing processes.
\end{definition}

The equivalent definition of a finite variation process is as follows. Let $[a, b]$  be a compact interval and consider some partition $\pi=\left\{x_{1},x_2, \ldots, x_{n(\pi)},x_{n(\pi)+1}\right\}$, i.e., $a=x_{1}<x_2<\cdots<x_{n(\pi)}<x_{n(\pi)+1}=b .$ The mesh  of a partition is denoted  by $|\pi|=\max _{i \in\{1, \ldots, n(\pi)\}}\left|x_{i+1}-x_{i}\right|.$

\begin{definition}
	 Let $f:[a, b] \rightarrow \mathbb{R}$  be a function  and let ${\Pi}$ be the family of all partitions of the compact interval $[a, b] .$ Then the function $f$ is of \name{finite variation}{finite variation!definition of} if
	\begin{equation}
	F V_{[a, b]}(f):=\sup _{\pi \in \Pi} \sum_{k=1}^{n(\pi)}\left|f(x_{k+1})-f(x_{k})\right|<
	\infty
	\end{equation}
	Similarly, a \cadlag   process $A=(A_t)_{t\geq 0}$ is   a \name{finite variation process}{finite variation!process} if it is adapted, and when   $FV_{[0,t]}(s\mapsto A_s(\omega))<
	\infty$ holds for all compact intervals $[0,t]$, $t\geq 0$, for ($\mathbb P$-almost) all $\omega\in \Omega$.
\end{definition}

Note that both definitions are of a pathwise level, thus the equivalence of the two different definitions follows from \citeb{p. 33}{book:carter} for instance. In addition, recall that a Brownian motion is not a finite variation process (see Example \ref{ex:Brownian}). 

\begin{definition}\label{def:semimart} An adapted, \cadlag process $X=(X_t)_{t\geq 0}$ is  said to be a \name{semimartingale}\index{martingale!semimartingale} if it admits a decomposition
\begin{equation}
X=X_{0}+M+A, \label{eq:decomp}
\end{equation}
where $X_0$ is $\mathcal F_0$-measurable, $M=(M_t)_{t\geq 0}$ is a local martingale with $M_0=0$, and $A=(A_t)_{t\geq 0}$ is a finite variation process  with $A_0=0$.  \end{definition}

The introduced notion above is called a \name{classical semimartingale}{semimartingale!classical}	in \citeb{p. 102}{book:protter}, yet is proven to be equivalent with their notion of a semimartingale. This equivalence particularly shows us that the class of    ``good'' integrators coincides with the class of semimartingales; see also \cite{article:good}. We  want to point out that the definition above suffices. Additionally, by the \name{Doob--Meyer decomposition} \citeb{Thm. III.13}{book:protter}, we deduce that submartingales and supermartingales are semimartingales.\footnote{Note that the result in \citeb{Thm. III.8}{book:protter} assumes the sub- or supermartingale $X$ to be of class D, causing $M$ to be a  (uniform integrable) martingale, but without this assumption $M$ is simply a local martingale \citeb{Thm. III.13}{book:protter}. Also compare these theorems with \cite[Thm. 2.16]{unpublished:peter}.} One  often encounters $\R^d$-valued  stochastic processes $X=(X^1,...,X^d)$ being  $d$-dimensional vectors of semimartingales. 

Finally, recall that in our more general setting, we thus assume a semimartingale not to be necessarily continuous but at least \cadlag. In particular, the components of a \cadlag semimartingale are again \cadlag (by construction). In the continuous case, the following holds. 
\begin{theorem}\label{thm:cont}
	Let $X=(X_t)_{t\geq 0}$ be a semimartingale  and suppose it has continuous paths. Then $X$ admits a decomposition as in \eqref{eq:decomp} where both $M$ and $A$ are continuous processes. Such a decomposition  is moreover unique. 
\end{theorem}
\begin{proof}
	Combining  \citeb{Thm. III.30}{book:protter}  and the corollary in \citeb{p. 130}{book:protter}  yields the assertion.
\end{proof}

Uniqueness of  decomposition \eqref{eq:decomp} needs to be understood as being unique up to {indistinguishability}. When $X$ is continuous, it may as well admit decompositions with $M$  and $A$ not   continuous; see Example \ref{ex:protter130}. In other words,  the uniqueness follows from requiring continuity of $A$ (and $M$) within the decomposition. 

Continuous processes are ``predictable'',  
  a significant property in order to obtain a unique decomposition (as well will see).
  Let us now state a more formal definition of predictability.

\begin{definition}
	The \name{predictable $\sigma$-algebra}{predictable!process} $\mathcal P$\index{$\mathcal P$} on $[0,\infty)\times \Omega$   \label{def:predict}is the  $\sigma$-algebra generated by all the adapted \caglad processes. A stochastic process which is  \name{predictably measurable}, that is, $\mathcal P$-measurable, is called a \name{predictable process}{process!predictable}\index{predictable!process}.
\end{definition}

We often implicitly exploit the   fact that any $\R^d$-valued stochastic process $X=(X_t)_{t\geq 0}$ can be interpreted interchangeably as the associated map $X:[0,\infty)\times \Omega\to\R^d, (t,\omega)\mapsto X_t(\omega)$. A   process $X=(X_t)_{t\geq 0}$ is  called \name{measurable}{process!measurable} when the associated map   is jointly measurable, i.e., $X^{-1}(A)\in \mathcal B([0,\infty))\times \mathcal F$ for all $A\in \mathcal B(\R^d)$. Observe $\mathcal P\subset \mathcal B([0,\infty))\times \mathcal F$. See   {\S}\ref{Sec1.1.4} for more on predictable processes and other measurability types.
\begin{theorem}[Theorem III.30 of \cite{book:protter}]\label{thm:predictable}
	Suppose $X$ is a semimartingale whose finite  variation process $A$ as in \eqref{eq:decomp} is predictably measurable. Decomposition   \eqref{eq:decomp} is then unique.
\end{theorem}
We henceforth use the terminology that a semimartingale $X$ is a \name{special semimartingale}{semimartingale!special}  whenever there is a (unique) decomposition $X=X_0+M+A$ with $A$ being predictable. This decomposition is said to be the \name{canonical decomposition} of $X$ (if it exists). For example, in case $X$ is a process with bounded jumps, i.e., $\sup_{t\geq 0}|\Delta X_t|\leq C<\infty$, it admits a canonical decomposition and is hence a special semimartingale; see \citeb{Thm. III.34}{book:protter}.
 
\begin{remark}\label{remark:2}
The uniqueness follows, both in the continuous and \cadlag setting, quite immediately from the fact that if a local martingale $M$, with $M_0=0$, is a  predictable finite variation process as well, then $M$ is indistinguishable from the zero process \citeb{p. 115}{book:protter}. 
\end{remark}

Many examples of semimartingales will arise in   {\S}\ref{Sec1.2} when discussing Lévy processes. 


\subsection{Stochastic integrals with respect to  semimartingales}\label{Sec1.1.2}
\noindent First of all,  the construction of stochastic integrals with respect to \cadlag semimartingales is not really that  different from when we would consider  continuous semimartingales. There are     a few subtleties, which we try to address as much as possible. Throughout these notes, we will let  \index{$\mathbb D\Rplus$}{$\mathbb D\Rplus$} (resp., \namemath{$\mathbb L\Rplus $}) denote the space of \textit{adapted} \cadlag (resp., \caglad) processes. 

Before we discuss this construction, recall  the notion of Lebesgue--Stieltjes integrability in the setting of stochastic processes.
  The following   is an essential ingredient. 

\begin{proposition}
	Let $A$ be a finite variation process. Then for every $\omega\in\Omega$ for which the sample path $t\mapsto A_t(\omega)$ is   of finite variation on compacts,   there exists a unique signed Borel measure $\mu_A(\,\cdot\,,\omega)$ on $[0,\infty)$ that satisfies $\mu_A(\{0\},\omega)=A_0(\omega)$ and $\mu_A((s,t],\omega)=A_t(\omega)-A_s(\omega)$, for all $0\leq s<t.$ 
\end{proposition}
\begin{proof}
	This follows from   \citeb{Prop. 2.20}{book:kallenberg} and the fact that a finite variation process is almost surely of finite variation on compacts. Alternatively, if $A$ is increasing, one  can     invoke Carathéodory's extension theorem. Subsequently, one  then follows  the lines below this proof.
\end{proof}
Observe that if we write $A=A^1-A^2$ for increasing processes $A^1$ and $A^2$, then $\mu_A(\,\cdot\,,\omega)=\mu_{A^1}(\,\cdot\,,\omega)-
\mu_{A^2}(\,\cdot\,,\omega)$ where $\mu_{A^1}(\,\cdot\,,\omega)$ and $
\mu_{A^2}(\,\cdot\,,\omega)$ are simply Borel measures on $\Rplus$, for those  elements $\omega\in\Omega$ where the sample paths $\omega\to A_t(\omega)$ are   of finite variation on compacts, or, equivalently, where the sample paths  $\omega\mapsto A^1_t(\omega)$ and $\omega\mapsto A^2_t(\omega)$ are indeed non-decreasing.

  \begin{definition} 
Suppose $A$ is a finite variation process,\label{def:Stieltjes} and  $H:\Rplus\times \Omega\to\R$  jointly measurable such that  $H(\,\cdot\,,\omega)|_{[0,t]}:[0,t]\to\R$ is $\mu_A(\,\cdot\,,\omega)$-integrable for $\mathbb P$-almost every $\omega\in\Omega$, for any $t\geq 0$.  Introduce the associated process $H=(H_t)_{t\geq 0}$, $H_t:\Omega\to\R,\omega\mapsto H(t,\omega)$.  Then   
 	\begin{equation}
 		\dint0\cdot  {H_s}A_s=\left(\dint0t  {H_s}A_s\right)_{t\geq 0}
 	\end{equation}
 	denotes the \name{integral of $H$ with respect to $A$}{Lebesgue--Stieltjes integral}, which is a stochastic process defined pathwise in the \name{Lebesgue--Stieltjes}{Lebesgue--Stieltjes integral} sense.  That is, for $\mathbb P$-almost every $\omega\in\Omega$ 
 	we let 
 	\begin{equation}
 		\left(\dint0t {H_s}A_s\right)(\omega)=\dint0t {H_s(\omega)}A_s(\omega)=\int_{\Rplus} \mathds{1}_{[0,t]}(s){H(s,\omega)}\mu_A(\mathrm{d}s,\omega), \quad t\geq 0. \label{eq:Leb-Stiel}
 	\end{equation} 
 	For the (possibly) remaining $\omega \in \Omega$, 
 	 we    simply set $\dint0\cdot  {H_s(\omega)}A_s(\omega)=0$. 
 \end{definition}



To ensure  the  integrability as in the above, 
 we require $\dint0t{|H_s|}|A|_s<\infty$ $\mathbb P$-a.s., for all $t\geq 0,$ where $(|X|_t)_{t\geq 0}$ denotes the \name{total variation process} of a finite variation process $X$.  That is, an increasing process such that for $\mathbb P$-almost every $\omega\in\Omega$ we have \begin{equation}
  |X|_t(\omega)=FV_{[0,t]}(s\mapsto X_s(\omega)),\end{equation} for all $ t\geq 0.$
  As usual, we trivially extend $|X|$ onto the whole sample space $\Omega.$  Of course, when we write $X=X^1-X^2$ for increasing processes $X^1,X^2$, then $|X|\leq X^1+X^2$ holds\vspace{-.05cm} $\mathbb P$-a.s.. 
%
%
%
A sufficient integrability condition for $H$ is:    $\mathbb E\dint0t{|H_s|}|A|_s<\infty,$ for all $ t\geq 0$. 

 C\`adl\`ag and \caglad processes   are   measurable (in case   $\mathbb P$-a.s. regularity is assumed, this may no longer be true according to Example \ref{ex:regularity}).  Suppose $H$ is such a process, then any sample  path is $\mu_A(\,\cdot\,,\omega)$-integrable,  as well, since  $s\mapsto H_s(\omega)$ restricted to any compact interval $[0,t]$ is known  to be bounded, and therefore Lebesgue--Stieltjes integrable. 
 
 As a matter of fact, suppose   $H=(H_t)_{t\geq 0}$ is measurable   process such that $\mathbb P$-almost  every sample path is continuous, e.g., whenever $H$ is a continuous process (i.e., has continuous paths everywhere). Then for every sequence   $(\pi^m)_{m\in \N}$  of partitions of $[0,t]$ with mesh tending to zero,  that is, $\lim_{m\to\infty} |\pi_m|=0$,  we obtain 
\begin{equation}
\dint 0t{H_s}A_s= H_0A_0+	\lim_{m\to\infty} \sum_{k=1}^{n(\pi^m)}H_{\tau_k^m}(A_{t^m_{k+1}}-A_{t^m_k})\quad \mathbb P\text{-a.s.}, \label{eq:Riemann}
\end{equation}
for    $\tau_k^m$ such that $t_k^m\leq \tau_k^m\leq t_{k+1}^m$. Indeed, it is  commonly known that if the \name{Riemann--Stieltjes integral}{Riemann--Stieltjes integral!classical definition} exists, it coincides with the Lebesgue--Stieltjes integral, and that the existence is as sured for the class of continuous integrands \citeb{p. 552}{article:horst}. In the classical sense of Riemann--Stieltjes integration, i.e., as in the above, the integrand cannot share points of
discontinuity with the integrator function. Nevertheless, if one considers the  generalised \name{Darboux}{Riemann--Stieltjes integral!Darboux's definition} definition of Riemann--Stieltjes integrals, thus  requiring every sequence $(\pi^m)_{m\in\N}$ of partitions   to become finer, i.e., $\pi^1\subset \pi^2\subset ...$,    then it is (only) necessary  that the integrand and integrator  are not simultaneously discontinuous from the left or
from the right  \citeb{p. 553}{article:horst}; see also \citeb{ p. 160}{book:apostol}.    Theorem C in \citeb{p. 553}{article:horst}---both a necessary and sufficient \vspace{-.05cm}statement---enables us to conclude that the Lebesgue--Stieltjes integral $\int_0^tH_s\,\mathrm dA_s$, with $H\in\mathbb L\Rplus$  and $A\in\mathbb D\Rplus$ by   convention, coincides with the existing Riemann--Stieltjes integral   (in the Darboux sense).

For an elaborate discussion on Lebesgue--Stieltjes integration, many elementary textbooks on real analysis suffice due to the pathwise character in Definition \ref{def:Stieltjes}. We   refer    to \cite{book:apostol,book:carter,book:kallenberg,article:horst} and the references therein.

 Take note that the limits in the above are way stronger than we actually need.
For instance, a  relatively easy computation  with   help of the dominated convergence theorem  shows that
\begin{equation}\lim _{m \rightarrow \infty} \sup _{0 \leq u \leq t}\left|H_0A_0+\sum_{k=1}^{n(\pi^m)} H_{t_k^m}(A_{t^m_{k+1}\wedge u}-A_{t^m_k\wedge u})-\int_{0}^u H_s \,\mathrm d A_s\right|=0,\label{eqref:similar}
\end{equation} holds for $H\in\mathbb  L\Rplus$; see \citeb{p. 16}{unpublished:timo}. We typically take $\tau_k^m=t_k^m$ in \eqref{eq:Riemann} as specific choice.
In words, we are able to approximate Lebesgue--Stieltjes integrals   to $[0,t]$ by   left Riemann sums, where the  convergence can even be in the  \textit{ucp-sense} (this  type of convergence will be introduced soon, prior to Definition \ref{def:integral}).  We merely exploit this in Proposition \ref{thm:import_form!}.

 Now, let us continue with defining integration with respect to a general semimartingale $X$. Recall that  if one aims for a pathwise definition of
 such an integral, one  finds themselves in a quite  hopeless position. A definition as in \eqref{eq:Leb-Stiel} is no option and if we   consider the Riemann--Stieltjes integral as in \eqref{eq:Riemann} with $\tau_k^m=t_k^m$ only and $A$ replaced by $X$, then this pointwise limit converges $\mathbb P$-a.s. for every  continuous process $H$  if and only if $X$ is a finite variation process; see \citeb{Sec. I.8}{book:protter} or \citeb{Sec. 6}{unpublished:peter}.  If one considers a limit in probability instead,  it will be of no help either   because then via   similar  reasoning $X$ still needs to be a finite variation process \citeb{p. 44}{book:protter}.
%
%
  The  key   to overcome this problem is to restrict ourselves to those integrands  that cannot see  into the future (of the integrator), that is, the integrands need to be adapted processes.
 
 \begin{definition}
 	A process $H$ is called \name{simple predictable}{simple predictable process}\index{process!simple predictable} whenever it can be written as
 	\begin{equation}
 		H_t=H_0\mathds{1}_{\{0\}}(t)+\sum_{k=1}^nH_k\mathds{1}_{(T_{k},T_{k+1}]}(t),
 	\end{equation}
 where $0=T_0\leq T_1\leq T_2\leq ...\leq T_{n+1}<\infty$ is a finite sequence of stopping times, such that the real-valued random variables $H_k$ are $\mathcal F_{T_{k}}$-measurable. Usually, we set $T_1=T_0=0$.
 \end{definition}
Observe, when allowing $H_k$ to be $[-\infty,\infty]$-random variables, one imposes the additional condition   that $|H_k|<\infty$ holds $\mathbb P$-almost surely. In that setting, a simple predictable process $H$ is adapted, and  has $\mathbb P$-a.s.\ \caglad paths. Since we  interpret processes in $\mathbb D \Rplus$  with  regularity everywhere, it follows that $H$ is indistinguishable from a $\mathbb D\Rplus$ process. Because we consider random variables that may not attain the values $\{-\infty,\infty\}$, we circumvent indistinguishability and obtain $H\in \mathbb D\Rplus$ immediately.
\begin{definition}
	Let $H$ be a simple  predictable process and  suppose $X$ is  a semimartingale. Then the \name{stochastic integral of $H$ with respect to $X$}{stochastic integral}   is defined by the $\mathbb D\Rplus$-process
	 \begin{equation}
	 	\label{eq:si_simple}\dint0\cdot {H_s}X_s=\left(\dint0t {H_s}X_s\right)_{t\geq 0}:=H_0X_0+\sum_{k=1}^nH_k(X^{T_{k+1}}-X^{T_k}).
	 \end{equation}
\end{definition}
 Notice that $\dint0t {H_s}X_s$ is   often written as \begin{equation}\dint{[0,t]}{} {H_s}X_s,\end{equation}  and we refer to  {\S}\ref{Sec1.1.5} for more information regarding this notation. Integration in the above sense     acts linearly on both integrands and integrators, i.e.,
 \begin{equation} 
\label{eq:linear1}  	\dint0\cdot {(\alpha H_s+\beta K_s)}X_s=\alpha\dint0\cdot{H_s}X_s+\beta\dint0\cdot{K_s}X_s\end{equation}
and
 \begin{equation}\label{eq:linear2}  \dint0\cdot {H_s}(\alpha X+\beta Y)_s=\alpha \dint0\cdot{H_s}X_s+\beta \dint0\cdot{H_s}Y_s
 \end{equation}
hold for every $\alpha,\beta\in\R$, semimartingales $X$ and $Y$, and simple predictable processes $H$ and $K.$
  Additionally observe, when $X$ is continuous,  the expression in \eqref{eq:si_simple} also is. Next, we generalise the integrand from being a simple predictable process $H$ to an  adapted \caglad process $H$.  
 
 Recall  a sequence of processes $(Y^n)_{n\in \N}$ converges to some stochastic process $Y$ \name{uniformly on compacts in probability}, abbreviated as \name{ucp}{uniformly on compacts in probability!\textit{abbreviated as} ucp}, whenever for all time $t$ fixed we have that $\sup_{s\leq t}|Y^n_s-Y_s|$ converges to 0 in probability.  It is tacitly   understood that the supremum is measurable when the processes   are either \cadlag or \caglad (because then the supremum can be restricted to the countable set of rational times, which will be clearly measurable). Moreover, let us  endow $\mathbb D\Rplus$ and $\mathbb L\Rplus$ with the  (compatible) metric  \begin{equation}\index{$d_{\textnormal{ucp}}$}d_{\textnormal{ucp}}(X,Y):=\sum_{n\in\N}2^{-n}( 1\wedge \sup_{s\leq n}|X_s-Y_s|).\end{equation}
 We implicitly use    the identification that indistinguishable processes are the same process. Also,   the desired   property holds:  $d_{\text{ucp}}(X^n,X)\to 0$ if and only if $X^n$ converges to $X$ in the ucp-sense. Finally, observe $(\mathbb D\Rplus,d_{\text{ucp}})$ and  $(\mathbb L\Rplus,d_{\text{ucp}})$ are complete metric spaces \citeb{p. 57}{book:protter}.
 
 One can show, as an intermediate result, that if $(H^n)_{n\in\N}$ is a sequence of simple predictable processes that converges to 0 in the ucp sense, then the sequence $(\dint0\cdot {H^n_s}X_s)_{n\in\N}$  converges to 0 in the ucp sense too \citeb{p. 58}{book:protter}.  This result makes a stochastic integral of an $\mathbb L\Rplus$-process  independent of the  chosen approximating sequence of simple predictable  processes.

 \begin{definition} \label{def:integral}
 	Take $H\in \mathbb L\Rplus$ and $X$ a semimartingale. Then we define the \name{stochastic integral of $H$ with respect to $X$}{stochastic integral}, denoted by $\dint0\cdot {H_s}X_s$, as the  $\mathbb D\Rplus$-process satisfying
 	\begin{equation}
 	\dint0\cdot {H_s^n}X_s\stackrel{\text{ucp}}{\longrightarrow}\dint0\cdot {H_s}X_s,
 	\end{equation}
 	 where $(H^n)_{n\in \N}$  is an arbitrary sequence of simple predictable processes such that $H^n\stackrel{\text{ucp}}{\longrightarrow} H$.
 \end{definition} 
Acknowledge that a stochastic integral is only determined $\mathbb P$-a.s.. As it will become usual, we   are notably using   the identification that indistinguishable processes are the same process. 
  Moreover,   stochastic integrals  as in Definition \ref{def:integral} are well-defined due to the fact that the metric space $(\mathbb D\Rplus,d_{\text{ucp}})$ is complete, and because the space of simple predictable processes is dense in the space of \caglad processes $\mathbb L\Rplus$ under the ucp-metric \citeb{p. 57}{book:protter}. 
One can extend the notion of a stochastic integral, with respect to any semimartingale, by considering a proper class of predictably measurable integrands, see  {\S}\ref{Sec1.1.4}, and thus we are not just restricted to adapted \caglad processes only. Before we will  dig deeper into this extension, we  state a few significant results.
 
 First of all,    a finite variation process $A$, alone, is of course a semimartingale too. This implies however that there are two different meanings of $\dint0\cdot {H_s}A_s$ when $H$ is an adapted \caglad process.  Fortunately, these   definitions coincide---hence, there is no ambiguity---thanks to the next result. 
 
 \begin{proposition}[Theorem II.17 of \cite{book:protter}]\label{prop:stieltjes}
 	Suppose $H\in \mathbb L\Rplus $ and let $A$ be a finite variation process. The Lebesgue--Stieltjes integral of $H$ with respect to $A$  is indistinguishable from the stochastic integral $\dint0\cdot {H_s}A_s$.
 \end{proposition}

Therefore, for $X$ a semimartingale, we   usually separate the stochastic integral as
\begin{equation}
\label{eq:seperation}	\dint0t{H_s}X_s=H_0X_0+\dint0t{H_s}A_s+\dint0t{H_s}M_s,
\end{equation}
where  we use that the   linearity as in \eqref{eq:linear1} and  \eqref{eq:linear2} clearly holds for adapted \caglad processes also. Usually $\dint0\cdot {H_s}A_s$ is to   be understood as a  Lebesgue--Stieltjes integral, and in particular,    equation   \eqref{eq:seperation} consequently suggests that it would have sufficed to define stochastic integrals with respect to local  martingales only.  

Secondly, one can   deduce that a Lebesgue--Stieltjes integral with respect to a finite variation process $A$ is again a finite variation process. Differently put,   being of  finite variation  is  preserved by stochastic integration.
This is an immediate consequence of Proposition \ref{prop:stieltjes} and the fact
\begin{equation}\label{eq:totalvar}
	\left|\dint0\cdot {H_s}A_s\right|_t=\dint0t{|H_s|}{|A|_s},\quad t\geq 0.
\end{equation}
As a matter of fact, there are multiple properties being preserved under stochastic integration. A few of  them are listed below.

\begin{theorem}
	Suppose $H\in \mathbb L\Rplus$. \label{thm:preserve}Then the following properties hold.
	\begin{enumerate}[\normalfont(i)]
		\item Suppose $X$ is a semimartingale, then $\dint0\cdot {H_s}X_s$ is a semimartingale too. In addition, if  $K$ is an adapted \caglad process, we have $\dint0\cdot {K_s}Y_s=\dint0\cdot {K_sH_s}X_s$ with $Y=\dint0\cdot {H_s}X_s$;
		\item Suppose $X$ is a finite variation process, then $\dint0\cdot {H_s}X_s$ is a finite variation process too;
			\item Suppose $X$ is a local martingale, then $\dint0\cdot {H_s}X_s$ is a local martingale too;
			\item Suppose   $X$ is a locally square integrable martingale, then $\dint0\cdot {H_s}X_s$ is  a locally square integrable  martingale too; 
			\item Suppose $X$ is a continuous semimartingale, then  $\dint0\cdot {H_s}X_s$ is also continuous.
	\end{enumerate}
\end{theorem}
\begin{proof}
	For (i), see   \citeb{Thm. II.19}{book:protter}; for (ii), we refer to equation \eqref{eq:totalvar} and the discussion preceding it; for (iii), see  \citeb{Thm. III.29}{book:protter}; and, furthermore,   for (iv),
	 see   \citeb{Thm. II.20}{book:protter}. The proof of (iii) is far from trivial and it turns out easier to show (iv) first \cite{book:protter}. 
     Finally, for (v)  we  exploit the fact that   the space of adapted continuous processes endowed with  $d_\text{ucp}$ is   complete as well.
\end{proof}

Observe   that stochastic integration does not preserve the martingale property in general. That is, if $X$ is a (true) martingale
  without additional integrability assumptions,  then this implies that  the stochastic integral $\dint0\cdot {H_s}X_s$ is not necessarily   a (true) martingale as well and therefore a local martingale only (by Theorem \ref{thm:preserve}). Corollary \ref{cor:truemart}, for instance, shows us an integrability condition to ensure $\dint0\cdot {H_s}X_s$ is a (true) square integrable martingale whenever $X$ is.

Moreover,  the theorem above together with \eqref{eq:seperation} yields the  immediate result that   $\dint0\cdot {H_s}X_s$ is a semimartingale with $\dint0\cdot {H_s}A_s$ being the finite variation part and $\dint0\cdot{H_s}M_s$ the local martingale part. 
This decomposition is unique (up to indistinguishability) when   the stochastic integral $\dint0\cdot {H_s}A_s$ is predictably measurable; e.g., take $A$ to be continuous.

%
%

%

 Lastly, we want to point out   the following. Suppose $Y$ is some stochastic process and let $\pi$ be a   \name{random partition} of finite stopping times, i.e., $0=T_1\leq T_2\leq ...\leq T_{n(\pi)}<T_{n(\pi)+1}<\infty,$ where $n(\pi)\in \N$ is a   positive integer. Subsequently, we define the \name{sampled process} $Y$ at $\pi$ by
 \begin{equation}
 	Y^\pi:=Y_0\mathds{1}_{\{0\}}+\sum_{k=1}^{n(\pi)}Y_{T_k}\mathds1_{(T_k,T_{k+1}]}.
 \end{equation}
 For $X$ a semimartingale, one easily verifies  that $\dint0\cdot {Y^\pi_s}X_s=Y_0X_0+\sum_{k=1}^{\pi(n)}Y_{T_k}(X^{T_{k+1}}-X^{T_k})$ holds for any adapted process $Y$ that is either a \caglad or \cadlag process. 
 
 \begin{theorem}[Theorem II.21 of \cite{book:protter}]
 	Let $X$ be a semimartingale and assume we either have\label{thm:approx} $Y\in \mathbb L\Rplus$ or $Y\in\mathbb D\Rplus$. Suppose  that $(\pi_n)_{n\in \N}$ is a  sequence of random partitions,   $0=T_0^n\leq T_1^n\leq ...\leq T_{k_n}^n\leq T_{k_n+1}^n<\infty$ where the $T^n_k$ are stopping times, such that
 	\begin{enumerate}[\normalfont(i)]
 		\item $\lim_{n\to\infty}T^n_{k_n}=\infty$ holds $\mathbb P$-\textnormal{a.s.}, and
 		\item $|\pi_n|=\sup_k|T^n_{k+1}-T^n_k|\to 0$ holds $\mathbb P$-\textnormal{a.s.} when taking $n\to\infty.$
 	\end{enumerate}
 Then $\dint0\cdot {Y^{\pi_n}_s}X_s=Y_0X_0+\sum_{k=1}^{k_n}Y_{T^n_k}(X^{T^n_{k+1}}-X^{T^n_k})\stackrel{\textnormal{ucp}}{\longrightarrow} Y_0X_0+\dint0\cdot {Y_{s-}}X_s.$
 \end{theorem}
    Theorem \ref{thm:approx} gives rise to an additional and simultaneously intuitive understanding of stochastic integrals; the approximating sums converge to the stochastic integrals for appropriate processes $Y$. This result is particularly enlightening regarding simulations; see Appendix \ref{B}. 
    
    Moreover, we want to point out   this result is very special because, in general, the sampled   processes $Y^{\pi_n}$ do not convergence in the ucp sense to the process $Y_-$  \citeb{p. 30}{unpublished:chris}. If it were the case that $Y^{\pi_n}\stackrel{\text{ucp}}{\longrightarrow}Y_-$ holds, then $\dint0\cdot {Y^{\pi_n}_{s}}X_s\stackrel{\text{ucp}}{\longrightarrow}Y_0X_0+\dint0\cdot {Y_{s-}}X_s$ were to be  true by definition.

\subsection{Quadratic variation processes and compensators}\label{Sec1.1.3}
 

\noindent Before we are able to extend the definition of a stochastic integral with a more general integrand, we    require  the concept of \textit{quadratic variation} of a (semi)martingale.

\begin{definition}
	Let $X$ and $Y$ be semimartingales. The \name{quadratic covariation process} of $X$ and $Y$, denoted   $[X,Y]=([X,Y]_t)_{t\geq 0}$, is defined as the $\mathbb D\Rplus$-process
	\begin{equation}
	\label{eq:intbyparts}	[X,Y]=XY-\dint0\cdot {X_{s-}}Y_s-\dint0\cdot {Y_{s-}}X_s.
	\end{equation}
The \name{quadratic variation process} of $X$ is  $[X]=[X,X].$ 
\end{definition}

Clearly, the operation $(X,Y)\mapsto [X,Y]$ defines a symmetric bilinear map, and therefore     the polarisation identity holds, i.e.  $ [X,Y]=\frac{1}{2}([X+Y,X+Y]-[X,X]-[Y,Y]).$ As a matter of fact,  the quadratic variation $[X]$ is an increasing process \citeb{p. 66}{book:protter}. By polarisation, we have     $[X,Y]$ is a finite variation process.
Observe that the definition above is conform \cite{book:jacod,book:kallenberg,book:protter}. Other  literature, such as  \citeb{p. 49}{unpublished:timo}, defines the quadratic covariation process $ [X,Y] $ as the unique process satisfying the intuitive  limit property in equation \eqref{eq:limit}.

\begin{theorem}[Theorem II.23 of \cite{book:protter}]
	Suppose $(\pi_n)_{n\in\N}$ is a sequence of random partitions as in Theorem \ref{thm:approx}. For $X$ and $Y$ semimartingales, we  \label{thm:ucpquad} have\begin{equation}
	\label{eq:limit}	X_0Y_0+\sum_{k=1}^{k_n}(X^{T_{k+1}^n}-X^{T_k^n})(Y^{T_{k+1}^n}-Y^{T_k^n})\stackrel{\textnormal{ucp}}{\longrightarrow}[X,Y].
	\end{equation}
\end{theorem}
We remark that $ [X,Y]_0=X_0Y_0 $ and $ \Delta[X,Y]=\Delta X\Delta Y $ hold. In particular, since $[X]$ is an increasing process and $\Delta[X]_t=(\Delta X_t)^2$ for all $t$, we can decompose $[X]$ pathwise into its continuous part, for which we use the notation $ [X]_t^c, $ and its pure jump part:
\begin{equation}
\label{eq:contpure}	[X]_t=[X]_t^c+X_0^2+\sum_{0<s\leq t}(\Delta X_s)^2,\quad t\geq 0.
\end{equation}
Note that $\Delta X$ itself is not \cadlag, though it is adapted. As a consequence of an earlier observation, we have $\mathbb P$-a.s.\ that $t\mapsto \Delta X_t$ equals zero except for at most countably many $t$, making the sum   in  \eqref{eq:contpure} well-defined over a countable subset of $ (0,t] $.  
A semimartingale $X$ is said to be a \name{quadratic pure jump} process if $ [X]^c=0 $. A finite variation process is a quadratic pure jump process \citeb{p. 71}{book:protter}; in particular, see Examples \ref{ex:poisson} and \ref{ex:compound}. 
Analogously, we denote $[X,Y]^c$ for the pathwise continuous part of  $[X,Y]$. A more elaborate discussion on the path-by-path  continuous part can be found in  Chapter I.4 of \citeb{p. 38}{book:jacod}. 

\begin{proposition}
	Suppose $X$ is a semimartingale and let\label{prop:quadvarFV} $A$ be   a  finite variation process. Their quadratic covariation is given by
	\begin{equation}
	[X, A]_{t}=X_0A_0+\int_{0}^{t} \Delta X_s d A_s=X_0A_0+\sum_{0<s \leq t} \Delta X_{s} \Delta A_{s},\quad t\geq 0.
	\end{equation}
	In particular, if either  $X$ or $A$ is   continuous, then $[X, A]=X_0A_0 .$
\end{proposition} 
\begin{proof}
	This is a special case of \citeb{Thm. II.28}{book:protter}, where one uses the fact that any finite variation process is a quadratic pure jump process; see  \citeb{Thm. II.26}{book:protter}.
\end{proof}
If $X, Y$ are semimartingales and $V, W$ are continuous finite variation processes, with initial data $V_0=W_0=0$, then
\begin{equation}
[X+V, Y+W]=[X, Y].
\end{equation}
 That is,  when calculating a covariation, we can simply disregard any continuous   terms of finite variation   added
to the processes $X,Y$.  Often,  in the  continuous setting, this property is taken into account within the definition of covariation; see \citeb{p. 42}{unpublished:peter} for instance.

Now let us state a crucial fact. Suppose $M$ is a locally square integrable martingale, then
\begin{equation}
\label{eq:square}	M_t^2-[M]_t=2\dint0t{M_{s-}}M_s
\end{equation}
is also a locally square integrable martingale by Theorem \ref{thm:preserve}. Consequently,   the following holds.

\begin{proposition}[Corollary II.3 of \cite{book:protter}]\label{prop:intermed}
	Let $M=(M_t)_{t\geq 0}$ be a local martingale. Then $M$ is a square integrable martingale if and only if $ \mathbb E[M]_t<\infty $ for all $t\geq 0.$ In either case, we have the identity $ \mathbb E M_t^2=\mathbb E[M]_t $ for all $t\geq 0.$
\end{proposition}

 To complete the   overview, we state the following sort of substitution rule. 
 

\begin{proposition}[Theorem    II.29 of \cite{book:protter}]\label{thm:import_form}
	Suppose $X$ and $Y$ are semimartingales, and let $H,K\in \mathbb L\Rplus$. Then we have
	\begin{equation}
		\left[\dint0\cdot {H_s}X_s,\dint0\cdot {K_s}Y_s\right]=\dint0\cdot {H_sK_s}[X,Y]_s.
	\end{equation}
\end{proposition}
  
 
\begin{corollary}
 Suppose $X$ is a  square integrable martingale \label{cor:truemart}	and let $H\in\mathbb L\Rplus$  be such that $\mathbb E\dint0t {H_s^2}[X]_s<\infty$ holds for all $t\geq 0$. Then the stochastic integral $\dint0\cdot {H_s}X_s$ is a square integrable martingale. In particular, we obtain
 \begin{equation}
 \label{eq:formula}	\mathbb E\left(\dint0t{H_s}X_s\right)^2=\mathbb E \dint0t{H_s^2}[X]_s,
 \end{equation}
for all $t\geq 0$.
\end{corollary}

The latter result is fundamental,  yet an direct corollary of the previous two results in combination with Theorem \ref{thm:preserve}. Recall   that \eqref{eq:formula}  is strongly connected to Itô's original treatment of stochastic integration (as is also pointed out in {\S}\ref{Sec1.2}).  

Throughout  this section, we\footnote{Warning,  several authors like \citeb{p. 12}{book:protter} and \citeb{p. 11}{book:jacod} define $L^2$-martingales / square integrable martingales by   additionally assuming that the uniform integrability condition $\sup_{t\geq 0}\mathbb E M_t^2<\infty$ holds, implying the existence of a closing variable $M_\infty$. We avoid  exploiting such an  assumption in our notes.}    will   denote \index{$\mathscr M^2$ and $\mathscr M^2_{\textnormal{loc}}$}{$\mathscr M^2$} (resp., {$\mathscr M^2_{\textnormal{loc}}$}) for the class of  $L^2$-martingales  (resp., locally square integrable martingales)  $M$ with $M_0=0$.  Note $\mathscr M^2\subset \mathscr M^2_{\textnormal{loc}}$. 
We   point out that in addition to the quadratic variation  process $[M]$, there is another   increasing process with similar bracketing notation.

\begin{definition}
	Suppose we have $M\in \mathscr{M}^2$ (resp. $M\in \mathscr M^2_{\textnormal{loc}}$). Then there exists a unique predictable increasing process $\langle M\rangle$, the so-called \name{predictable quadratic variation}{predictable!quadratic variation}, such that $M^{2}-\langle M\rangle$ is a martingale (resp. local martingale).
\end{definition}
The existence and uniqueness   follows from a special case of
the Doob--Meyer decomposition; see also \cite[Sec. 3]{unpublished:peter}  and \citeb{Sec. I.4}{book:jacod}.
Once more, recall that uniqueness is to be understood as being unique up to indistinguishability. For general semimartingales, the predictable quadratic variation may not exist  \citeb{p. 123}{book:protter}.
It turns out that the  processes  $[M]$ and $\langle M\rangle$ coincide in the continuous setting; see Proposition \ref{prop:coincide}. In general, however, this is no longer true (as expected). A concrete example where this  fails  is to be found in {\S}\ref{Sec1.2.2}.  
  
  \begin{proposition}\label{prop:coincide}
  	Suppose $ M \in \mathscr M^2_{\textnormal{loc}}$ is   continuous. Then both the processes $[M] $ and $\langle M\rangle$ are continuous. Moreover, $\langle M\rangle =
  	[M]$ holds \textnormal{(}up to indistinguishability\textnormal{)}.
  \end{proposition}
\begin{proof}
 	 Theorem \ref{thm:preserve} parts (iii) and  (v)  give us that the quadratic variation process $[M]$ is continuous, namely recall the formula in \eqref{eq:intbyparts}, and consequently $[M]$ is clearly predictable. By equation \eqref{eq:square} we can conclude $[M]$ is a predictable increasing process making $M^2-[M]$ a (continuous) local martingale. The uniqueness yields $\langle M\rangle =
	[M]$   up to indistinguishability.
\end{proof}

%
%
%
%

Suppose $M$ is   an $L^2$-martingale, we now only know $M^2-[M]$ is  a local martingale. In fact, we have   $M^2-[M]$ is a true martingale. See \citeb{p. 79}{book:Kurtz}  or \citeb{p. 102}{unpublished:timo} for a complete proof; it   basically follows from the fact that    ucp-convergence in equation \eqref{eq:limit} can be replaced by $L^1$-convergence  for $L^2$-martingales \citeb{p. 67}{book:Kurtz}.
Furthermore, one can also  define the predictable  quadratic covariation, for instance by means of the polarisation formula, but we do not need that here.


Often the   quadratic variation process $[M]$ is  referred to as the \name{bracket process} and $\langle M\rangle$ as the \name{angle bracket process}.  The reason for introducing both processes is not only because we want to raise awareness of the subtle difference between the two in the general discontinuous setting, but also because the following discussion yields   a useful result; see equation \eqref{eq:formula2}.

%
%
%
%
%
%
%
%
%

\begin{definition}
	A finite variation process $A$ is said to be of \name{locally integrable variation} if there exists a fundamental sequence of stopping times $ (T^n)_{n\in \N} $ such that $ \mathbb E|A|_{T^n}<\infty $, for each $n\in\N$.
\end{definition}

  The existence and uniqueness of the compensator (see next definition) is a consequence of \name{Rao's Theorem} \citeb{Thm. III.15}{book:protter},  which follows non-trivially from the    Doob--Meyer decomposition.

\begin{definition}\label{def:compensator}
	Let $A$ be a finite variation process with $A_0=0$, and assume $A$ is of locally integrable variation. The unique predictable finite variation process $\tilde A$ with $\tilde A_0=0$ such that $A-\tilde A$ is a local martingale, is called the \name{compensator} of $A$.
\end{definition}

 Proposition \ref{prop:intermed} implies that,  for every $M\in \mathscr{M}^2$,  we have that  $[M]$ is of locally integrable variation. Even more is true, namely the compensator of  $[M]$ is the angle bracket process $\langle M\rangle$. Observe $\mathbb E[M]_t=\mathbb E\langle M\rangle _t$ for all $t\geq 0$ (thanks to the law of total expectation and $ M_0=0$).

 
\begin{lemma}\label{lem:compensator}
	Suppose $H\in \mathbb L\Rplus$  and let $A$   be an increasing process  of locally integrable variation with $A_0=0$. Then for $\tilde A$, the compensator of $A$, we have
	\begin{equation}
		\mathbb E\dint0t{H_s}A_s=		\mathbb E\dint0t{H_s}\tilde A_s,\quad t\geq 0,
	\end{equation} 
if either $\mathbb E\dint0t {H_s}A_s<\infty$   or   $\mathbb E\dint0t {H_s}\tilde A_s<\infty$ holds for all $t\geq 0$.
\end{lemma}
\begin{proof}(Inspired by \citeb{p. 118}{book:protter}.) Note $A-\tilde A$ is a local martingale, hence $\dint0\cdot{H_s}(A-\tilde A)_s$   is a local martingale also by Theorem \ref{thm:preserve} part (iii). Subsequently, by the martingale property and the law of total expectation, we have $\mathbb E\dint0t{H_s}(A-\tilde A)_s=0$ for all $t\geq 0.$ The assertion now follows due to either one of the assumptions and because stochastic integration acts linearly on integrators; recall  equation \eqref{eq:linear2} is valid for adapted \caglad processes.
\end{proof}

Observe $\mathbb E\dint0t {H_s}A_s<\infty$    for all $t\geq 0$ is thus equivalent to  $\mathbb E\dint0t {H_s}\tilde A_s<\infty$ for all $t\geq 0.$
Combining now Lemma \ref{lem:compensator} and Corollary \ref{cor:truemart} yields the following result.
\begin{corollary}
	 Suppose $M\in \mathscr{M}^2$  and\label{cor:handy} assume that  $H\in \mathbb L\Rplus$ satisfies the integrability assumption $\mathbb E\dint0t {H_s^2}\langle M\rangle _s<\infty$, for all $t\geq 0$. Then   $\dint0\cdot {H_s}M_s\in \mathscr{M}^2$  holds, and we have  
	\begin{equation}
		\label{eq:formula2}	\mathbb E\left(\dint0t{H_s}M_s\right)^2=\mathbb E \dint0t{H_s^2}\langle M\rangle _s,
	\end{equation}
	for all $t\geq 0$.
\end{corollary}

Equation \eqref{eq:formula2}  turns out   useful, as claimed previously, because $\langle M\rangle $ has in  {\S}\ref{Sec1.2.2} a nicer appearance than $[M]$. Conform to the upcoming subsection, we  may write $H\in \mathcal P(M)$ if the stochastic process  $H$ satisfies the conditions in either Corollary \ref{cor:truemart} or Corollary \ref{cor:handy}. 
 
Ultimately, we state a predictable version of Proposition \ref{thm:import_form} for quadratic variations.
\begin{proposition} \label{thm:import_form!}
	Suppose $M\in \mathscr{M}^2_{\textnormal{loc}}$  and   $H\in \mathbb L\Rplus$. Then    $\dint0\cdot {H_s}M_s\in \mathscr{M}^2_{\textnormal{loc}}$  with 
	\begin{equation}
		\left\langle \dint0\cdot {H_s}M_s \right\rangle =\dint0\cdot {H_s^2}\langle M\rangle_s.\label{eq:concl}
	\end{equation}
\end{proposition}
\begin{proof}

	For a direct reference, where it is actually stated in  more generality (see also {\S}\ref{Sec1.1.4}), we refer to \citeb{Thm. I.4.40}{book:jacod}.
	Alternatively, we can also prove it   with the knowledge above. Observe that
		$\dint0\cdot {H_s}M_s\in \mathscr{M}^2_{\text{loc}}$  holds by Theorem \ref{thm:preserve} part (iv). By Proposition \ref{thm:import_form} we deduce
		\begin{equation}
			\int_0^\cdot H_s^2\,\mathrm d[M]_s
		\end{equation}
	has $\left\langle \dint0\cdot {H_s}M_s \right\rangle$ as compensator. Simultaneously, we have that
	\begin{equation}
	\int_0^\cdot H_s^2\,\mathrm d[ M] _s-\int_0^\cdot H_s^2\,\mathrm d\langle M\rangle _s=	\int_0^\cdot H_s^2\,\mathrm d([M]-\langle M\rangle )_s
	\end{equation}
is a local martingale, since $[M]-\langle M\rangle $ is a local martingale (see Theorem \ref{thm:preserve} part (iii)).

 Due to the fact $\langle M\rangle $ is a predictable process, we obtain that the increasing stochastic integral $\int_0^\cdot H_s^2\,\mathrm d\langle M\rangle _s$ is     predictably measurable as well. Indeed, for $t\geq 0$ fixed, write $\int_0^t H_s^2\,\mathrm d\langle M\rangle _s$ as a limit of Riemann sums, each one of which is predictable on $[0,t]$, which convergences in ucp to $\int_0^t H_s^2\,\mathrm d\langle M\rangle _s$. Consequently, the limit is predictable, showing that the integral of interest is predictable too \citeb{p. 157}{book:protter}.
  By  uniqueness of the compensator, we derive the identity in \eqref{eq:concl}.  
\end{proof}

\subsection[Stochastic integrals  with respect to \texorpdfstring{$L^2$}{L²}-martingales:    feasible extensions on the class of integrands]{Stochastic integrals  with respect to $L^2$-martingales:   feasible extensions on the class of integrands}\label{Sec1.1.4}
\noindent For a general semimartingale, there is a very  good   reason why we restrict ourselves to    adapted \caglad integrands in the first place.      Example \ref{ex:generalise} and the subsequent remark demonstrate  for instance that, in general,  adapted \cadlag integrands    already fail  several  preservation properties of Theorem \ref{thm:preserve}.     Throughout these notes, we mostly work with integrands in $\mathbb L\Rplus,$ hence we keep this section brief.  Recall that stochastic integrals with respect to Brownian motion can have quite general integrands; we  will particularly demonstrate when this can be achieved.

In general, we can go beyond the class of adapted \caglad processes. Both   \cite[Ch. IV]{book:protter} and  \cite[Ch. I]{book:jacod} construct, for a general semimartingale $X$,   stochastic integrals with predictable integrands   satisfying the ``$X$-integrable'' property; see \citeb{p. 165}{book:protter}.  In these notes, we    consider    extending  the class of integrands for $L^2$-martingales only.  
 Notice that by equation \eqref{eq:seperation} we will subsequently obtain a proper definition for stochastic integrals   with respect to those semimartingales with a   square integrable martingale part.
%
%
The extensions are in line with, for example, \cite{book:chung} and \cite{unpublished:timo}. It is   beneficial to restrict ourselves to  (\cadlag) $L^2$-martingales, because then the constructions of stochastic integrals  are more or less analogous to the construction we know for  continuous martingales (e.g.,  a Brownian motion, see also \cite[Ch. 4]{unpublished:timo}, for instance).  

%

 \subsubsection*{The extension to   the class of predictable integrands}
 
\noindent  Our   framework is thus limited to $M\in\mathscr M^2$. For locally square integrable martingales $M\in \mathscr{M}^2_{\textnormal{loc}}$ we refer to \citeb{p. 43}{book:chung} or \citeb{p. 161}{unpublished:timo}.

For starters, we discuss some equivalent definitions for the \name{predictable $\sigma$-algebra}{predictable!$\sigma$-algebra} $\mathcal P$\index{$\mathcal P$}, and in particular recall Definition \ref{def:predict}.  More precisely, it means that $\mathcal{P}$ is generated by events of the form $\left\{(t, \omega)\in \Rplus\times \Omega: X_{t}(\omega) \in B\right\}$ where $X$ is an adapted \caglad process and $B \in \mathcal{B}(\R)$. 


\begin{definition}
	Subsets of $\Rplus\times \Omega$ which are of the type $(s, t] \times F_s$ where $0 \leq s<t<\infty$ and $F_s \in \mathcal{F}_{s},$ or  of the type $\{0\} \times F_{0}$ where $F_{0} \in \mathcal{F}_{0} ,$ are called  \name{predictable rectangles}{predictable!rectangles}. The set of all predictable rectangles is denoted by $\mathcal{R}$\index{$\mathcal{R}$}.
\end{definition}

\begin{proposition}[Theorems 3.1 and 3.2 of \cite{book:chung}]\label{prop:predict}\index{process!predictable}\index{predictable!process}
	The following $\sigma$-algebras equal the  predictable $\sigma$-algebra $\mathcal P$:
		\begin{enumerate}[\normalfont(i)]
		\item The  $\sigma$-algebra on $[0,\infty)\times \Omega$     generated by all the adapted continuous processes;
		\item The  $\sigma$-algebra on $[0,\infty)\times \Omega$     generated by all the adapted left-continuous processes;
		\item The  $\sigma$-algebra on $[0,\infty)\times \Omega$     generated by all predictable rectangles, i.e., $\sigma(\mathcal R)$. 
	\end{enumerate}
	
\end{proposition}

As the proposition above  already might suggest, not every
adapted \cadlag process is predictable.   Indeed, in Lemma \ref{lem:pred}  we show  that Poisson processes are not predictable (such a process jumps at ``completely random'' times).


On the contrary, an arbitrary deterministic
process---i.e., one that does not depend on a particular sample $\omega\in\Omega$---is predictable.  
\begin{proposition}\label{prop:predictdet}
	Let $X$ be a deterministic $\R$-valued process, i.e., $X_t(\omega)=f(t)$ for all $\omega\in\Omega$, with $f:\Rplus\to\R$    some Borel-measurable function.
	 Then $X$ is predictably measurable.
\end{proposition}
\begin{proof}
	(Inspired by \citeb{p. 136}{unpublished:timo}.) Clearly a deterministic process is adapted. The assertion that  $X$ is $\mathcal P$-measurable follows quite immediately from Proposition \ref{prop:predict} part (iii) and the commonly known fact that the Borel $\sigma$-algebra $\mathcal B(\Rplus)$ equals $\sigma(\{(a,b]:a,b\geq 0\})$.
\end{proof}

The above provides us additional intuitive understanding of what it means to be predictable. Consequently, it also indicates that $\mathcal P$ is quite a natural $\sigma$-algebra to look at (in the first place). For other (non-trivial) predictably measurable processes, we refer to \citeb{p. 492}{book:kallenberg}.

A main  ingredient in  $L^2$-martingale integrator theory, are the so-called Doléans measures.  

\begin{definition}
	Given a   martingale $M\in\mathscr M^2$, we define its \name{Doléans measure} $\mu_{M}$\index{$\mu_{M}$} on the predictable $\sigma$-algebra $\mathcal{P}$ by
	\begin{equation}
		\mu_{M}(A)=\int_{\Omega} \int_{0}^{\infty} 1_{A}(s, \omega)\, \mathrm{d}[ M]_{s}(\omega) \mathbb{P}(\mathrm d \omega)=\mathbb{E} \int_{0}^{\infty} 1_{A} \,\mathrm{d}[ M],\quad A\in\mathcal P. \label{eq:doleans}
	\end{equation}
\end{definition}

It is not immediately clear why $\mu_M$ is well-defined on $\mathcal P$. Notice that, for every fixed $\omega\in\Omega$,  the function $t \mapsto \mathbf{1}_{A}(t, \omega)$   integrated over   $[0,\infty)$ with respect to the Lebesgue--Stieltjes measure $\mu_{[M]}(\,\cdot\,,\omega)$ results into a random variable $\Omega\to[-\infty,\infty]$, which in turn gets  averaged over the probability space $(\Omega, \mathcal{F}, \mathbb P)$. In more detail, the $\mathcal F$-measurability of
\begin{equation}
\Omega\to[-\infty,\infty],\quad  \omega\mapsto 	\int_{0}^{\infty} 1_{A}(s, \omega) \,\mathrm{d}[ M]_{s}(\omega) 
\end{equation} is clear when $A\in\mathcal R$ is a predictable rectangle. Applying the monotone class theorem   yields  the $\mathcal F$-measurability  for all $A\in\mathcal \sigma(R)=\mathcal P $ (see also Exercise 3.13 of \citeb{p. 112}{unpublished:timo}).  
 Since this random variable is non-negative, it makes sense to allow expectations to attain the value $\infty.$ One easily verifies   $\sigma$-additivity, which follows from the monotone convergence theorem (conform to the problem in Remark \ref{remark:clarify}), hence $\mu_M$ is indeed a measure.

 Moreover, by Proposition \ref{prop:intermed}, we have 
\begin{equation}
	\mu_{M}([0, t] \times \Omega)=\mathbb E([M]_{t})=\mathbb E(M_{t}^{2})<\infty,
\end{equation}
for all $t\geq 0,$ implying that the Doléans measure $\mu_{M}$ is $\sigma$-finite. This observation enables us to give an alternative interpretation of the Doléans measure: observing that $\mu_M(R)<\infty$ holds for every $R\in\mathcal R$, we may conclude  $\mu_M$ uniquely  extends to a measure on $\mathcal P$ by Carathéodory's extension theorem, where uniqueness is a consequence of  the $\sigma$-finiteness. The  explicitly constructed measure  above and the one obtained from Carathéodory yield the same measure (again  because of the $\sigma$-finite property).
 
  \begin{remark}\label{remark:continuous}
  	The formula   \eqref{eq:doleans}   makes sense for all  sets $A$ in $\mathcal{B}([0, \infty)) \times \mathcal{F}$ as well;   consider the rectangles $(s,t]\times F$ with $F\in\mathcal F$ (instead of in $\mathcal F_s)$. But,   in case we want to extend the class of integrands for stochastic integrals beyond the class of predictable processes in a useful manner---i.e., such that (semi)martingales  properties remain preserved---formula \eqref{eq:doleans} does not always provide the suitable extension  \citeb{p. 136}{unpublished:timo}.   Example \ref{ex:generalise} demonstrates this.
  	
  	In particular, whenever the integrator $M$ is assumed to be continuous, we can make use of the Doléans measure on $\mathcal{B}([0, \infty)) \times \mathcal{F}$, also denoted by $\mu_M$, to extend the notion of stochastic integrals on the class of suitably integrable \textit{progressively measurable processes} (defined soon). 
  \end{remark}

\begin{remark} Often, as is done in \citeb{p. 33}{book:chung}, one introduces the Doléans measure   as the unique extension of the pre-measure $\lambda_M:\mathcal R\to[0,\infty]$ to a measure on the $\sigma$-algebra $\mathcal P$, where
	\begin{equation}
		\lambda_M((s,t]\times F_s)=\mathbb E[\mathds 1_{F_s}(M_t-M_s)^2]=\mathbb E[\mathds 1_{F_s}(M_t^2-M_s^2)],\label{eq:secondeq}
	\end{equation}
for all $0\leq s<t<\infty$ and $F_s\in\mathcal F_s.$ The second equality in \eqref{eq:secondeq} is due to the martingale property. Both approaches result into the same measure, because   $M^2-[M]$ is a true martingale:
\begin{align*}
	\lambda_M((s,t]\times F_s)&=\mathbb E\big[\mathds 1 _{F_s}\big(M_t^2-[M]_t+[M]_t-M_s^2\big)\big]\\&=\mathbb E\big[\mathds 1 _{F_s}(M_s^2-[M]_s+[M]_t-M_s^2\big)\big]\\	&=\mathbb E\big[\mathds 1_{F_s}\big([M]_t-[M]_s\big)\big]\\&=\mu_M((s,t]\times F_s).
\end{align*}
Via analogous reasoning (using that $M^2-\langle M\rangle$ is a true martingale), we deduce that $\mu_M=\nu_M$ holds  with $\nu_M$ being  the $\sigma$-finite measure on the predictable $\sigma$-algebra $\mathcal P$  defined by 
	\begin{equation}
	\nu_{M}(A)=\int_{\Omega} \int_{0}^{\infty} 1_{A}(s, \omega) \mathrm{d}\langle M\rangle _{s}(\omega) \mathbb{P}(\mathrm d \omega)=\mathbb{E} \int_{0}^{\infty} 1_{A} \mathrm{d}\langle  M\rangle,\quad A\in\mathcal P. \label{eq:doleans2}
\end{equation}
In conclusion, there are   three (non-trivially) equivalent definitions of a Doléans measure.
	Surprisingly, the observation  $\mu_M=\nu_M$ is---to the best of our knowledge---nowhere   highlighted. For continuous $L^2$-martingales $M$, the definitions of $\mu_M$ and $\nu_M$ correspond trivially.
\end{remark}

We measure the size of square martingales in $\mathscr M^2$ via the quantity \begin{equation}\|M\|_{\mathscr M^2}:=\sum_{n\in\N}2^{-n}( 1\wedge \mathbb EM_n^2),\quad M\in \mathscr M^2.\end{equation}  
Warning, $\|\cdot\|_{\mathscr M^2}$ fails to be a norm, but it does satisfy the triangle inequality. Therefore, we are able to endow the space $\mathscr M^2$ with the metric $d_{\mathscr M^2}(M,N)=\|M-N\|_{\mathscr M^2}, M,N\in\mathscr M^2.$ As usual, we    use  the identification that indistinguishable processes are the same process.

For any measurable, adapted   process  $X$ (recall Remark \ref{remark:continuous}), and for all time $t\geq0$, we define the $L^2$-norm
\begin{equation}
\|X\|_{{M}, t}:=\left(\int_{[0,t] \times \Omega} X ^{2} \mathrm{d} \mu_{M}\right)^{1 / 2}
=\left(\mathbb E \int_0^t X _s ^{2} \mathrm{d}[M]_{s} \right)^{1 / 2}=\left(\mathbb E \int_0^t X _s ^{2} \mathrm{d}\langle M\rangle _{s} \right)^{1 / 2},
\end{equation}
and  we write  $\mathcal P(M)$ for the collection of all predictable processes $X$ satisfying $\|X\|_{{M}, t}<\infty$   for every $t\geq 0$.

 Finally, consider on $\mathcal P(M)$  the metric  $d_{\mu_M}(X, Y)=\|X-Y\|_{\mu_ M}, X,Y\in \mathcal P(M)$, where
\begin{equation}
\|X\|_{\mu_M}=\sum_{n\in \N} 2^{-n}(1 \wedge\|X\|_{{M}, n}),\quad X\in \mathcal P(M)
.\end{equation}
Again, $\|\cdot\|_{\mu_M}$ is no norm. In order to let $d_{\mu_ M}$ be   a well-defined metric on $\mathcal P(M)$, we   identify two processes $X$ and $Y$ in $\mathcal P(M)$ as the same processes if they are $\mu_M$\name{-equivalent}{equivalent!$\mu_M$-equivalent}, i.e., 
\begin{equation}
\mu_{M}\big((t,\omega)\in \Rplus\times \Omega:X_t(\omega)\neq Y_t(\omega)\big)=0.
\end{equation}
We are now able to extend the definition of a stochastic integral. We  will write $M_n\stackrel{\mathscr M^2}{\longrightarrow} M$ if $\|M_n-M\|_{\mathscr M^2 }\to 0$   as $n\to\infty$, for $(M_n)_{n\in\N},M\in  \mathscr M^2.$ Similarly, one defines a limit $H^n\stackrel{\mu_M}{\longrightarrow} H$.
\begin{definition}\label{def:integralM2}
	  Suppose $M\in \mathscr M^2$ and let $H\in \mathcal P(M).$ The \name{stochastic integral of $H$ with respect to $M$}{stochastic integral}, again denoted by $\dint0\cdot {H_s}M_s$, is  the square integrable martingale satisfying
	\begin{equation}\label{eq:integralM2}
	\dint0\cdot {H_s^n}M_s\stackrel{\mathscr M^2}{\longrightarrow}\dint0\cdot {H_s}M_s,
\end{equation}
where $(H^n)_{n\in \N}$  is an arbitrary sequence of simple predictable processes such that $H^n\stackrel{\mu_M}{\longrightarrow} H$.
%

\end{definition}
Analogous to the implicitly required facts for Definition \ref{def:integral}, one can show that  the space of simple predictable processes is dense in  $\mathcal P(M)$ under the $d_{\mu_M}$-metric, and that  the space of   $L^2$-martingales $\mathscr{M}^2$ endowed with the $d_{\mathscr M^2}$-metric is complete. We refer to  \citeb{  p. 144}{unpublished:timo} and \citeb{p. 108}{unpublished:timo}, respectively, for their proofs.
In addition, the   subspace of continuous square integrable martingales is closed under the $d_{\mathscr M^2}$-metric, thus complete  again.  This  gives us that the stochastic integral in Definition \ref{def:integralM2} preserves the continuity property as well, i.e.,  the stochastic integral is continuous whenever the integrator $M$ is. Lastly, the   $L^2$-convergence as in \eqref{eq:integralM2} is stronger than ucp-convergence (which follows from Doob's maximal inequality; see Theorem \ref{thm:Doob2}), hence Definition  \ref{def:integral} and Definition \ref{def:integralM2} yield the same limit in case our integrand  $H\in\mathbb L\Rplus\subset \mathcal P (M)$ is \caglad (which is desirable of course).  

We claim that all properties in the previous sections still hold with $\mathbb L\Rplus$ being  replaced by $\mathcal P(M).$ For instance, the linearity of integrands and integrators are again immediately clear. A somewhat less trivial result is that we can extend 
Corollaries \ref{cor:truemart}  and \ref{cor:handy}  to   the class of predictable integrands. Because we know the latter is true for simple predictable processes,  the result below  basically   follows from   taking $n\to\infty$. 
 
 \begin{theorem}\label{thm:square}Suppose $M\in \mathscr M^2$ and let $H\in \mathcal  P(M).$ 
 	Then  we have
 	 \begin{equation}
 		\label{eq:formula3}	\mathbb E\left(\dint0t{H_s}M_s\right)^2=\mathbb E \dint0t{H_s^2}[M]_s=\mathbb E \dint0t{H_s^2}\langle M\rangle_s,
 	\end{equation}
 	for all $t\geq 0$. Also, in case $H, K \in  \mathcal{P}(M)$ are $\mu_{M}$-equivalent, then
 	$\dint0\cdot {H_s}M_s$ and $\dint0\cdot {K_s}M_s$ are indistinguishable.
%
\end{theorem}
\begin{proof}	(Inspired by \citeb{p. 148}{unpublished:timo}). 
	See the discussion above.  In more detail, one applies the  inverse triangle inequalities  
	\begin{equation}
		\big|\|H\|_{\mu_M}-\|K\|_{\mu_M}\big|\leq \|H-K\|_{\mu_M}\quad\text{and}\quad \big|\|M\|_{\mathscr M^2}-\|N\|_{\mathscr M^2}\big|\leq \|M-N\|_{\mathscr M^2};
	\end{equation} uses the fact that equation   \eqref{eq:formula3} holds for simple predictable processes (which  follows from the previously mentioned corollaries); and   takes any    approximating sequence as in Definition \ref{def:integralM2}. The final assertion follows directly from the now proven result \eqref{eq:formula3}.
%
\end{proof}

Recall that there is no ambiguity between the two definitions of a stochastic integral when an integrand $H\in \mathbb L\Rplus\subset \mathcal P(M)$ is \cadlag. Suppose now that $M$ is also a finite variation process. Then $\dint0\cdot {H_s}M_s$ could either be understood as a Lebesgue--Stieltjes   or a stochastic integral. It would favourable if these two notions coincide, as in Proposition \ref{prop:stieltjes}. This is indeed the case.
 \begin{proposition}[Proposition 5.36 of \cite{unpublished:timo}]\label{prop:stieltjes2}
	Suppose $M\in\mathscr M^2$ is a square integrable martingale and of finite variation. Let $H\in \mathcal P(M) $. Then the Lebesgue--Stieltjes integral of $H$ with respect to $M$  is indistinguishable from the stochastic integral $\dint0\cdot {H_s}M_s$.
\end{proposition}
In the continuous setting, one does not have to worry about the latter, simply due to the fact that  continuous martingales which are also of finite variation do not exist (with the zero process as exception); see Remark \ref{remark:2}.

 \subsubsection*{The extension to   the class of progressively measurable integrands}
 \noindent Let us first introduce some other  measurability types for stochastic processes.  We have already encountered the predictable $\sigma$-algebra $\mathcal P$ and the product  $\sigma$-algebra $ \mathcal B(\Rplus)\times \mathcal F$, corresponding to predictable and measurable processes. Likewise, one can define the \name{optional $\sigma$-algebra} $\mathcal O$ as the $\sigma$-algebra generated by all   adapted \cadlag processes. Equivalent notions  as in Proposition \ref{prop:predict} can be found in \citeb{p. 63}{book:chung}. Moreover, a  stochastic process $X=(X_t)_{t\geq 0}$ is said to be \name{progressive}{process!progressive} or \name{progressively measurable}  if the mappings $[0,t]\times \Omega, (s,\omega)\mapsto X(s,\omega)$ are measurable with respect to $\mathcal B([0,t])\times \mathcal F_t$, for all $t\geq 0$. We denote by $\mathcal M$  the smallest $\sigma$-algebra on $[0,\infty)\times \Omega$ making all the progressive processes measurable. A little warning though, note that $\mathcal M$-measurability does not mean progressive  measurability (just as for $\mathbb L\Rplus$ and $\mathcal P$); therefore, calling a stochastic process progressive would actually be better. Finally, denote by $\mathcal V$  the $\sigma$-algebra generated by all adapted  measurable processes. 
 
 One can easily verify that progressive processes are measurable and adapted. 
 Moreover, we have  the following relationships:
 \begin{equation}
 	\mathcal P\subset \mathcal O\subset \mathcal M\subset \mathcal  V\subset  \mathcal B(\Rplus)\times \mathcal F\label{eq:inclusions}.
 \end{equation}
 See \citeb{p. 63}{book:chung} or \citeb{p. 102}{book:protter}. In general, these inclusions are strict. It is important to observe that \eqref{eq:inclusions} is only valid in case regularity everywhere is assumed within the definitions of $\mathbb L\Rplus$ and $\mathbb D\Rplus$, hence in $\mathcal P$ and $\mathcal O$, respectively. 

 \begin{example}\label{ex:regularity}
 	Suppose $X$ is a  (left- or right-)continuous   process and $Y$ is a  version of $X$ with (resp., left- or right-)continuous paths $\mathbb P$-a.s., then $Y$ is indistinguishable from $X$, but     the stochastic process   $Y$ may fail  the property that any  section
 	\begin{equation}
 		Y(\omega):\Rplus\to\R^d,\,t\mapsto Y_t(\omega),
 	\end{equation}
 	is Borel measurable,   hence $Y$ may fail  to be product measurable, i.e., $\mathcal B(\Rplus)\times \mathcal F$-measurable.
 	
 	Indeed, let us define the stochastic process $Y=(Y_t)_{t\geq 0}$ by  \begin{equation}Y_t(\omega)=f(t)\mathds 1_A(\omega),\end{equation}  where $f:\Rplus\to\R^d$ is completely arbitrary  and $\mathbb P(A)=0$. Then $Y$ is indistinguishable from the zero process, because $\mathbb P(\Omega \backslash A)=1$ and $\Omega \backslash A\subset \{\omega\in\Omega:Y_t(\omega)=0\quad \forall t\geq 0\}.$ Nevertheless, for any $\omega\in A$, the section $t\mapsto Y_t(\omega)$ fails to be Borel measurable if the function $f$ is   not Borel measurable (and such functions exist,  e.g., $f(t)=\mathds 1_{V}(t)$ where $V\subset [0,1]$ is a Vitali set.)  	  
 \end{example}
  In particular, any   $\mathcal A$-measurable process with $\mathcal A$ one of the previously mentioned $\sigma$-algebras, satisfies the property that each section $t\mapsto Y_t(\omega)$ is Borel measurable.
 
 Note that progressively measurable processes   define a useful class. For example, the random variable $X_T$ is $\mathcal F_T$-measurable for $T$ a finite stopping time  if $X$ is progressive  \citeb{p. 122}{book:kallenberg}.  We   use this in Lemma \ref{prop:useful-id}. Also,  a Lebesgue--Stieltjes integral is progressive, hence adapted, if the integrand is assumed to be progressively measurable \citeb{p. 23}{book:karatzas}.
 Via a simple approximation, one   deduces that processes in  $\mathbb L\Rplus$  and $\mathbb D\Rplus$---under the assumption we have regular paths everywhere---are progressively measurable \citeb{p. 5}{book:karatzas}.   Many authors prefer to think of  \cadlag and \caglad processes as progressive  processes (and not only up to indistinguishability), which again motivates the regularity everywhere convention. 
  Lastly, if a stochastic process is measurable and adapted, then it has a progressively measurable version; we refer to  \citeb{p. 68}{book:meyer} for its rather demanding  proof. 
 
   As mentioned several times by now, Example \ref{ex:generalise} shows us that  in general we cannot go beyond the class of predictable integrands. Chapter 3 of \cite{book:chung} demonstrates that if
 	\begin{enumerate}
 		\item $M=(M_t)_{t\geq 0}$ is assumed to be a continuous $L^2$-martingale; or
 	\item  the Doléans measure $\mu_M$ is assumed to be absolutely continuous with respect to $\mathrm ds\times\mathbb P$,
 	\end{enumerate}
  then it is possible to extend the definition of a stochastic integral on a larger class of integrands. 
 Hypothesis  2.\ is a mild condition satisfied by numerous processes, according to \citeb{p. 57}{book:chung} and \citeb{p. 135}{book:karatzas}. In particular, a Brownian motion has Doléans measure $\mathrm ds\times \mathbb P$. The typical Lévy processes as introduced in  {\S}\ref{Sec1.2.2} also satisfy this mild condition.
  
   \subsubsection*{1. Continuous $L^2$-martingales}
\noindent Like we already briefly mentioned in Remark \ref{remark:continuous},  one can make   use   of the Doléans measure $\mu_M $ on $\mathcal B(\Rplus )\times \mathcal F$ to provide a suitable extension of the stochastic integral; thus,   preserving all the (semi)martingale properties. One may follow a direct method (as the reader is assumed to be familiar with), i.e.,  showing that the class of simple predictable processes lies dense   in the class of progressively measurable processes (with respect to the metric $d_{\mu_M}$) where  the usual integrability condition---$\|X\|_{M,t}<\infty$ for all $t\geq 0$---is satisfied. This is shown in, e.g., Proposition 2.8 of \citeb{p. 137}{book:karatzas}. One can then simply enhance Definition \ref{def:integralM2} by letting $\mathcal P(M)$ denote the class of progressively measurable processes $X$ with $\|X\|_{M,t}<\infty$ for all $t\geq 0$. 

Alternatively, the method  conform \citeb{p. 68}{book:chung} provides   an appropriate extension of the stochastic integral  
with as  class   suitably integrable processes that are $\mathcal P_{\text{cont.}}$-measurable:   define the augmented $\sigma$-algebra $\mathcal P_{\text{cont.}}:=\mathcal P\vee \mathcal N_{\text{cont.}}$, where $\mathcal N_{\text{cont.}}$ is the collection of all $\mu_M $-null sets. Consequently, due to the general result in \citeb{p. 59}{book:chung2}---altered in lines with the discussion in \citeb{p. 69}{book:chung}---a process $X$ is $\mathcal P _{\text{cont.}}$-measurable if and only if there exists a predictable process $\bar X$ which is $\mu_M$-equivalent with $X$. This enables us to set\index{stochastic integral}
\begin{equation}
	\int_0^\cdot X_s\,\mathrm dM_s:=\int_0^\cdot \bar X_s\,\mathrm dM_s.
\end{equation}
All properties remain valid, because we  do not really obtain genuinely new stochastic integrals.
 Besides, any progressively measurable process is $\mathcal P _{\text{cont.}}$-measurable, see  \citeb{Thm. 3.10}{book:chung}, and therefore both methods   yield the same notion of a stochastic integral  on the class of suitably integrable progressively measurable process (also denoted by $\mathcal P(M))$.

It is worth noting that the latter  method is  very similar to  the approach for 2., so for more information on the previous we refer to the discussion below.

\subsubsection*{2. The mild condition for \cadlag $L^2$-martingales}

 \noindent Throughout the remainder of this subsection, we suppose   that the Doléans  measure $\mu_M$ on $\mathcal P$ of $M=(M_t)_{t\geq 0}\in\mathscr M^2$   is absolutely continuous with respect to $\mathrm ds\times \mathbb P $, that is, if $(\mathrm ds\times \mathbb P)(A)=0$ for  some $A \in \mathcal{P}$, then $\mu_{M}(A)=0$. The absolute continuity is   abbreviated   by $\mu_M\ll \mathrm ds\times \mathbb P.$   
 \begin{definition}
 	We let \index{$\mathcal P^*$}$\mathcal{P}^{*}:=\mathcal P\vee \mathcal N^*$ denote the augmented $\sigma$-algebra, where $\mathcal N^*$ is the collection of $\mathrm ds\times \mathbb P$-null sets. Equivalently, we have
 \begin{equation}
 \mathcal{P}^{*}=\left\{A \in \mathcal{B}(\Rplus) \times \mathcal{F}: \text {there exists a } P \in \mathcal{P}\right.  \text { such that } (\mathrm ds \times \mathbb P)(A \,\triangle \,P)=0\}.\label{eq:equiv}
 \end{equation}
 \end{definition}
Equation \eqref{eq:equiv} gives us that any function $X$ is $\mathcal P^*$-measurable if and only if there exists a $\mathcal P$-measurable function $Z$ with $(\mathrm ds\times \mathbb P)(X\neq Z)=0.$ This is stated in  Lemma 3.5(ii) of \citeb{ p. 69}{book:chung}. In particular, the equivalent notions of $\mathcal P^*$ and the latter  result holds in a much more general setting, namely for any measurable space and any $\sigma$-algebra. This  follows from  \citeb{p. 59}{book:chung2}; see also \citeb{p. 35}{unpublished:timo}.  
 
 Since $\mu_{M}\ll\mathrm ds\times \mathbb P$ holds on   $\mathcal{P}$, we obtain by the Radon--Nikodym theorem that there exists a non-negative predictably measurable function $f_{M}: \Rplus \times \Omega\to\R$ such that
\begin{equation}
\mu_{M}(A)=\int_{A} f_{M} (\mathrm ds \times \mathbb P), \quad  A \in \mathcal{P}.
\end{equation}
This formula is key to our extension. Clearly,  whenever a function or a set is $\mathcal P$-measurable,   it is also $\mathcal P^*$- and $\mathcal B(\Rplus)\times \mathcal F$-measurable. This enables us to define the following measure.
 \begin{definition}
 We denote by $\mu_{M}^{*}$\index{$\mu_{M}^{*}$} the measure on the $\sigma$-algebra $\mathcal{P}^{*}$  defined by 
\begin{equation}
	\mu_{M}^*(A)=\int_{A} f_{M} (\mathrm ds \times \mathbb P), \quad  A \in \mathcal{P}^*.
\end{equation}
 The measure $\mu_{M}^{*}$ is an extension of $\mu_{M}$ from $\mathcal{P}$ to the larger $\sigma$-algebra $\mathcal{P}^{*}$, i.e., $\mu_M^*(A)=\mu_M(A)$ for all $A\in\mathcal P$.
 \end{definition}

For the continuous setting, recall  that any progressive process is $\mathcal P_{\text{cont.}}$-measurable. It turns out that  any  progressively measurable process is also $\mathcal P^*$-measurable, and even more is true.


%
 
%
%
%
%
%

%
%
%


\begin{theorem}\label{thm:equiv}
	  Suppose $X$ is an adapted  measurable stochastic process. Then there exists a $\mathcal{P}$-measurable process $\bar{X}$ such that
 \begin{equation} (\mathrm ds \times \mathbb P)\left((t, \omega) \in \Rplus \times \Omega: X(t, \omega) \neq \bar{X}(t, \omega)\right)=0.\end{equation}
Under the assumption $\mu_{M} \ll  \mathrm ds\times \mathbb P$, we also have
 \begin{equation} \mu_M^*\left((t, \omega) \in \Rplus \times \Omega: X(t, \omega) \neq \bar{X}(t, \omega)\right)=0.\label{equiv!}\end{equation}

\end{theorem}
\begin{proof}
The second assertion is clearly evident. For a proof of the first result, we refer to \citeb{p. 66}{book:chung} or \citeb{p. 190}{unpublished:timo}. The latter reference provides a relatively more direct approach.
\end{proof}

Stochastic processes $X$ and $\bar{X}$ are said to be $\mu_{M}^{*}$\name{-equivalent}{equivalent!$\mu_M^*$-equivalent} if they satisfy \eqref{equiv!}. Theorem \ref{thm:equiv} tells us that for any adapted  measurable process $X$ there is a $\mu_M^*$-equivalent predictably measurable process $\bar X$. Hence, in line with Theorem \ref{thm:square}, it is natural to define\index{stochastic integral}
\begin{equation}\label{eq:naam}
	\int_0^\cdot X_s\,\mathrm dM_s:=	\int_0^\cdot \bar X_s\,\mathrm dM_s,
\end{equation}
 for all  adapted measurable processes $X$ satisfying
\begin{equation}
\int_{[0, t] \times \Omega}|X|^{2} \,\mathrm  d \mu_{M}^{*}=\int_{[0, t] \times \Omega}|X|^{2}f_M  (\mathrm ds\times \mathbb P)<\infty, \quad \text { for all } t\geq 0, \label{eq:let}
\end{equation}
since $\int_{[0, t] \times \Omega}|X|^{2} \,\mathrm  d \mu_{M}^{*}=\int_{[0, t] \times \Omega}|\bar X|^{2} \,\mathrm  d \mu_{M}^{*}=\|\bar X\|_{M,t}^2.$  

Recall that  any adapted measurable process has a progressively measurable version. Hence, in line with the previous, we write  $\mathcal P^*(M)$ for  the collection of progressively measurable processes such that integrability condition \eqref{eq:let} holds. Observe that equation \eqref{eq:let} is equivalent with presuming
\begin{equation}
\mathbb E\int_{0}^t|X|^{2}\,\mathrm  ds  <\infty, \quad \text { for all } t\geq 0,
\end{equation}
in case, e.g., the Radon--Nikodym derivative $f_M$ is constant. It is worth noting that this holds for Lévy processes in {\S}\ref{Sec1.2.2}.
 
The stochastic integral  extended with $\mathcal P^*(M)$ as integrands enjoys all the desired properties we derived before (except for one, see Remark \ref{remark:counter}). This is due to the fact that the collection of processes which appear as stochastic integrals has not been expanded; this directly follows from the construction by \eqref{eq:naam}.  

 We end this section with two important remarks.

 \begin{remark} When $M=(M_t)_{t\geq 0}$ is a Brownian motion,  conditions 1. and 2. are both satisfied. In particular, we have $\mathcal{P}_{\text{cont.}}=\mathcal P^*$ and 
 	\begin{equation}
 		\mu_M=\mathrm ds\times \mathbb P=\mu_M^*.
 	\end{equation} Therefore, the extensions in 1. and 2. coincide if $M$ is a Brownian motion. In general, we have
 \begin{equation}
 \mathcal M\subset  \mathcal P_{\text{cont.}}\subset \mathcal B(\Rplus)\times \mathcal F\quad\text{and}\quad \mathcal M\subset \mathcal V\subset \mathcal P^*\subset \mathcal B(\Rplus)\times \mathcal F,
\end{equation}
but $\mathcal V \subset\mathcal P_{\text{cont.}}$ does not hold for all continuous $L^2$-martingales \citeb{p. 71}{book:chung}. The reader may also want to compare Proposition 2.6 in \citeb{p. 134}{book:karatzas}  and Proposition 2.8 in \citeb{p. 137}{book:karatzas}.

There is  no guarantee that if $X$ and $\bar X$ are $\mu_M$-equivalent, they are also $\mu_M^*$-equivalent, or vice  versa.  Therefore, for general continuous $L^2$-martingales that satisfy the absolute continuity condition  as well, there might be an ambiguity on the class of suitably integrable progressive measurably processes.   \mark{not sure}
\end{remark}
When  $M$ is continuous, the stochastic integral is to be understood in the usual sense, i.e., as discussed in 1.. The approach in 2.\ should be interpreted as a last resort method. Nevertheless, for all Lévy processes in {\S}\ref{Sec1.2.2}, both methods do coincide and there is no ambiguity.

\begin{remark}\label{remark:counter}
	In contrast to  Proposition \ref{prop:stieltjes2}, if the integrand is progressively measurable but not predictable, then the stochastic integral may not coincide with  the Lebesgue--Stieltjes integral  for integrators with paths of finite variation. We refer to Example \ref{ex:generalise} for an illustration of this claim. Even though this may seem unfavourable, the extension is still successful because integrals are martingales.
%
\end{remark}

\subsection{Other relevant notions and machinery}\label{Sec1.1.5}
\noindent In this final part of this section, we     summarise the necessary facts from the theory of stochastic integration that we have not discussed yet. We try to keep things brief.

\begin{proposition}\label{prop:convex}
	Suppose $X$ is a martingale \textnormal{(}resp. submartingale\textnormal{)}, and let $\varphi:\R\to\R$ be  a convex \textnormal{(}resp. non-decreasing, convex\textnormal{)} function. Define the   process  $\varphi(X)=(\varphi(X_t))_{t\geq 0}$ and assume that $\varphi(X)_t$ is integrable for all $t\geq 0$. Then $\varphi(X)$ is a submartingale.
\end{proposition}

\begin{proof}
	This follows immediately from Jensen's inequality for conditional expectations.
\end{proof}

Observe the processes $|X|$ and $X^+=X\vee 0$ correspond to the convex functions $\varphi(x)=|x|$ and $\varphi(x)=\max(0,x)$ respectively.

\begin{theorem}\namethrm{Doob's supremal inequalities}\label{thm:Doob}
Let $X$ be a submartingale. Then for every $ \lambda> 0$ and $0 \leq s \leq t$ one has 
	\begin{equation}
\label{eq:supremal}	\textstyle \lambda \mathbb{P}\left(\sup _{s \leq u \leq t} X_{u} \geq \lambda\right) \leq \mathbb{E} X_{t}^{+}
\quad\text{and}\quad \lambda \mathbb{P}\left(\inf _{s \leq u \leq t} X_{u} \leq-\lambda\right) \leq \mathbb{E} X_{t}^{+}-\mathbb{E} X_{s}.
	\end{equation}
	In particular, suppose that $X$ is a non-negative submartingale. Then for any  continuous, non-decreasing, convex function $\varphi:\Rplus \to[0,\infty)$, we have
\begin{equation}
	  \mathbb{P}\left(\sup  _{s \leq u \leq t} X_{u} \geq \label{eq:supremal2}\lambda\right)     \leq\frac{\mathbb{E} [\varphi(X_{t})]}{\varphi(\lambda) },\quad \text{if }\varphi(\lambda)>0. 
\end{equation}
\end{theorem}
\begin{proof}
	The well-known results in equation \eqref{eq:supremal} can be found in any elementary textbook about martingales, see, e.g., \citeb{p. 13}{book:karatzas}. The result in equation \eqref{eq:supremal2} follows  from the     fact  
	\begin{equation}
\textstyle		\varphi(\sup_{s \leq u \leq t}X_u)=\sup_{s \leq u \leq t}\varphi(X_u)
	\end{equation}
holds for those function $\varphi$ satisfying the conditions posed,
and by subsequently applying Proposition \ref{prop:convex} together with the first part of equation   \eqref{eq:supremal}.
\end{proof}
\begin{corollary} \label{cor:Doob}
	For $X$ a martingale, we have $|X|$ is a non-negative submartingale.\footnote{In here, the absolute value of $X$, denoted by $|X|$, should not be confused with the total variation process.} Hence, in the setting of Theorem \ref{thm:Doob}, we get
	\begin{equation}
		\mathbb{P}\left(\sup  _{s \leq u \leq t} X_{u} \geq \lambda\right)  \leq \mathbb{P}\left(\sup  _{s \leq u \leq t} |X_{u}| \geq \lambda\right)  \leq\frac{\mathbb{E} [\varphi(|X_{t}|)]}{\varphi(\lambda) },\quad \text{if }\varphi(\lambda)>0. 
	\end{equation}
\end{corollary}

Observe that the first inequality holds clearly. Moreover, note that this result is some kind of a generalisation of  Markov's inequality.

\begin{theorem}\namethrm{Doob's maximal inequality}\label{thm:Doob2}
	If $X$ is a non-negative submartingale, then one has  for all $p>1$ the estimate
	\begin{equation}
	\left(\mathbb E\left[\sup _{s \leq u \leq t} |X_{u}|^p\right]\right)^{1/p}=	\left(\mathbb E\left[\sup _{s \leq u \leq t} |X_{u}|\right]^p\right)^{1/p} \leq \frac{p}{p-1} (\mathbb E|X_{t}|^p)^{1/p}.
	\end{equation}
\end{theorem}
\begin{proof}
	Again, this can be found in any textbook about martingales; see,	e.g., \citeb{p. 13}{book:karatzas}.
\end{proof}
Last but definitely not least, we state the famous Burkholder--Davis--Gundy inequalities.
\begin{theorem}\namethrm{Burkholder--Davis--Gundy inequalities}\label{thm:BDG}
For any $1 \leq p<\infty$ there exist  
constants $c_{p}, C_{p}>0$ such that, for all local martingales $X$ with $X_{0}=0$,
the following inequalities hold:
\begin{equation}
c_{p} \mathbb{E}\left[[X]_{t}^{p / 2}\right] \leq  \mathbb E\left[\sup _{0 \leq s \leq t} |X_{s}|^p\right] \leq C_{p} \mathbb{E}\left[[X]_{t}^{p / 2}\right],\quad t\geq 0.
\end{equation}
For continuous local martingales, this statement holds true for all $0<p<\infty$. 
Furthermore, the above is also valid with the instant $t$ replaced by a finite stopping time $T$.
\end{theorem}
\begin{proof}
	For    the result on continuous local martingales, we refer to \citeb{p. 166}{book:karatzas}. For the more general setting; see \citeb{Thm. IV.48}{book:protter}.  
\end{proof}
Observe that $p=2$ is a special case, with  $c_2=1$ and $C_2=4$, which follows quite directly from Proposition \ref{prop:intermed} and Doob's maximal inequality in Theorem \ref{thm:Doob2}. The power of the Burkholder--Davis--Gundy inequalities lies within the other values of  $p\neq 2.$   
 
Now, let us conveniently introduce some   notation. Observe   stochastic integrals with respect to semimartingales, recall {\S}\ref{Sec1.1.2}, should be understood as integrals over the compact domains $[0,t]$ for all $t\geq0$. These   integrals take the initial value of a semimartingale $X$ into account  by construction, since  $X_0\neq 0$ in general. Hence, for any proper integrand $H$, we write
\begin{equation}
	\int_{[0,t]}H_s\mathrm{d}X_s=\int_0^tH_s\mathrm{d}X_s.
\end{equation}
Subsequently, we introduce the familiar integral notation 
\begin{equation}
\int_{(s, t]} H_u \mathrm d X_u:=\int_0^tH_u\mathrm{d}X_u -\int_0^sH_u\mathrm{d}X_u\quad \text { for all } 0 \leq s \leq t.
\end{equation}
Observe that the expression above yields  an integral over the domains $(0,t]$ for all $t\geq 0$.
In addition, we can make sense of integrals over $[0,T]$ and $(S,T]$ over finite stopping times $T,S$ such that $S\leq T$. We do this simply by setting
\begin{equation}
	\dint{[0,T]}{}{H_s}X_s:=\left(\dint0\cdot{H_s}X_s\right)_T\,\text{ and }\, \int_{(S, T]} H_u \mathrm d X_u:=\int_{[0,T]}H_u\mathrm{d}X_u -\int_{[0,S]}H_u\mathrm{d}X_u.
\end{equation}
Observe, the above coincides with the ordinary    when the stopping times $T$ and $S$ are assumed to be fixed times $t$ and $s$.

Finally, although we will not be needing it in further chapters, we introduce the following notations for the left limit of the previous integrals. Define pathwise 
\begin{equation}
\int_{[0, t)} H_u \mathrm d X_u:=\lim _{s \nearrow t} \int_{[0,s]} H_u \mathrm  d X_u \quad\text{and}\quad \int_{(0, t)} H_u \mathrm d X_u:=\lim _{s \nearrow t} \int_{(0, s]} H_u \mathrm d X_u.
\end{equation}
This clearly yields the identity $
\int_{[0, t]} H_s d Y_s=\int_{[0, t)} H_s d X_s+H_t \Delta X_t$, and a similar result holds for the integral over $(0,t]$. In particular, we obtain
\begin{equation}\label{eq:contint}
	\dint0t{H_u}X_u=	\int_{[0,t]}H_u\mathrm{d}X_u=		\int_{(0,t]}H_u\mathrm{d}X_u=\int_{[0,t)}H_u\mathrm{d}X_u=	\int_{(0,t)}H_u\mathrm{d}X_u,
\end{equation}
for all continuous semimartingales $X$, where $X_0=0.$
On the other hand, we   have to be   careful with the endpoints of a stochastic integral whenever we consider \cadlag semimartingales. The result in \eqref{eq:contint} can be extended to more general intervals with $0$ replaced by any $s\geq 0$.

Now let us recall the definition of a quadratic covariation process; see equation \eqref{eq:intbyparts}. Then for any semimartingales $X$ and $Y$, we have the integration by parts formula
\begin{equation}
\label{eq:intbyparts2}	X_tY_t=\dint0t {X_{s-}}Y_s+\dint0t {Y_{s-}}X_s+[X,Y]_t,
\end{equation}
for all $t\geq 0$. Note, we use the convention $X_{0-}=(X_-)_0=0$ here. Implementing the introduced notations, we can also write equation \eqref{eq:intbyparts2} as
\begin{equation}
		X_tY_t=\dint{(0,t]}{} {X_{s-}}Y_s+\dint{(0,t]}{} {Y_{s-}}X_s+[X,Y]_t,
\end{equation}
which might be a more pleasant formula to look at. Observe   the integration by parts formula is a special case of  It\^o's formula stated below.

\begin{theorem}\namethrm{Itô's formula} Suppose\label{thm:ito} $X=(X^1,...,X^d)$ is a $d$-dimensional vector of semimartingales, and let $f:\R^d\to\R$ have continuous second order partial derivatives. Then the process $f(X)=(f(X_t))_{t\geq 0}$ is a semimartingale, and satisfies the equivalent formulae
	\begin{align}\label{eq:ito1}
		f(X_t)= &\, f(X_0)+\sum_{i=1}^d\dint{(0,t]}{}{\frac{\partial f}{\partial x_i}(X_{s-})}X_s^i+\frac{1}{2}\sum_{1\leq i,j\leq d} \dint{(0,t]}{}{\frac{\partial^2 f}{\partial x_i\partial x_j}(X_{s-})}[X^i,X^j]_s\\
		&\,+\sum_{0<s\leq t}\left\{f(X_s)-f(X_{s-})-\sum_{i=1}^n\frac{\partial f}{\partial x_i}(X_{s-})\Delta X_s^i-\frac{1}{2}\sum_{1\leq i,j\leq d} \frac{\partial^2 f}{\partial x_i\partial x_j}(X_{s-})\Delta X_s^i\Delta X_s^j\right\}\nonumber
	\end{align}
and
\begin{equation}\label{eq:ito2}
	\begin{aligned}
		f(X_t)= &\, f(X_0)+\sum_{i=1}^d\dint{(0,t]}{}{\frac{\partial f}{\partial x_i}(X_{s-})}X_s^i+\frac{1}{2}\sum_{1\leq i,j\leq d} \dint{(0,t]}{}{\frac{\partial^2 f}{\partial x_i\partial x_j}(X_{s-})}[X^i,X^j]^c_s\\
		&\,+\sum_{0<s\leq t}\left\{f(X_s)-f(X_{s-})-\sum_{i=1}^n\frac{\partial f}{\partial x_i}(X_{s-})\Delta X_s^i\right\},
	\end{aligned}
\end{equation}
where the infinite sums in \eqref{eq:ito1} and \eqref{eq:ito2} converge.
\end{theorem}
\begin{proof}
	The statement together with  equation \eqref{eq:ito2} is stated as Theorem 33 in \citeb{p. 81}{book:protter}. Substituting the continuous part $[X^i,X^j]^c$ by the quadratic covariation $[X^i,X^j]$ and its jump part yields equation \eqref{eq:ito1}; see also the proof of Theorem 32 of \citeb{p. 78}{book:protter}.
\end{proof}

%
%
%
%

%
%
%
%
%

Observe,   these  infinite sums  are finite variation processes. Additionally,  if $X$ is continuous,  the formula simplifies considerably, because then the sum over $0<s\leq t$ vanishes. 

Note that Itô's formula is often written in differential form (especially within the continuous case). 
In a  general  setting, any   measurable process $H=(H_{t})_{t\geq 0}$, which is only   non-zero for a countable set of
times,  can   alternatively be considered as a differential when  $\sum_{s \leq t}\left|H_{s}\right|<\infty$.  This
is achieved by simply replacing integration with summation, i.e.,
\begin{equation}\label{eq:shorthand}
	d Y_t=H_t \quad \Longleftrightarrow \quad Y_{t}=Y_{0}+\sum_{s \leq t} H_{s} .
\end{equation}
This   enables us to express the general Itô's formula  in differential notation, of which we will give an example; see  \eqref{eq:short1}--\eqref{eq:short3}. Moreover,   we have
$d[X, Y]^{c}=d[X, Y]-\Delta X \Delta Y$ in differential notation,   where $[X,Y]^c$ is the continuous part of the covariation.

Consider, for example, the vector $(Y_t,t),t\geq 0,$ where $Y$ is a semimartingale. Define the semimartingale $X$  by $X_t=f(Y_t,t),t\geq 0$, where
\begin{align}\label{eq:short1}
	f(Y_t,t)=&\,f(Y_0,0)+\dint{(0,t]}{}{\frac{\partial f}{\partial s}(Y_{s-},s)}s+\dint{(0,t]}{}{\frac{\partial f}{\partial\nonumber x}(Y_{s-},s)}Y_s+\frac{1}{2}\dint{(0,t]}{}{\frac{\partial^2 f}{\partial x^2}(Y_{s-},s)}[Y]^c_s\\&\,+\sum_{0<s\leq t}\left\{f(Y_{s},s)-f(Y_{s-},s)-\frac{\partial f}{\partial x}(Y_{s-},s)\Delta Y_s\right\}.
\end{align}
Since the deterministic time process is continuous, writing either $s-$ or $s$ in the second argument of $f$ in the above makes no difference. In differential notation, we write
\begin{equation}\begin{aligned}\label{eq:short}
   \mathrm dX_t&=\frac{\partial f}{\partial t}(Y_{t-},t)\,\mathrm dt+\frac{\partial f}{\partial x}(Y_{t-},t)\,\mathrm dY_t+\frac12\frac{\partial^2 f}{\partial x^2}(Y_{t-},t)\,\mathrm d[Y]_t^c \\&\quad\quad\quad\quad\quad\quad\quad\quad\quad\quad\quad\quad +f(Y_{t},t)-f(Y_{t-},t)-\frac{\partial f}{\partial x}(Y_{t-},t)\Delta Y_t.
   \end{aligned}
\end{equation}
This shorthand notation  is, e.g., conform \citeb{p. 280}{book:tankov}. Under the presumption        $Y$ has at most a finite amount of jumps on compact time intervals, then $\sum_{0<s\leq t}|\Delta Y_s|<\infty$ holds, for all $t\geq 0$, and    $\sum_{0<s\leq t}\Delta Y_s$ is a finite variation process. Hence, we   decompose   $Y=Y^c+Y_0+\sum_{0<s\leq t}\Delta Y_s$  and differential equation \eqref{eq:short} becomes
\begin{equation}\begin{aligned} \label{eq:short3}
		\mathrm dX_t=\frac{\partial f}{\partial t}(Y_{t-},t)\,\mathrm dt+\frac{\partial f}{\partial x}(Y_{t-},t)\,\mathrm dY_t^c+\frac12\frac{\partial^2 f}{\partial x^2}(Y_{t-},t)\,\mathrm d[Y]_t^c  +f(Y_{t},t)-f(Y_{t-},t),
	\end{aligned}
\end{equation}
where $Y^c$ is the continuous part\footnote{Usually, ones  denotes  by the process $X^c$ the continuous \textit{local martingale} part of $X$, see \citeb{p. 226}{book:protter} and \citeb{p. 45}{book:jacod}, which is determined uniquely and due to Theorem \ref{thm:cont} (in the finite jumps on compacts case).}
 of $Y$, satisfying  the infinitesimal relationship $\mathrm d Y^c=\mathrm d Y-\Delta Y$; see Remark 8.3 in \citeb{p. 279}{book:tankov}.
 Observe, the infinite    sum $\sum_{0<s\leq t}\left\{f(Y_{s},s)-f(Y_{s-},s)\right\}$ does  not converge in general (see also Exercise 12 of \citeb{p. 95}{book:protter}).

A useful application of Itô's formula yields the following result. 
\begin{theorem}\namethrm{Doléans--Dade exponential}\label{thm:dade}
	Let $X$ be a semimartingale such that $X_0=0$. Then  there exists a unique semimartingale $Z$ that satisfies the equation $Z_t=1+\dint0t{Z_{s-}}X_s$, and is given by the equivalent explicit formulae
	\begin{equation}
		Z_t=\exp\left(X_t-\frac12[X]_t\right)\prod_{0<s\leq t}(1+\Delta X_s)\exp\left(-\Delta X_s+\frac{1}{2}(\Delta X_s)^2\right)
	\end{equation}
and
\begin{equation}
	Z_t=\exp\left(X_t-\frac12[X]_t^c\right)\prod_{0<s\leq t}(1+\Delta X_s)\exp\left(-\Delta X_s \right),
\end{equation}
where the infinite product converges.
\end{theorem}
\begin{proof}
	This result can   be found in   Theorem 37 of \citeb{p. 84}{book:protter} or Theorem 26.8 of \citeb{p. 522}{book:kallenberg}. Proving the uniqueness of $Z$ can be achieved as a trivial consequence of the general theory on stochastic differential equations, conform \cite{book:protter}, for which we refer to {\S}\ref{Sec1.3}, however a more direct approach can also be followed, see the proof of \cite{book:kallenberg}.
\end{proof}

%
%
%
%
%
%


To the best of our knowledge, the result as presented below is not   highlighted in any piece of literature. Intuitively, in equation \eqref{eq:shift}    the domain of integration shifts. The stochastic process $\bar X^S$ arises naturally when looking at Lévy processes and their strong Markov property; see Proposition \ref{prop:strong-M}.


\begin{proposition}\label{prop:useful-id}
Let $X$ be a semimartingale with $X_0=0$, and   $S$   a finite stopping time. Define the   process $\bar X^S$ by $\bar  X^S_t=X_{S+t}-X_S$ for all $t\geq 0$. Then $\bar X^S$ is a semimartingale adapted to $\mathbb G=(\mathcal G_t)_{t\geq0}$ with $\mathcal G_t=\mathcal F_{S+t}$, and  $(\Omega,\mathcal F,\mathbb G,\mathbb P)$ satisfies the usual conditions. 

 Moreover, for every $\mathbb F$-adapted measurable process $H=(H_t)_{t\geq 0}$ for which the stochastic integral $\dint0\cdot {H_s}X_s$ is well-defined, we obtain 
	\begin{equation}
	\label{eq:shift}	\int_{(S,S+t]}H_u\mathrm dX_u=\int_0^tH_{S+u}\mathrm d\bar  X_u^S,
	\end{equation}
for all  $t\geq 0$,
where $(H_{S+t})_{t\geq 0}$ is a $\mathbb G$-adapted measurable process.
\end{proposition}
\begin{proof} The first part of the assertion is quite evident and its proof is as follows. Let us  suppose the decomposition $X=M+A$, then we have  $\bar X^S=\bar M^S+\bar A^S$. Clearly the components are $\mathbb G$-adapted,  $\bar A^S$ is the  finite variation part, and $\bar M^S$ is the local martingale part. Proving the latter, assume without loss of generality $M$ is an $\mathbb F$-martingale. Then $\mathbb P$-a.s. we have
	\begin{equation}
		\mathbb E[\bar M^S_t|\mathcal G_u]=\mathbb E[M_{S+t}|\mathcal F_{S+u}]-\mathbb E[M_S|\mathcal F_{S+u}]=M_{S+u}-M_S=\bar M^S_u,
	\end{equation}
for all $0\leq u\leq t.$ Observe,   we applied the martingale property---in fact we made use of \name{Doob's Optional Sampling Theorem}  \citeb{p. 9}{book:protter}, and used the observation that   the random variable $M_S$ is $\mathcal F_{S}$-measurable \citeb{p. 122}{book:kallenberg} with $\mathcal F_S\subset \mathcal F_{S+u}$. This   proves the first part.
	
	For the second part, observe $X_0=\bar X_0^S=0$ holds, which in particular yields   that integrating over domains $[0,t]$ and $(0,t]$ gives no longer different results.
	Notice that it suffices to prove the equality in   \eqref{eq:shift} for simple predictable processes $H$. The final assertion consequently  follows, as usual, from a limiting argument.
		
		Let us consider a simple predictable process
	\begin{equation}
		H_t=H_0\mathds{1}_{\{0\}}(t)+ \sum_{k=1}^nH_k\mathds{1}_{(T_{k},T_{k+1}]}(t),\quad t\geq 0.
	\end{equation}
 Taking $(0,0]=\emptyset$ as convention, one easily verifies the following expression:
\begin{equation}
	H_{S+t}=H_S\mathds{1}_{\{0\}}(t)+\sum_{k=1}^nH_k\mathds{1}_{(0\,\vee \,(T_{k}-S),0\,\vee \,(T_{k+1}-S)]}(t),\quad\text{for all }t\geq 0.
\end{equation}
Importantly, observe that $(H_{S+t})_{t\geq 0}$ can be interpreted as a simple predictable process as well, yet with respect to the filtration $\mathbb G$. Indeed, the random variables $\tilde T_k=0\,\vee\, (T_k-S)$ are stopping times with respect to $\mathbb G$, that is, $	\{\tilde T_k\leq u\}\in\mathcal G_u=\mathcal F_{S+u}$ for all $u\geq 0$. 

In more detail, for any $u\geq 0$ fixed, we exploit the fact   $S+u$ is again a stopping time, hence
\begin{equation}
\label{eq:stopping}	\{\tilde T_k\leq u\}\cap \{S+u\leq t\}=\{T_k\leq S+u\}\cap \{S+u\leq t\}\in \mathcal F_{t}
\end{equation}
holds for   $t\geq u.$ This is due to basic properties of stopping times and their associated $\sigma$-algebras; see for instance \citeb{Lem. 7.1}{book:kallenberg}.
 The result in equation \eqref{eq:stopping} is also valid for $0\leq t<u$, but this is somewhat trivial because we have  $	\{\tilde T_k\leq u\}\cap \{S+u\leq t\}=\emptyset\in \mathcal F_t.$

In the  final computation below,  we denote $\bar X$ instead of $\bar X^S$ to avoid notational inconveniences. A proof by exhaustion yields
\begin{align*}
 	\int_{(S,S+t]}H_u\mathrm dX_u&=\int_{[0,S+t]}H_u\mathrm{d}X_u -\int_{[0,S]}H_u\mathrm{d}X_u\\
 	&=\sum_{k=1}^nH_k(X_{S+t}^{T_{k+1}}-X_{S+t}^{T_{k}})-\sum_{k=1}^nH_k(X_{S}^{T_{k+1}}-X_{S}^{T_{k}})\\
 	&=\sum_{k=1}^nH_k\left(\big[X_{S+t}^{T_{k+1}}-X_{S}^{T_{k+1}}\big]-\big[X_{S+t}^{T_{k}}-X_{S}^{T_{k}}\big]\right)\\
 	&=\sum_{k=1}^nH_k(\bar X_{t}^{\tilde T_{k+1}}-\bar X_{t}^{\tilde T_{k}})\\
 	&= 	\int_0^tH_{S+u}\mathrm d\bar  X_u,
\end{align*}
which completes the proof.
\end{proof}

In retrospect, observe    in the series of equalities  at  the end of the proof of Proposition \ref{prop:useful-id}, the penultimate equality is far from trivial. (The others follow by definition.) We have
\begin{equation}
\bar X_t^{\tilde T_k}=	X_{S+t}^{T_{k}}-X_{S}^{T_{k}},\quad t\geq 0,
\end{equation}
for any integer $k$, which can be proven by   checking all possibles cases.  For instance, consider the non-trivial scenario:  $\omega\in\Omega$ such that  $S(\omega)+t\geq T_k(\omega)> S(\omega)$. Then $t\geq \tilde T_k(\omega)> 0$ and
\begin{align*}
	\bar X_t^{\tilde T_k}&=\bar X_{t\wedge \tilde  T_k}=X_{S+(t\wedge \tilde T_k)}-X_S=X_{S+\tilde T_k}-X_S=X_{S+(T_k-S)}-X_S\\&=X_{T_k}-X_S=X_{(S+t)\wedge T_k}-X_{S\wedge T_k}=	X_{S+t}^{T_{k}}-X_{S}^{T_{k}},
\end{align*}
where we omit    $\omega$ to simplify notation. The remaining cases are left   as an exercise to the reader.

\section{A brief tutorial on Lévy processes}\label{Sec1.2}
\noindent As before,   assume $(\Omega,\mathcal F,\mathbb F, \mathbb P)$ is some filtered probability space satisfying the usual conditions. 
In order to make sense of certain classes of Lévy processes, we outline the most fundamental results accompanied with several illustrative examples. 
An accessible overview on the basic results of Lévy processes can be found in \cite{book:applebaum,book:tankov,book:kallenberg,book:kyp,book:protter,book:sato} and the references therein.

To begin with, we state a (quite general) definition of a Brownian motion before we head towards the concept of Lévy processes. The subsequent definition is conform \citeb{p. 72}{book:karatzas}.
 
\begin{definition}
	\label{def:Brownian}  Let $d$ be a positive integer and $\mu$ a probability measure on $ (\mathbb{R}^{d}, \mathcal {B} (\mathbb{R}^{d} ) ) .$ A  $\mathbb R^d$-valued process $B=(B_t)_{t\geq 0}$ on $(\Omega,\mathcal F,\mathbb F,\mathbb P)$ is said to be a \name{$d$-dimensional Brownian motion with initial distribution $\mu$}{Brownian motion!with initial distribution $\mu$}, if it is $\mathbb F$-adapted and satisfies the properties
	\begin{enumerate}[\normalfont(i)]
		\item  $\mathbb P\left[B_{0} \in \Gamma\right]=\mu(\Gamma) $ for all $ \Gamma \in \mathcal{B}(\mathbb{R}^{d})$;
		\item $B$ has normally distributed increments, with mean zero and covariance matrix $(t-s) I_d$, i.e., for all $0\leq s\leq t$, we have $B_t-B_s\sim \mathcal N(0,(t-s)I_d)$, with $I_d$ the $d\times d$ identity matrix;
		\item $B$ has \name{increments independent of the past}, i.e., for all $0 \leqslant s \leqslant t,$ we have $B_{t}-B_{s}$ is independent of $\mathcal{F}_{s}$;
		\item $B$ is a continuous process.
	\end{enumerate}
\end{definition}
If $\mu(\{x\})=1$ for some  point  $x\in\R^d$, we say that $B$ is a \name{$d$-dimensional Brownian motion starting at $x$}{Brownian motion!starting at $x$}.	Moreover, throughout these notes, we let $W=(W_t)_{t\geq 0}$ denote a one-dimensional Brownian motion starting at 0, also known as a \name{standard Brownian motion}{Brownian motion!standard Brownian motion}.\index{Brownian motion!$d$-dimensional}

The reader is assumed to be familiar with the answer to the question: why does a standard Brownian motion exist in the first place? We will refer to \citeb{Ch. 2}{book:karatzas},  which consists of multiple approaches of constructing a standard Brownian motion. When this is known, the existence of a $d$-dimensional Brownian motion with initial distribution $\mu$ is affirmative, namely let 
  $X(\omega_{0})=\omega_{0}$ be the identity random variable on the probability space $(\mathbb{R}^{d}, \mathcal{B}(\mathbb{R}^{d}), \mu),$ and for each $i=1, \ldots, d,$ suppose $W^{(i)}$ is a
standard, one-dimensional Brownian motion on some filtered probability space $(\Omega^{(i)}, \mathcal{F}^{(i)}, \mathbb F^{(i)}, \mathbb P^{(i)}) .$ On the product space
\begin{equation}\label{eq:prob}
\left(\mathbb{R}^{d} \times \Omega^{(1)} \times \cdots \times \Omega^{(d)}, \mathcal {B} (\mathbb{R}^{d} ) \otimes \mathcal{F}^{(1)} \otimes \cdots \otimes \mathcal {F}^{(d)}, \mu \times \mathbb P^{(1)} \times \cdots \times \mathbb P^{(d)}\right),
\end{equation}
which we will now abbreviate by $(\Omega',\mathcal F',\mathbb P')$, define the process $B=(B_t)_{t\geq0}$ as
\begin{equation}
B_{t}(\omega) = X (\omega_{0})+\left(W_{t}^{(1)} (\omega_{1} ), \ldots, W_{t}^{(d)} (\omega_{d} )\right).
\end{equation}
Conventionally, for a stochastic process $X$ we denote by $\mathcal F^X_t$   the smallest $\sigma$-algebra for which all $X_s,s\leq t,$ are measurable. This   yields the filtration $\mathbb F^X=(\mathcal F_t^X)_{t\geq 0}$;  so-called the \name{natural filtration}. 
We conclude that $B=(B_t)_{t\geq 0}$ becomes the desired object on the filtered probability space $(\Omega',\mathcal F',\mathbb F^B,\mathbb P')$.

However,  the filtered probability space above does not necessarily satisfy the usual conditions. This can   be solved relatively easily by considering the \name{augmented filtration} $\mathbb F'$, which is an enlargement of  the filtration $\mathbb F^B$ by just adding   $\mathbb P'$-null sets. Remarkably, the new filtered probability space $(\Omega',\mathcal F',\mathbb F',\mathbb P')$ subsequently satisfies the usual conditions. We will not go into more detail, but refer to \citeb{p. 22}{book:protter} and \citeb{p. 89}{book:karatzas} for a proper discussion considering augmentation. There they discuss it within the setting of Lévy  processes and strong Markov processes, respectively. (In particular, a Brownian motion is both, see also  Proposition \ref{prop:strong-M}.) 


\begin{remark}\label{remark:3}
	Before we   continue, we want to stress   the fact  that   condition (iv) in Definition \ref{def:Brownian} is somewhat superfluous.
	  Indeed,  let us assume $B=(B_t)_{t\geq 0}$ satisfies all the conditions  except   (iv). Then by the  \name{Kolmogorov Continuity Theorem} \citeb{p. 53}{book:karatzas} we obtain   that  $B$ has a (unique) continuous version. 	  
	  Hence, we     assume without loss of generality   $B$ is a continuous process (or, in line with Remark \ref{remark:1}, take the continuous version without special mention).
\end{remark}



For the time being, let us take $d=1$. Observe that a one-dimensional Brownian motion $B$ with initial distribution $\mu$ is  a martingale as long as $\mathbb E|B_0|=\int_\R |x|\mu(\mathrm d x)<\infty.$ 




\begin{example}\label{ex:Brownian}
	Let $W=(W_t)_{t\geq 0}$ be a standard Brownian motion. It is clearly a continuous square integrable martingale. In particular, one easily deduces that $(W_t^2-t)_{t\geq 0}$ is   a martingale too. Following the discussion in {\S}\ref{Sec1.1.3}, we obtain $[W]_t=\langle W\rangle_t=t$ for all $t\geq 0.$
Alternatively, a more direct approach by means of Theorem \ref{thm:ucpquad} also shows that $[W]_t=t$ holds for all $t\geq 0$; see \citeb{p. 18}{book:protter}  for a proof. 

Moreover, almost all sample paths $t\mapsto W_t(\omega)$ are of unbounded variation on any interval, i.e., on any compact interval $[a,b]$ we have $FV_{[a,b]}(s\mapsto W_s(\omega))=\infty$ for $\mathbb P$-almost every $\omega\in\Omega$. We refer to \citeb{p. 19}{book:protter} for a proof, however it basically follows from the general fact that the quadratic variation vanishes for any continuous function of finite variation \citeb{p. 255}{book:kallenberg}.

 Ultimately, observe equation \eqref{eq:formula} in this setting becomes
\begin{equation}
	\mathbb E\left(\dint0t{H_s}W_s\right)^2=\mathbb E\dint0t{H^2_s}s,
\end{equation}
for suitable integrands $H$, which was one of the crucial results in Itô's original treatment of a stochastic integral, as is pointed out in \citeb{p. 77}{book:protter}.
\end{example}

It is not   that difficult to see that a $d$-dimensional Brownian motion $B$ starting at 0, when written as the vector $B=(B^1,...,B^n)$ consisting of independent standard Brownian motions,  satisfies
\begin{equation}
\label{eq:character} [ B^{i}, B^{j} ]_{t}=\langle B^{i}, B^{j} \rangle_{t}=\delta_{i j} t \quad \text { for } 1 \leq i, j \leq d,
\end{equation}
where $\delta_{i j}$ is the Kronecker symbol. Observe, the product $B^iB^j$ for $i\neq j$ is again a continuous martingale, hence $[B^i,B^j]$ is a continuous finite variation process as well as a local martingale by Theorem \ref{thm:preserve}. This yields $[B^i,B^j]=0$  due to the same argument as   in Remark \ref{remark:2}.

It turns out that this property \eqref{eq:character} characterises a Brownian motion among all continuous local martingales, as we state in the theorem below. Be aware, this is no longer valid when we consider \cadlag local martingales, which follows from the computations in {\S}\ref{Sec1.2.2}.


\begin{theorem}[\name{Lévy's Theorem}]
  Suppose $M=(M^1,...,M^d)$ is a  $d$-dimensional  vector consisting of continuous local martingales such that we have   $M_{0}=0$ $\mathbb P$-a.s. and
\begin{equation}
[M^{i}, M^{j}]_{t}=\delta_{i j} t \quad \text { for } 1 \leq i, j \leq d.
\end{equation}
Then  $M$ is a $d$-dimensional Brownian motion starting at 0.
\end{theorem}
\begin{proof}
	A proof of this statement can be found for instance in \citeb{p. 157}{book:karatzas}.
\end{proof}

Another important result is that  continuous local martingales, starting at 0 with  quadratic variation tending to infinity, can be characterised by a standard Brownian motion via a time-change argument; see Theorem \ref{thm:dubins}. 

\begin{theorem}\label{thm:dubins}\namethrms{Time-Change for Martingales [Dambis, Dubins \& Schwarz]}{Time-Change for Martingales}\index{Dambis--Dubins--Schwarz Theorem}  Let $M$ be a continuous local martingale with $M_0=0$ such that $\lim _{t \rightarrow \infty}[ M]_{t}=$
$\infty$ holds $\mathbb P$-a.s.. Define, for each   instant $t\geq 0$ fixed, the $\mathbb F$-stopping time
\begin{equation}
S(t)=\inf \left\{s \geq 0 ;\, [ M]_{s}>t\right\}.\label{eq:ST}
\end{equation}
Then the time-changed process $B=(B_t)_{t\geq 0}$, defined by
\begin{equation}
B_{t} = M_{S(t)}, \quad  s\geq 0,
\end{equation}
is a standard  Brownian motion adapted to $ \mathbb G=(\mathcal G_t)_{t\geq 0} $, where $\mathcal G_t=\mathcal F_{S(t)}$. In particular, the  latter filtration also satisfies the usual conditions, and   \textnormal{(}up to indistinguishability\textnormal{)} we have 
\begin{equation}
M_{t}=B_{[ M]_{t}}, \quad t\geq 0.
\end{equation}

\end{theorem}
\begin{proof}
		A proof of this statement can be found for instance in \citeb{p. 174}{book:karatzas}.
	A $d$-dimensional analogue can be found in \citeb{p. 179}{book:karatzas}.
\end{proof}
Following the discussion in {\S}\ref{Sec1.1.4}, we immediately obtain the following corollary. 
\begin{corollary}\label{cor:dubins}
	Suppose $W$ is a standard Brownian motion  and assume $H$ is a progressively measurable process such that for all $t\geq0$ we have   $
		\mathbb E\dint0t{H_s^2}s<\infty.$
Then  $M=(M_t)_{t\geq 0}$ defined by
$
	M_t=\dint0t{H_s}W_s
$
is a continuous $L^2$-martingale with
\begin{equation}
	[M]_t=\dint0t{H_s^2}s,\quad t\geq 0.
\end{equation}
Presume $\lim _{t \rightarrow \infty}[ M]_{t}=$
$\infty$. Let $(S(t))_{t\geq 0}$ be the family of $\mathbb F$-stopping times as in \eqref{eq:ST}, then the stochastic process $B=(B_t)_{t\geq 0}$ satisfying 
\begin{equation}
B_t=\int_{[0, S(t)]}H_s\,\mathrm d W_s\quad\text{and}\quad	M_t=B_{\textstyle  \int_0^t{\scriptstyle H_s^2 \mathrm ds}},\quad t\geq 0.
\end{equation}
is a standard Brownian motion on $(\Omega,\mathcal F,\big(\mathcal F_{S(t)}\big)_{t\geq 0},\mathbb P)$.

%
\end{corollary}

Now that we have extensively dealt with     Brownian motions starting at 0, it is time to initiate the necessary  definitions and results with regard  to Lévy processes. Let us write $\langle\cdot,\cdot\rangle$ and $\|\cdot\|$ for the standard inner product and norm   on the Hilbert space $\R^d$.

\begin{definition}\label{def:Levy}
	An $\mathbb{R}^{d}$-valued process $X=\left(X_{t}\right)_{t \geqslant 0}$   on   $(\Omega,\mathcal F,\mathbb F, \mathbb P)$ is called a \name{Lévy process}{Lévy process!definition of}, if it is $\mathbb F$-adapted and satisfies the properties
	\begin{itemize} 
	\item[\textnormal{(i)}] $\mathbb P(X_0=0)=1$;
	\item[\textnormal{(ii)}] $X$ has \name{stationary increments}, i.e., for all $0 \leqslant s \leqslant t,$ we have that $ X_{t}-X_{s}$ equals in distribution to $X_{t-s}$;
	\item[\textnormal{(iii)}] $X$ has \name{increments independent of the past}, i.e., for all $0 \leqslant s \leqslant t,$ we have $X_{t}-X_{s}$ is independent of $\mathcal{F}_{s}$;
		\item[\textnormal{(iv)}] $X$ is \name{continuous in probability}, i.e., for any      $t_0 \geqslant 0$ fixed, we have     $ \mathbb{P}\left(\left\|X_{t}-X_{t_0}\right\|>\varepsilon\right) \rightarrow 0$     as $t \rightarrow t_0$ for all $\varepsilon>0$. 
	\end{itemize}
\end{definition}

We see that   a Brownian motion starting at 0 clearly belongs to the class of   Lévy processes. One could think of Lévy processes as  tilted Brownian motions with jump occurrences. We are going to make this comment more rigorous in {\S}\ref{Sec1.2.1}. As a matter of  fact, we will   then see that all Lévy processes are semimartingales.

Observe, if $X=(X_t)_{t\geq 0}$ is a process satisfying parts (ii) and (iii) of Definition \ref{def:Levy},  then the characteristic function of $X_t$, that is,
\begin{equation} \varphi_{t}(u) =\mathbb{E}\left[e^ {\mathrm{i} \langle u, X_{t}\rangle}\right],\quad u \in \mathbb{R}^d, \,t \geq 0, \end{equation}
satisfies $\varphi_{t+s}(u) =\varphi_{t}(u) \varphi_{s}(u) $ for all $u\in\mathbb R^d$. 
If $X$ is a process that is continuous in probability, we obtain that the function $t\mapsto \phi_t(u)$ is continuous \citeb{p. 43}{book:applebaum}.  
 Combining all properties of a Lévy process yields   
\begin{equation}
\phi_t(u)=e^{ t \psi(u)},
\end{equation}
where $\psi: \mathbb{R} \rightarrow \mathbb{C}$ is a continuous function satisfying $\psi(0)=0 .$ The function $\psi$ is known as the \name{characteristic exponent} of the Lévy process (see also Theorem \ref{thm:levykhint}).


Moreover, due to      stationarity of   increments, condition (iv) is non-trivially equivalent---as  is  pointed out in \citeb{p. 43}{book:applebaum}---to being  right-continuous in probability at $t=0,$ i.e.,
\begin{equation}\label{eq:rcinprob}
  \lim _{h \searrow 0} \mathbb{P}\left(\left\|X_{h}\right\|>\epsilon\right)=0, \quad \epsilon>0.
\end{equation}
If $Y$ is some \cadlag stochastic process, then  $Y$ is in particular right-continuous at $t=0$, hence satisfies equation \eqref{eq:rcinprob}. Conclusively, condition (iv) of a Lévy process $X$ will be superfluous when we additionally assume that $X$ has $\mathbb P$-a.s. \cadlag sample paths. Even more is true.

\begin{remark}
	The definition above is, e.g., conform \cite{book:applebaum} and \cite{book:protter}. Lots of literature such as  \cite{book:sato} and \cite{book:kyp} assume    $X$ is a \cadlag process. Analogous to Remark \ref{remark:3}, this stronger  assumption can be made without loss of generality. Indeed, for every Lévy process $X$ there exists a unique (up to indistinguishability) \cadlag version, which is proven for instance in  \citeb{p. 63}{book:sato}, \citeb{p. 21}{book:protter}, and \citeb{p. 87}{book:applebaum}. Therefore, as usual, we  will     always take the \cadlag version  without special mention.
\end{remark}

Whenever $X$ is a \cadlag process which is continuous in probability, we obtain that $X$ has no fixed discontinuities $\mathbb P$-a.s., that is, $\mathbb P(\Delta X_t\neq 0)=0$ for all $t\geq 0.$ In other words,  the jumps are   fully ``unpredictable'' and   thus random. 
%
%
%

%
%
%
%
 
%




\begin{definition}
	Let $(M,d)$ be a metric space. An $(M,d)$-valued process $X=(X_t)_{t\geq 0}$ on a filtered probability space $(\Omega,\mathcal F,\mathbb F,\mathbb P)$ is said to be \name{continuous in probability}, or \name{stochastically continuous}, if for any $t_0\geq 0$ fixed, we have
	\begin{equation}
		\mathbb P(d(X_t,X_{t_0})>\epsilon)\to 0,\quad \text{as }t\to t_0,
	\end{equation}
for all $\epsilon>0.$
\end{definition}

For every Lévy process $X=(X_t)_{t\geq 0}$, we have the significant property  that the characteristic function  of $X_t$, for any fixed instant $t\geq 0,$  has an  explicit expression; see   Lévy--Khintchine's representation in Theorem \ref{thm:levykhint}.  For this we need to   define what is meant by  a Lévy measure. There are several equivalent definitions; the  definition below is conform \citeb{p. 29}{book:applebaum} and \citeb{p. 37}{book:sato}. 


 \begin{definition}
 	A \name{Lévy measure}{Lévy measure} on $\mathbb{R}^{d}$ is a $\sigma$-finite measure $\nu$ on $\mathbb{R}^{d}$ such that $\nu(\{0\})=0$ holds and
 	\begin{equation}
 	\int_{\mathbb{R}^{d}}\left(1 \wedge\|x\|^{2}\right) \mathrm{d} \nu(x)<\infty.\label{eq:levyint}
 	\end{equation}
 \end{definition} 
 Note that any Lévy measure $\nu$  satisfies the following:   $\nu(A)<\infty$  for all $A\in\mathcal B(\R^d)$  \name{bounded away from zero}, i.e., the closure of $A$ may not contain 0. In words, Lévy measures appear to have no mass at the
 origin, while singularities can occur around the origin. Differently put,  the corresponding L\'evy processes allow infinitely many small jumps  to happen, yet only a finite number of big jumps are possible. We   justify the latter translation after Example \ref{ex:complex}.
 

\begin{theorem}\label{thm:levykhint}
	 \namethrm{Lévy--Khintchine representation} Let $X=(X_t)_{t\geq 0}$ be a Lévy process.  
	 \begin{enumerate}[\normalfont(i)]
	 	\item For each $t \geqslant 0$, the characteristic function $\varphi_{t}$ of the random variable $X_{t}$ satisfies
	 	\begin{equation}
	 		\varphi_{t}(u)=\mathbb{E}\left[e^{i\left\langle u, X_{t}\right\rangle}\right]=e^{t \psi(u)}, \quad u \in \mathbb{R}^{d},
	 	\end{equation}
	 	with \name{characteristic exponent} $\psi(u)$ given by
	 	\begin{equation}\label{eq:repres}
	 		\psi(u)=i\langle u, b\rangle-\frac{1}{2}\langle u, \Sigma u\rangle+\int_{\mathbb{R}^{d}}\left(e^{i\langle u, x\rangle}-1-i\langle u, x\rangle \mathds{1}_{\{\|x\| \leqslant 1\}}\right)\,\nu( \mathrm{d} x),
	 	\end{equation}
	 	where $b \in \mathbb{R}^{d},$ $ \Sigma$ is a real  symmetric non-negative  definite $d \times  d$ matrix, and ultimately $\nu$ is a Lévy measure on $\mathbb{R}^{d}$.
	 	\item  The representation of $\psi$ by $b,$ $\Sigma$,    $\nu$ as in \eqref{eq:repres}  is in fact unique. Thus $X$ can be associated with a tuple $(b, \Sigma, \nu)$,  known as the \name{characteristic triplet} of $X .$
	 	 
	 \end{enumerate}
	  Conversely, suppose $b \in \mathbb{R}^{d},$ let $ \Sigma$  be any  real  symmetric non-negative definite $d \times  d$ matrix, and   take $\nu$ to be some Lévy measure on $\mathbb{R}^{d}$.
	 \begin{itemize} 
	 	\item[\textnormal{(iii)}] Then there is a filtered probability space $(\Omega,\mathcal F,\mathbb F,\mathbb P)$---satisfying the usual conditions---on which a Lévy process $X$ exists whose  characteristic triplet is given by $(b, \Sigma, \nu)$.	 \end{itemize}
	 
\end{theorem} 
\begin{proof}
Parts (ii)	and (iii) are relatively simple to prove, for which we refer to \citeb{p. 40}{book:sato} and \citeb{p. 41}{book:sato} (or   \citeb{p. 30}{book:applebaum} for instance) respectively. Observe the theorem above is often stated in terms of \name{infinitely divisible distributions}. In short, there is a 1-1 correspondence  between infinitely divisible distributions and Lévy processes; see   \citeb{Thm. 7.10}{book:sato}.
The ``usual conditions'' part of (iii) follows from  \citeb{Thm. I.31}{book:protter}; the natural filtration   of any Lévy process combined with all the $\mathbb P$-null sets is known to be right continuous.

Part (i) on the other hand is much more difficult to handle with. Nowadays it is seen as a by-product of the \name{Lévy--Itô decomposition}, e.g., conform \citeb{p. 127}{book:applebaum}. This decomposition is stated in    Theorem \ref{thm:LevyIto} for the one-dimensional case, and again we refer to \citeb{p. 119}{book:sato} or \citeb{p. 126}{book:applebaum} for its $d$-dimensional analogue. However, part (i) can alternatively be proven via a direct analytical approach, as is done in \citeb{p. 42}{book:sato}.
\end{proof}
Conclusively, the above motivates us to speak of a Lévy process $X=(X_t)_{t\geq 0}$ with characteristic triplet $(b,\Sigma,\nu)$.  As characteristic functions uniquely determine the underlying probability distributions,    it is  important to note   each Lévy process is uniquely determined by its   triplet.

\begin{remark} More generally, the Lévy--Khintchine representation depends on the choice of \name{truncation function}, i.e.,\label{remark:levykhint}  a bounded function  $h:\R^d\to\R^d$ which   satisfies $h(x)=x$ for all $x$ in a neighbourhood of 0. Typically, in   \eqref{eq:repres} we have set $h(x)=x\mathds 1_{\{\|x\|\leq 1\}}$, and we refer to the latter   as the canonical choice.
	
	If we choose to use another truncation function $h$ instead of the canonical one, then it yields that the characteristic exponent in equation \eqref{eq:repres} can be rewritten as  
	\begin{equation}
		\psi(u)= i\left\langle u, b_{h}\right\rangle-\frac12{\langle u, \Sigma u\rangle}+\int_{\mathbb{R}^{d}}\left(\mathrm{e}^{i\langle u, x\rangle}-1-i\langle u, h(x)\rangle\right) \nu(\mathrm{d} x),
	\end{equation}
	with $b_{h}\in\R^d$  defined by  
	\begin{equation}
		b_{h}:=b+\int_{\mathbb{R}^{d}}\left(h(x)-x \mathds 1_{\{\|x\|\leq 1\}}\right) \nu(\mathrm{d} x).
	\end{equation}
	Note that the matrix $\Sigma$ and Lévy measure $\nu$ act invariantly with respect to truncation functions. In order to  avoid any confusion,  we will denote the triplet by $\left(b_{h}, \Sigma, \nu\right)$. When $h$ is clear from the context, we may drop this notation, and simply write $(b,\Sigma, \nu)$. 
	
	In these notes, we always take the canonical truncation function  unless specified otherwise.
\end{remark}

The next   example is simply a direct corollary of the Lévy--Khintchine representation.  
Sometimes, as in Example \ref{ex:complex},  one writes $\psi_t(u)=t\psi(u)$; 
see \citeb{p. 75}{book:jacod} for instance.

\begin{example}\label{ex:complex}
	Let $X$ be a Lévy process on $(\Omega,\mathcal F, \mathbb F, \mathbb P)$. Then,    for any $u\in \R^d$, we have
	\begin{equation}
		\left(\frac{\exp(i\langle u,X_t\rangle)}{\exp \psi_t(u) }\right)_{t\geq 0}
	\end{equation}
	is a complex-valued martingale (that is, both the real and imaginary parts are martingales) with respect to the same filtration $\mathbb F$. Indeed, for any $t\geq s$,  we find via   simple verification
	\begin{align*}
		\mathbb E\left[\frac{\exp(i\langle u,X_t\rangle)}{\exp \psi_t(u) }\mid \mathcal F_s \right]&=\mathbb E\left[\frac{\exp(i\langle u,X_t-X_s\rangle)}{\exp \psi_{t-s}(u )}\frac{\exp(i\langle u,X_s\rangle)}{\exp \psi_s(u) }\mid \mathcal F_s \right]\\
		&=\frac{\exp(i\langle u,X_s\rangle)}{\exp \psi_s(u )}\frac{\mathbb E\left[{\exp(i\langle u,X_t-X_s\rangle)} \mid \mathcal F_s \right]}  {\exp \psi_{t-s}(u) }\\
		&=\frac{\exp(i\langle u,X_s\rangle)}{\exp \psi_s(u )}\frac{\mathbb E\left[{\exp(i\langle u,X_t-X_s\rangle)}  \right]}  {\exp \psi_{t-s}(u) }\\
		&=\frac{\exp(i\langle u,X_s\rangle)}{\exp(\psi_s(u) }\frac{\mathbb E\left[{\exp(i\langle u,X_{t-s}\rangle)}  \right]}  {\exp \psi_{t-s}(u))}\\
		&=\frac{\exp(i\langle u,X_s\rangle)}{\exp \psi_s(u) },
	\end{align*}
	making explicitly use of the fact that increments are stationary and independent of the past. Also, notice that adaptivity and integrability in this context is clear.
\end{example}

 Let $(b,\Sigma,\nu)$ be the triplet of the Lévy process $X=(X_t)_{t\geq 0}$. As we already have suggested, the Lévy measure $\nu$ describes the jump behaviour of   $X$.  We motivate this by   mentioning a few properties and, subsequently, by giving an equivalent definition of   Lévy measures.
 
 Firstly, notice that if $\nu (\R^d)=0$   holds, then the characteristic exponent simplifies into  
 \begin{equation}
 	\psi(u)=i\langle u,b\rangle - \frac12\langle u,\Sigma u\rangle,\quad u\in \R^d.
 \end{equation}
This characteristic exponent typically belongs to the continuous process
\begin{equation}
	X_t=bt+\sigma B_t,\quad t\geq 0,
\end{equation}
where $\sigma$ is some $d\times d$-matrix with $\sigma \sigma^T=\sigma ^T\sigma=\Sigma$ and $B=(B_t)_{t\geq 0}$ a $d$-dimensional Brownian motion starting at 0  (see also Example \ref{ex:BM}). Observe $X_1 \sim \mathcal N(b  , \Sigma  )$. Recall that the choice of $\sigma$ plays no role \citeb{p. 355}{book:vandervaart}, however, one often chooses the (unique) {non-negative square root} $\sigma:=\sqrt{\Sigma}$ of the dispersion matrix $\Sigma$.
In conclusion, the Lévy process $X$ attains $\mathbb P$-almost surely no jumps, i.e., the process $X$ is continuous, if and only if $\nu (\R^d)=0$.  

Secondly, whether $\nu (\R^d)$ is finite or not tells us something about the amount of jumps $X$ attains. For a sample path $t\mapsto X_t(\omega)$, for any realisation $\omega\in\Omega$, we can speak of the set of all its jumping times $\{t\in \Rplus:\Delta X_t(\omega)\neq 0\}$.

 \begin{proposition}[Theorem 21.3 of  \cite{book:sato}]\label{prop:intensity}
 	Let $X=(X_t)_{t\geq 0}$ be a Lévy process with characteristic triplet $(\gamma,\Sigma,\nu)$.
 	\begin{enumerate}[\normalfont(i)]
 		\item If $\nu(\R^d)<\infty$ holds, then $\mathbb P$-almost every path of $X$ has a finite number of jumps on every compact time interval. Moreover,   the set of all jumping times is $\mathbb P$-a.s. countably infinity.
 		\item If $\nu(\R^d)=\infty$ holds, then the set of all jumping times of $\mathbb P$-almost every path of $X$ is countably infinite and dense in $\Rplus$.
 	\end{enumerate}
 \end{proposition}
The above motivates the following terminology. We say that a Lévy process $X$ is of \name{finite intensity}, or that it has \name{finite activity}, if and only if $\nu(\R^d)<\infty.$ Analogously,  a Lévy process $X$ is said to be of \name{infinite intensity}, or that it has \name{infinite activity}, if and only if $\nu(\R^d)=\infty.$ 

 The countability of  all the jumping times on $\Rplus$, for both cases, should actually not come as a surprise, since this immediately follows from $X$ being \cadlag (remember the discussion  succeeding Definition \ref{def:first}). Moreover, recall that for any \cadlag process $X$---with \cadlag sample paths everywhere (without loss of generality)---the number of jumps $\Delta X_{s}$ such that  $\left\|\Delta X_{s}\right\| \geq \varepsilon$ holds, with $0<s\leq t$, has to be finite everywhere and for all fixed $\varepsilon>0 .$ Consequently, for any    measurable Borel set $A\in\mathcal B(\R^d)$   bounded away from zero, that is, $\bar A\subset  \R^d\backslash\{0\}$, and all $t\geq 0$, we obtain that the function $N_t^A:\Omega\to\R$, defined by
 \begin{equation}
 N_{t}^{A}:=\#\left\{s \in[0, t]: \Delta X_{s} \in A\right\}=\sum_{0<s\leq t}\mathds 1_{\{\Delta X_s\in A\}}=\sum_{n\in\N} \mathds 1_{\{T^A_n\leq t\}},\label{eq:begin}
\end{equation}
where
\begin{equation}
	 T^{A}_{1}:=\inf \left\{t>0: \Delta X_{t} \in A\right\}, ...,
	 T^{A}_{n+1}:=\inf \left\{t>T^{A}_{n}: \Delta X_{t} \in A\right\}, \quad n\in\N,\label{eq:stoptimes}
\end{equation}
 is a well-defined random variable, which only attains values in $\mathbb N_0=\{0,1,2,...\}$.\index{$\mathbb N_0 $, natural numbers including zero} Indeed, since the process $X$ is \cadlag and adapted, we have that all the $T_n^A$ are  stopping times, hence
 \begin{equation}
 	\{N_t=n\}=\{\omega\in\Omega:T_n^A(\omega)\leq t<T_{n+1}^A(\omega)\}\in\mathcal F_t\subset\mathcal F,\label{eq:adaptiv}
 \end{equation}
 for each $n\in\N_0$, where $T_0^A:=0.$ Another, and a somewhat more direct,  approach showing that  $N_t^A$ is finite ($\mathbb P$-almost) everywhere can be found in    \citeb{p. 2}{book:sulem} and \citeb{p. 101}{book:applebaum}. 

 Moreover, note that \index{${N}^A$} $N^A=(N_t^A)_{t\geq 0}$ defines an adapted counting process without explosions, for any $A\in\mathcal B(\R^d)$ bounded away from zero.

 \begin{definition}
Let  $(T_n)_{n\in\N}$ be a  strictly increasing sequence\label{def:counting}  	of positive random variables. A stochastic process $N=(N_t)_{t\geq 0}$ on $(\Omega,\mathcal F,\mathbb P)$, given by
\begin{equation}
	N_t=\sum_{n\in\N} \mathds 1_{\{T_n\leq t\}},\quad t\geq 0,
\end{equation}
with values in $\mathbb N_0\cup \{\infty\}$ is said to be the \name{counting process} associated to the sequence $(T_n)_{n\in\N}$. 

In addition, let us set $T=\sup_{n\in\N}T_n$, which is called the \name{explosion time} of $N$. If $T=\infty$ holds $\mathbb P$-a.s., then $N$ is  a counting process \name{without explosions}. 
 \end{definition}
Observe that counting processes without explosions have \cadlag sample paths $\mathbb P$-a.s.. Consequently, all processes   $N^A=(N_t^A)_{t\geq 0}$ are \cadlag. The adaptivity of $N^A$ is also clear; it  follows from \eqref{eq:adaptiv}. In general, it can easily  be seen that a counting process $N$ is adapted if and only if the associated random variables $(T_n)_{n\in\N}$ are stopping times. Finally, notice that all adapted counting processes are increasing process  and hence   of finite variation.

 \begin{remark}\label{remark:measurability}
	For every remaining $A\in\mathcal B(\R^d\backslash \{0\})$, or equivalently, for all $A\in\mathcal B(\R^d)$ with $0\notin A$ (but now $0\in \bar A$ is allowed), it is also possible to properly define  \begin{equation}N_{t}^{A}=\#\left\{s \in[0, t]: \Delta X_{s} \in A\right\}=\sum_{0<s\leq t}\mathds 1_{\{\Delta X_s\in A\}}.\end{equation}
	Whenever $X=(X_t)_{t\geq 0}$   has a  finite amount of jumps on every compact interval (e.g., compound Poisson processes, see Example \ref{ex:compound}), then we can exactly mimic all the lines starting from equation \eqref{eq:begin}  until   \eqref{eq:end}. Every observation  above then remains valid. 
	
	In the more general case, we can still see $N_t^A$, for any $t\geq 0$, as a countable sum of indicator functions together with stopping times, namely:
	\begin{equation}
		N_t^A=\sum_{k\in\N}  \sum_{n\in\N}\mathds 1_{\{T_n^{A_k}\leq t\}},
	\end{equation}
where $A_k:=A\cap \{x\in\R^d:\|x\|>1/k\}$, for all $k\in\N$, and the $T_n^{A_k}$ are as in equation \eqref{eq:stoptimes} with $A=A_k$. Importantly, note that $N_t^A:\Omega\to[-\infty,\infty]$ is now written as a countable sum of random variables, hence a random variable itself (attaining values on the extended real line). 
All random variables $N_t^A$ remain $\mathcal F_t$-measurable, but we may have $\mathbb EN_1^A=\infty$. In general, the process $N=(N_t)_{t\geq 0}$ fails to be a counting process.
\end{remark}

As a matter of fact, when $X$ is a Lévy process,   it turns out that $N^A$ is a Poisson process (see Example \ref{ex:poisson}) for any measurable set $A\in\mathcal B(\R^d)$  bounded away from zero \citeb{p. 26}{book:protter}, which  implies
\begin{equation}\mathbb E N^A_1=\mathbb E\big[\#\left\{s \in[0, 1]: \Delta X_{s} \in A\right\}\big]<\infty.\label{eq:end}\end{equation}
This observation is going to be key for the alternative definition of a Lévy measure.

 One should think of   Lévy measures  as describing the expected number of jumps  
of a certain size within a time interval of length one. This is justified by the following (promised) definition.
\begin{definition}\label{def:Levy2}
Suppose $X$ is a Lévy process. Then we write $\nu^{X}$ for the $\sigma$-finite measure on $\mathbb R^d$ satisfying $\nu^X(\{0\})=0$ and
\begin{equation}\nu^X(A)=\mathbb E\big[\#\left\{s \in[0, 1]: \Delta X_{s} \in A\right\}\big]=\mathbb E\left[\sum_{0<s\leq t}\mathds 1_{\{\Delta X_s\in A\}}\right],\label{eq:def}\end{equation}
for all $A\in\mathcal B(\R^d\backslash \{0\})$, is called the \name{Lévy measure associated to $X$.}{Lévy measure}
\end{definition}
This definition of a Lévy measure is in line with, e.g., \citeb{p. 26}{book:protter} and \citeb{p. 88}{book:tankov}. The measure $\nu$ of a Lévy process $X$ with characteristic triplet  $(b,\Sigma,\nu)$ coincides with $\nu^X$. We claim that this follows from (the proof of) the Lévy--Itô decomposition; see Theorem \ref{thm:LevyIto}. Due to its equivalence, we will drop the $X$  and always write $\nu$  from now on (with Remark \ref{remark:clarify} as only exception).

\begin{remark}
	It is actually not immediately clear why $\nu^X$ is well-defined.\label{remark:clarify} Observe that, for any $\omega\in\Omega$ fixed, the set function \begin{equation}\mathcal B(\R^d\backslash \{0\})\to[0,\infty],\,A\mapsto N_1^A(\omega)\end{equation}
	 is a counting measure on $\R^d\backslash \{0\}.$ Consequently, since we know from Remark \ref{remark:measurability}  that all functions $N_1^A:\Omega\to[-\infty,\infty]$ are measurable, we obtain
	 \begin{equation}
	 	\nu^X\left(\bigsqcup _{k=1}^{\infty }E_{k}\right)=\int_\Omega \left(\sum _{k=1}^{\infty }N_1^{E_k}(\omega)\right) \mathrm d\mathbb P(\omega)=\sum _{k=1}^{\infty }\int_\Omega  N_1^{E_k}(\omega) \,\mathrm d\mathbb P(\omega)=\sum _{k=1}^{\infty }\nu^X (E_{k}), 
	 \end{equation}
for all sequences ${\displaystyle (E_{k})_{k=1}^{\infty }}$ of pairwise disjoint sets in $\mathcal B(\R^d\backslash \{0\})$. This follows from the monotone convergence theorem. We conclude, the set function
 \begin{equation}\nu ^X: \mathcal B(\R^d\backslash \{0\})\to[0,\infty],\,A\mapsto \mathbb E N_1^A\end{equation}
defines a measure on $\R^d\backslash\{0\}$, which is in fact $\sigma$-finite:   $\nu^{X}( \{x\in\R^d:\|x\|>1/k\})<\infty$ holds for all $k\in\N$. This justifies the explicit expression in \eqref{eq:def} for any set $A\in\mathcal B(\R^d\backslash \{0\})$.
%
%

	If one would rather  think of expectations for integrable functions only (i.e., with finite expectation), note that equation \eqref{eq:def} defines a $\sigma$-finite \textit{pre-measure} on the collection of   all Borel measurable sets bounded away from zero. Carathéodory's extension theorem then tells us that there is a unique $\sigma$-finite measure on $\R^d\backslash \{0\}$ extending the pre-measure. Either way, we obtain the same measure (due to the $\sigma$-finiteness).


Finally, in order to obtain the unique measure $\nu^X$ on $\R^d$ as in Definition \ref{def:Levy2}, one may take the push-forward measure of $\nu^X$ on $\R^d\backslash \{0\}$ under  the inclusion map $\iota:\R^d\backslash \{0\}\to\R^d$. This immediately yields the  $\nu^X(\{0\})=0$ property, and it is  also obviously $\sigma$-finite again.
\end{remark}

  Proposition \ref{prop:intensity} essentially follows from the equivalent interpretation of a Lévy measure. 
  Before we head towards the next subsection, it is worthwhile  mentioning that  the moments of a Lévy process $X$ are directly related to integrability of its Lévy measure $\nu$.
  \begin{theorem}[Theorem 25.3 of \cite{book:sato}]\label{thm:expec}
  	Let $X$ be a Lévy process with characteristic triplet $(b, \sigma^2, \nu) .$ Suppose $p\in\Rplus$. The random variable
  	$X_{t}$ has a finite $p$-th moment for any instant $t\geq 0$, i.e., $  \mathbb{E}\left|X_{t}\right|^{p}<\infty,\, t\geq 0,$ if and only if $\int_{|x| \geq 1}|x|^{p} \,\nu(\mathrm{d} x)<\infty$.
  \end{theorem}
  As a matter of fact, we have stated the above for the typical functions \begin{equation}g(x)=|x|^p\vee 1,\quad p\geq 0,\end{equation}  but any non-negative measurable function $g:\R\to\R$ bounded on compacts and satisfying the submultiplicative property $g(x+y)\leq \alpha g(x)g(y)$, for some constant $\alpha>0$, suffices.
 
 Another and final important property of   Lévy processes (unrelated  to the discussion on equivalent definitions of $\nu$) is it that they are    strong Markov processes. 
 Observe that this result in terms of applications can be very useful in combination with Proposition \ref{prop:useful-id}.

\begin{proposition}\namethrms{Strong Markov and renewal property for Lévy processes}{Lévy process!strong Markov property}\index{Lévy process!renewal property}\label{prop:strong-M}
	Let $X$ be a Lévy process on $(\Omega,\mathcal F, \mathbb F, \mathbb P)$ and assume  $T$   is a    stopping time. Conditioned on the set $\{T<\infty\}$, we have that the process    $\bar X^T$, defined  by
	\begin{equation}
		\bar X^T_t=X_{t+T}-X_T,\quad t\geq 0,
	\end{equation}
	is    a Lévy process on $(\Omega,\mathcal F,(\mathcal F_{t+T})_{t\geq 0},\mathbb P)$, $\bar X^T$ is independent of $\mathcal F_T$, and moreover $\bar X^T$ has the same distribution as $X$.  
\end{proposition}
\begin{proof}
	For a proof, see \citeb{p. 23}{book:protter}. It basically exploits the result in Example \ref{ex:complex}.
\end{proof}

\subsection{Examples, illustrations, and the Lévy--Itô decomposition}\label{Sec1.2.1}
\noindent  In the remainder of this section, we restrict  ourselves to the one-dimensional setting.  We claim that any result below extends to the $d$-dimensional case, but for our intents and purposes $d=1$ suffices and appears to be less technical sometimes.


  The terms of   a Lévy triplet $(b,\sigma^2,\nu)$, where we replace   $\Sigma $ by  $\sigma^{2}$,  suggest that a Lévy process can be seen as having three independent components: a linear drift, a Brownian motion, and a  Lévy jump process. In particular, recall the discussion after Remark \ref{remark:levykhint}. This observation has a mathematical rigorous interpretation, which is given by the Lévy--It\^o decomposition in Theorem \ref{thm:LevyIto}. Before we head towards this result, we provide the most fundamental examples of Lévy processes.
  
  Let us recall that a Lévy process is continuous if and only if $\nu(\R)=0$, hence the following example demonstrates all possible Lévy processes that are pathwise continuous.
  
  \begin{example}\label{ex:BM}
  	Suppose $b\in\R$ and $\sigma^2\geq0.$ Then, a stochastic process $X=(X_t)_{t\geq 0}$ is said to be a Brownian motion with \name{drift coefficient} $b$ and \name{dispersion coefficient} $\sigma^2$, if \begin{equation}\label{eq:brdrift}X_{t}=b t+\sigma W_{t},\quad t \geqslant 0,\end{equation} 
  	where $W=(W_t)_{t\geq 0}$ is a standard Brownian motion. Every $X$ of the form \eqref{eq:brdrift} is a Lévy process with characteristic triplet $(b, \sigma^{2}, 0)	$.
  	
  	Figure \ref{fig:drift}   illustrates a continuous Lévy process, i.e., a process as in equation \eqref{eq:brdrift}. Note, we can write   $X_t=\sum_{j=1}^n X_{j\Delta}-X_{(j-1)\Delta  }$ with $n\Delta  =t$, where $X_{j\Delta}-X_{(j-1)\Delta} \sim \mathcal N(b\Delta,\sigma^2\Delta)$, which particularly motivates the numerical scheme in Appendix \ref{matlab-drift}.
%
%
  \end{example}

Throughout these notes, we will  comment very little to none upon the accuracy and stability of our simulations (see {\S}\ref{Sec1.3} and its references, e.g.,  \cite{article:platen} and \cite{book:kloeden}, for more about ``the error'' of our approximations). Instead, throughout this section we refer to, e.g.,  \cite{book:kroese} and \cite{article:manuge} for an extensive overview on the relevant Monte Carlo methods. Especially, the scheme  in Appendix \ref{matlab-drift} can be found in \citeb{p. 181}{book:tankov}, \citeb{p. 182}{book:kroese}, and \citeb{p. 4}{article:manuge}. 
%
%
%



 \begin{figure}[!ht]
	\begin{center}
		\includegraphics[width=0.675\textwidth]{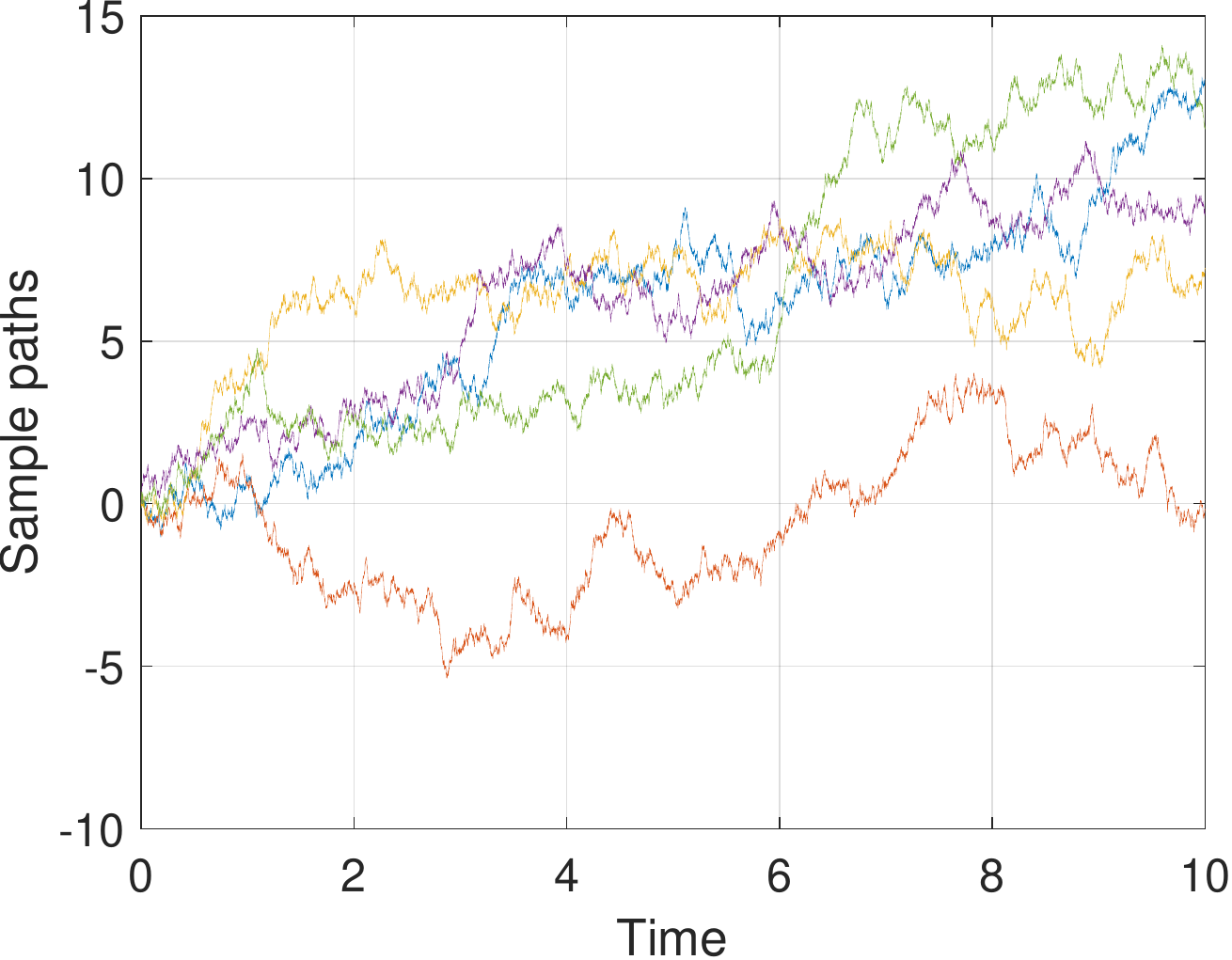}
	\end{center}
	\caption{Five sample path realisations of the one-dimensional Lévy process $L=(L_t)_{t\geq 0}$, given by $L_t=b t+\sigma W_t$ with $b=1$, $\sigma =2$ and where $W=(W_t)_{t\geq 0}$ is a standard Brownian motion.}  \label{fig:drift}   
\end{figure}




\begin{example}
	Let $B=(B_t)_{t\geq 0}$  be a standard Brownian motion.  Theorem\label{ex:interm} \ref{thm:approx} in combination with a straightforward computation (see \citeb{p. 59}{book:protter}, for instance), or a simple application of Itô's formula, see Theorem \ref{thm:ito}, yields the following famous identity:
	 \begin{equation}
		\dint0t{B_s}B_s=\frac12B_t^2-\frac12t.\label{eq:SI}
	\end{equation}
\pagebreak This (or any other) stochastic integral process  can be simulated by approximating the integral; see the  numerical scheme in Appendix \ref{matlab-WdW}  for the integral in \eqref{eq:SI}, which is inspired by Theorem \ref{thm:approx}  and conform to, e.g., \cite{phdthesis:norton}. 
On the other hand, one can  simulate this process in \eqref{eq:SI} by using the explicit expression on  the right hand side.
Figure \ref{fig:compare} compares both methods.
\end{example}
\begin{figure}[!ht]
	\centering
	\includegraphics[width=0.4715\linewidth]{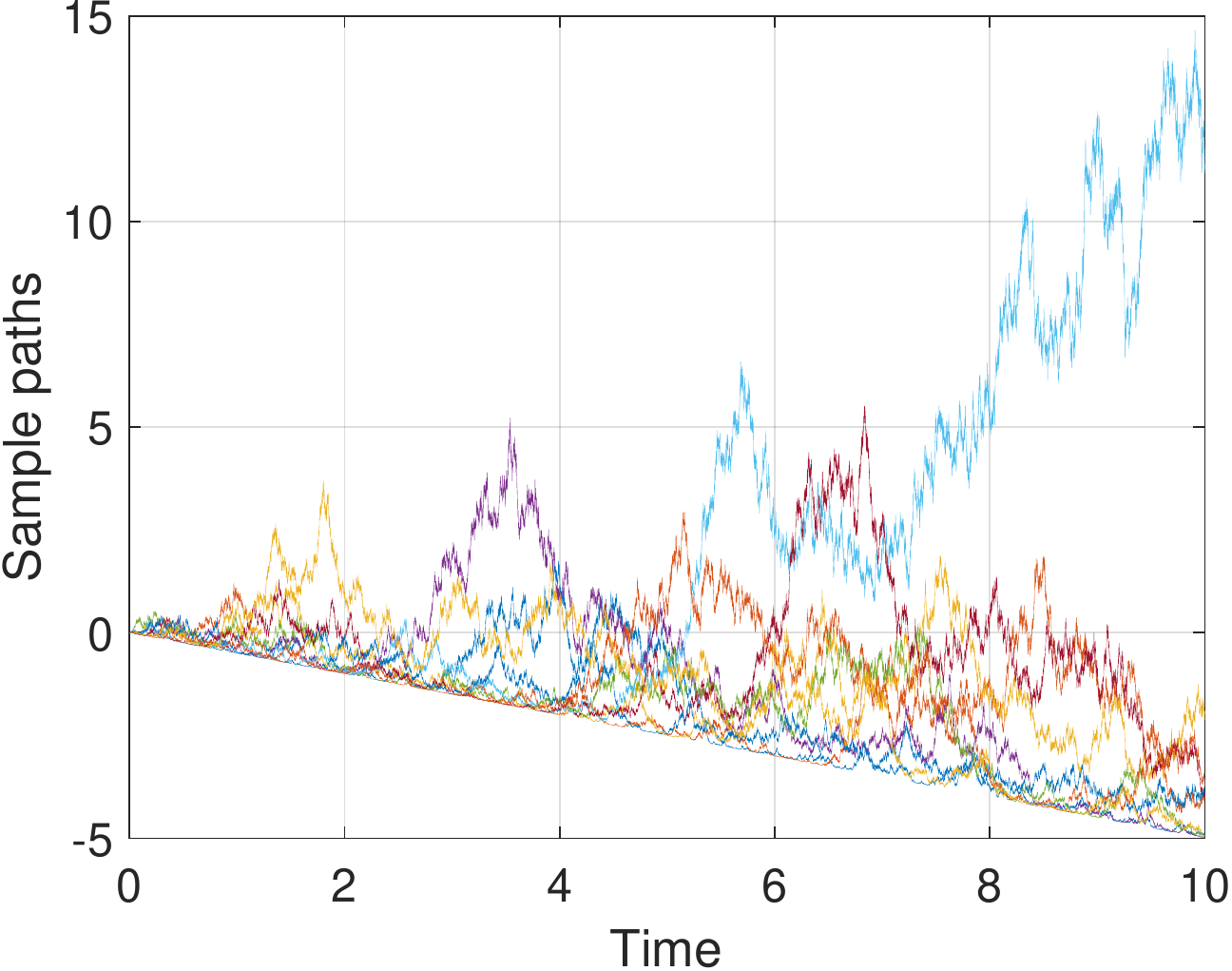} 
	\includegraphics[width=0.495\linewidth]{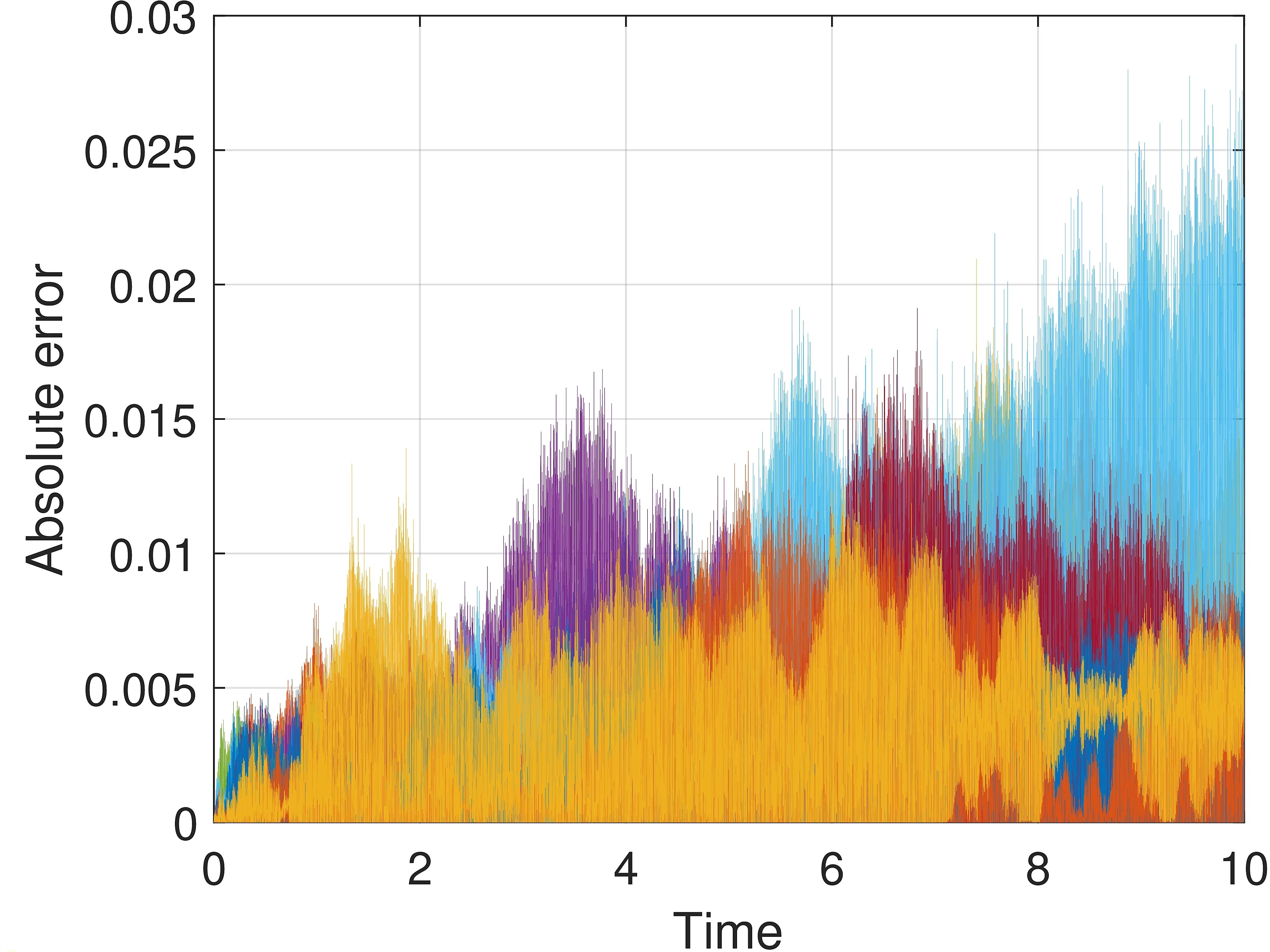}
	\caption{On the left, we see ten sample path realisations of the stochastic integral in equation \eqref{eq:SI}, simulated by the integral approximation scheme in Appendix \ref{matlab-WdW}. On the right, we plot the  absolute difference  between the latter and the direct numerical approximation of the stochastic process $\frac12B_t^2-\frac12t$ with $\Delta=10^{-6}.$ (We notice that  (much) larger choices for $\Delta$ give reasonable results too.) }
	\label{fig:compare}
\end{figure}
  Figure \ref{fig:compare} illustrates that the difference between the two methods appears to be negligible, hence motivates us that the scheme in Appendix \ref{matlab-WdW} is appropriate. Moreover, we provide in Appendix \ref{matlab-fdW}  a more general numerical scheme for stochastic integrals  with a deterministic integrand.

Now, let us return to giving explicit examples of Lévy processes. Another fundamental example are  Poisson   processes. We   introduce the latter in line with \citeb{p. 13}{book:protter}.

\begin{figure}[!t]
	\begin{center}
		\includegraphics[width=0.675\textwidth]{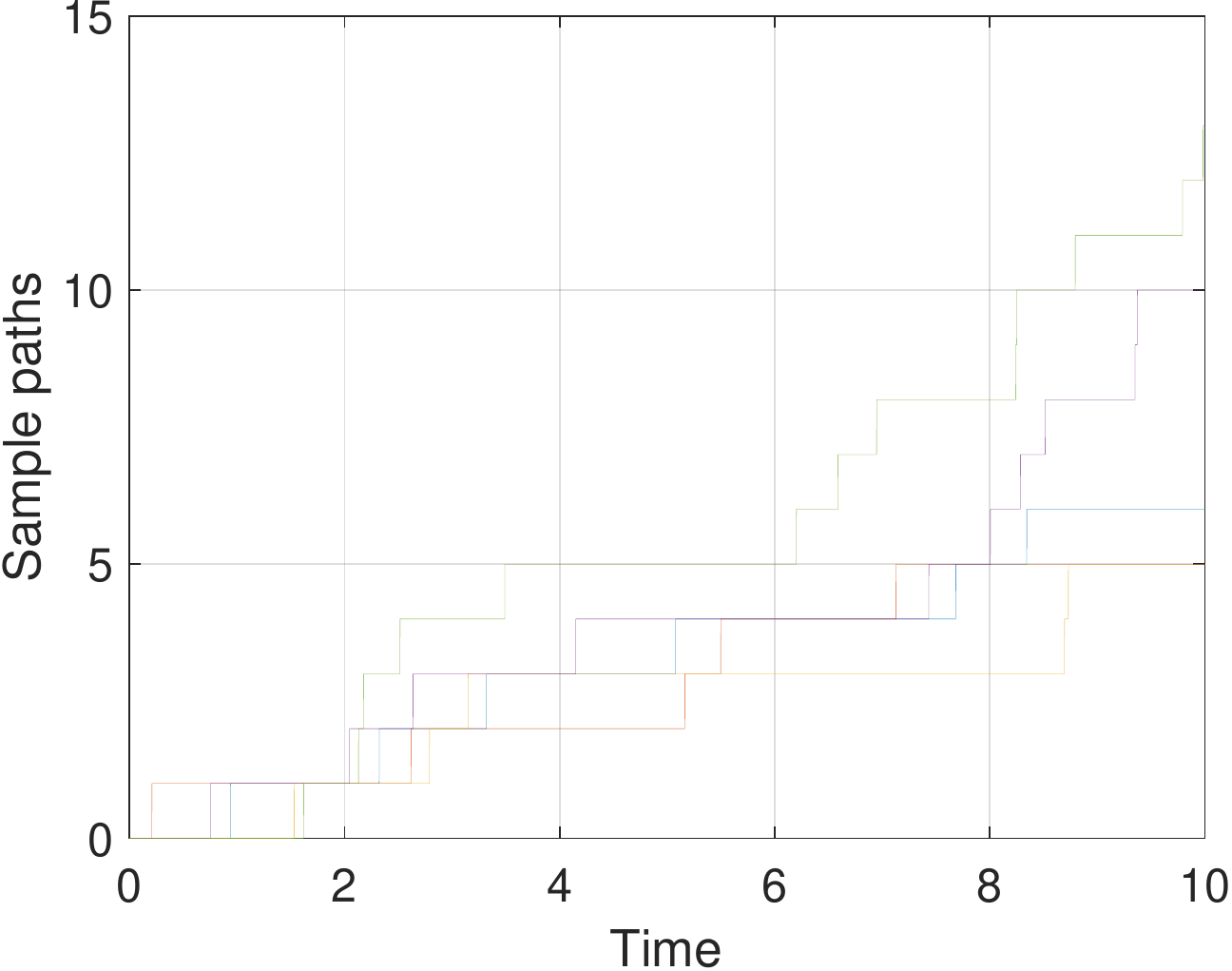}
	\end{center}
	\caption{Five sample path realisations of a Poisson process $N=(N_t)_{t\geq 0}$ with intensity $\lambda=1$.}     
\end{figure}

\begin{example}
	\label{ex:poisson} An adapted counting process $N$ (recall Definition \ref{def:counting}) is called a \name{Poisson process} on $(\Omega,\mathcal F,\mathbb F,\mathbb P)$ if it has stationary increments that are also independent of the past. Observe, conditions (i)--(iii) of a Lévy process as in Definition \ref{def:Levy} are already satisfied. These conditions alone imply---according to Theorem 23 of \citeb{p. 13}{book:protter}---that\footnote{By convention, we allow $\lambda =0$ and $0^0=1$. Then for all $t\geq0$,  $N_t$ is  Poisson distributed with parameter 0.} 
	\begin{equation}
		\mathbb P(N_t=n)=\frac{e^{-\lambda t}(\lambda t)^n}{n!},\quad n\in\N_0,
	\end{equation}
for all $t\geq 0$, where $\lambda=\mathbb E N_1$ is the \name{intensity} of $N$. Moreover, we obtain  that $N$ is continuous in probability and does not have explosions. This yields   that $N$ is a \cadlag process (one may assume without loss that all paths are \cadlag; recall  the discussion in {\S}\ref{Sec1.1.1}).

Moreover, due to the above we obtain that $N=(N_t)_{t\geq 0}$ is a Lévy process with characteristic triplet $(\lambda,0,\lambda \delta_1),$ where $\delta_1$ denotes the Dirac measure concentrated at atom 1. Consistent with Remark \ref{remark:levykhint}, we have presented the Lévy triplet with the canonical truncation function. In many textbooks, e.g., \citeb{p. 32}{book:protter} and \citeb{p. 5}{book:kyp}, one however considers the truncation function $h(x)=x\mathds 1_{\{|x|<1\}}$ instead, yielding the characteristic triplet  $(0,0,\lambda \delta_1)$. We prefer   the canonical one. This is because one should think of $\lambda t$ as the drift term of $N_t$, since 
\begin{equation}
	\big(N_t-\lambda t\big)_{t\geq 0},
\end{equation}
\noindent known as the corresponding \name{compensated Poisson process}{compensated!Poisson process}, is a martingale with respect to the filtration $\mathbb F$. This follows by a simple verification.

Finally, let $(T_n)_{n\in \N}$ be the (unique) sequence of stopping times associated to $N$. In many fields, we say $T_n$ are     \name{arrival times}. Now, set $T_0 =0$, as usual, and let  $\bar T_n:=T_{n}-T_{n-1}$, $n\in\N$ be   the \name{interarrival times} of the Poisson process $N$. It is commonly known that all  $\bar T_n\sim \text{Exp}(\lambda)$ are exponentially distributed with parameter $\lambda$, since the probability of no  arrivals (i.e., jumps) in $(s,s+t]$, given that there was an arrival  at $s$, is given by
\begin{align*}
	\mathbb P\big(\bar T_{n+1}>s+t \mid\bar T_n=s\big)&=\mathbb P\big(N_{t+s}-N_s=0\mid \bar T_n=s\big) \\
	&=\mathbb P\big(N_{t+s}-N_s=0\big) \\
	&=e^{-\lambda t},
\end{align*}
 for any $s,t>0$. The ultimate and penultimate equalities   follow from the facts that  increments are  stationary and independent of the past, respectively. 

The latter observation is useful, since it gives rise to a relatively straightforward numerical scheme; see Appendix \ref{matlab-poisson}. This   scheme is conform to, e.g., \citeb{p. 181}{book:tankov}, \citeb{p. 172}{book:kroese} and \citeb{p. 5}{article:manuge}. For alternatives, we like to refer to  \cite{book:kroese} as well.
%
%
%
%
\end{example}

In addition to Example \ref{ex:poisson}, recall  we always assume that the filtered probability space satisfies the usual conditions.	It should be noted that the existence of such Poisson processes is affirmative;  the natural filtration   of $N$ combined with all the $\mathbb P$-null sets is   right-continuous; see  \citeb{Thm. I.31}{book:protter} (which is actually valid for any Lévy process, hence we will no longer repeat this in the further reading). Recall, actually, that this  has already been dealt with in Theorem \ref{thm:levykhint}.



Remember that we have encountered Poisson processes before, namely, for $X$ any Lévy process with triplet $(b,\sigma^2,\nu)$, the process $N^A=(N^A)_{t\geq 0}$ given by  $N_{t}^{A}=\#\left\{s \in[0, t]: \Delta X_{s} \in A\right\}$ is a Poisson process with intensity $\lambda =\mathbb E N^A_1=\nu(A)<\infty,$ for all Borel measurable sets $A\in\mathcal B_0$  bounded away from zero. It is quite easy to see that $N^A$ inherits every Lévy property of $X$; we refer to \citeb{p. 26}{book:protter} for a somewhat more elaborate discussion.

Now that we have elaborately introduced Poisson processes, we are able to clarify some comments made in  {\S}\ref{Sec1.1} regarding measurability and stochastic integration.  Poisson processes appear to be no predictable processes.  Intuitively speaking this is very convincing, since the jumps of a Poisson process just cannot be “predicted in
advance”. We provide a rigorous proof of the non-predictability below.
\begin{lemma}\label{lem:pred}
	Let $N=(N_t)_{t\geq 0}$ be a Poisson process. Then $N_-=(N_{t-})_{t\geq 0}$ is a predictable process, but the process $N$ itself is not predictable.
\end{lemma}
\begin{proof}
All adapted \caglad processes are predictably measurable, hence $N_-$ is obviously a predictable process.  
Now, recall the definition of a compensator, see Definition \ref{def:compensator}, and the discussion thereafter. We easily deduce from Example \ref{ex:poisson} that $\lambda t$ is the compensator of the finite variation process $N$. By unicity of the compensator, we conclude $N$ cannot be predictable.
\end{proof}

With the present knowledge above, we are finally able to convince the  reader why restricting to integrands in $\mathbb L\Rplus$, i.e., adapted \caglad stochastic processes, is far from innocent. Recall  from {\S}\ref{Sec1.1.4} that, for general $L^2$-integrable martingales, the space of integrands can be expanded to suitable predictable processes. This suggests that a Poisson process $N\in\mathbb D\Rplus$ will not be an appropriate integrand, which Example \ref{ex:generalise} confirms.

Before we continue, remember that continuous $L^2$-martingales are ``perfect'' integrators in the sense that we may expand the class of integrands to progressively measurable processes without to many technicalities. In particular, a Brownian motion perfectly fits into that framework and, consequently, a Poisson process $N\in\mathbb D\Rplus$ would be a suitable integrand.

%
%

\begin{example}\label{ex:generalise} Let $N=(N_t)_{t\geq 0}$ be a Poisson process and consider $M=(M_t)_{t\geq 0}$ to be the  compensated Poisson process, thus defined by $M_t=N_t-\lambda t,\,t\geq 0$. Due to the fact that $M$ is a finite variation process, we may  examine the following Lebesgue--Stieltjes integral:
\begin{equation}
	\begin{aligned}
	Y_t&=\int_0^t N_s \,\mathrm dM_s = \int_0^t (N_{s-} +  \Delta N_s)\,\mathrm  dM_s  \\
	&= \int_0^t N_{s-}\,\mathrm dM_s + \sum_{0<s\le t}(\Delta N_s)^2
	= \int_0^t N_{s-}\,\mathrm dM_s + N_t, \quad t\geq 0.\label{eq:calc}
\end{aligned}
\end{equation}
If an integrator is of finite variation, stochastic integrals coincide with the Lebesgue--Stieltjes integrals for adapted \caglad integrands; see Proposition \ref{prop:stieltjes}. The same holds true for predictable integrands; see Proposition \ref{prop:stieltjes2}. To prevent ambiguity, we would initially want   that both notions of integrals   remain  indistinguishable for any other extended class of integrands. 

Simultaneously, however, it is favourable (or actually, it is necessary for useful applications) to keep \textit{all} the preservation properties   in Theorem \ref{thm:preserve}  whenever we would expand  the class of integrands. We see that  $Y$ is a semimartingale and again of finite variation, hence preservation properties (i) and (ii) of Theorem \ref{thm:preserve} remain intact.   Nevertheless, 
due to the fact    $M$ is a local martingale (in fact, a true martingale), we obtain  that  \begin{equation}\int_0^\cdot  N_{s-}\,\mathrm dM_s\end{equation} 
is a local martingale according to Theorem \ref{thm:preserve}. This yields that $Y$ is not a local martingale, because a Poisson process $N$ fails to be a local martingale (since it is an increasing process).

This comes down to the conclusion that the space of integrands cannot be expanded properly to $\mathbb D\Rplus$, in case we endeavour to neither lose the preservation property of local martingales nor the agreement of stochastic  and Lebesgue--Stieltjes integrals for integrators of finite variation. This problem simply does not occur in the continuous  integrator setting, because there do not exist non-trivial (that is, apart from the zero process) continuous (local) martingales that are finite variation processes as well; see Remark \ref{remark:2}.

Accepting this, we are still able to make sense of   $Y=(Y_t)_{t\geq 0}$ as a stochastic integral---thus, preserving all properties of Theorem \ref{thm:preserve}---according to  {\S}\ref{Sec1.1.4}. This is namely due to the fact that the Doléans measure $\mu_M$ is absolutely continuous with respect to $\mathrm ds\times \mathbb P$. Indeed, one can show that $\mu_M=\lambda(\mathrm ds\times \mathbb P)$ holds; see Remark \ref{remark:doleansabs} for a computation in a somewhat more general setting.
In particular, in line with the extension of {\S}\ref{Sec1.1.4}, we have
\begin{equation}
	Y_t=\int_0^t N_s \,\mathrm dM_s =\int_0^t N_{s-} \,\mathrm dM_s,\quad t\geq 0,\label{eq:conclude}
\end{equation}
since   $N_-$ is a predictable version of   $N$, and because by Fubini--Tonelli's theorem we then get  
 \begin{equation}(\mathrm ds\times \mathbb P)(\{(t,\omega)\in\Rplus\times \Omega:N_{t-}(\omega)\neq N_t(\omega)\})=\int_{\Rplus}\mathbb P(N_{s-}\neq N_s)\,\mathrm ds=0.\end{equation}
These two conditions are sufficient to conclude the indistinguishability of the   stochastic processes in equation \eqref{eq:conclude}. For additional information, we refer to \citeb{p. 204}{unpublished:timo}.
\end{example}

\begin{remark} Example \ref{ex:generalise}   shows us that  we cannot  expand  the class  of integrands to the space of adapted \cadlag processes $\mathbb D\Rplus$, or at least, not in the most straightforward sense. The main problem we encountered is losing the local martingale preservation property, which is unfavourable. Other insightful examples which stresses out this loss can be found in \citeb{p. 65}{book:protter} and \citeb{p. 312}{article:ahn}.  
	
	The perspicacious reader might now wonder if the semimartingale preservation property,  as in  part (i) of Theorem \ref{thm:preserve}, can get lost as well if we expand the class beyond predictable integrands. For instance, in the previous example, both integrals in \eqref{eq:calc} and \eqref{eq:conclude} result into semimartingales.
	  In the paper \citeb{p. 312}{article:ahn}, the authors provide  an exhaustive counter-example of an adapted, \cadlag process $H$ and a martingale $M$ such
		that a stochastic integral process   makes sense but is not a semimartingale.
		
		 In retrospect, the preservation property of local martingales is not the only problem that one perceives when trying to make things work on $\mathbb D\Rplus$ as a  class of integrands.
\end{remark}

Another claim we made in {\S}\ref{Sec1.1} is that some given decomposition of a   semimartingale does  not need to be unique, unless---recall Theorem \ref{thm:predictable}---we require from the decomposition that the finite variation process is predictably measurable (whenever possible).  We illustrate   non-uniqueness via the existence of Poisson processes.

\begin{example}\label{ex:protter130}
Suppose $X=X_0+M+A$ is a semimartingale on the filtered probability space $(\Omega,\mathcal F,\mathbb F,\mathbb P)$. In case  $(\Omega,\mathcal F,\mathbb F,\mathbb P)$ admits a Poisson process, say $N=(N_t)_{t\geq 0}$, observe that
	\begin{equation}
		X_t=X_0+(M_t+N_t-\lambda t)+(A_t-N_t+\lambda t),\quad t\geq 0,\label{eq:canonical}
	\end{equation}
yields another decomposition of $X$.

Hence, if $X$ were a special martingale with canonical decomposition $X=X_0+A+M$, we see that \eqref{eq:canonical} is not a canonical decomposition of $X$, simply due to the fact that a Poisson process is not predictable (recall Lemma \ref{lem:pred}).
\end{example}

Ultimately, after having deviated a little, we are ready to introduce the last fundamental class of Lévy processes. These are the so-called compound Poisson processes.

%
%
%
%
%
%
%
%
%
%
%







%
%
%

\begin{figure}[!t]
	\begin{center}
		\includegraphics[width=0.675\textwidth]{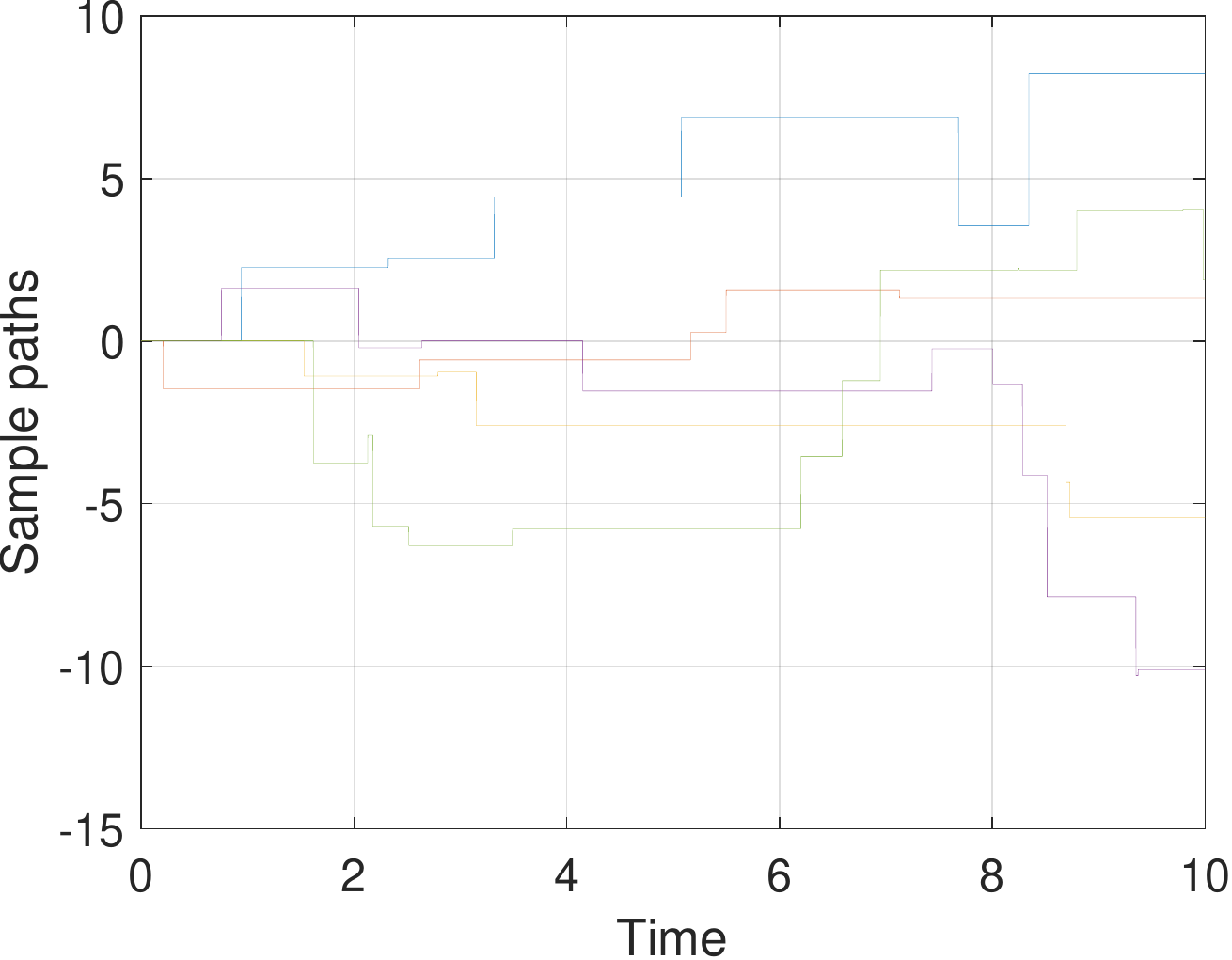}
	\end{center}
	\caption{Five sample path realisations of  the process $C=(C_t)_{t\geq 0}$, a compound Poisson process  given  by $C_t=\sum_{k=1}^{N_t}Z_k$ with $N=(N_t)_{t\geq 0}$ a Poisson process with intensity $\lambda=1$ and  $Z_1\sim \mathcal N(0,\sigma^2)$, $\sigma=2$.}  
	\label{fig:compound}  
\end{figure}

\begin{example}A \name{compound Poisson process} on the filtered probability\label{ex:compound}  space $(\Omega,\mathcal F,\mathbb F,\mathbb P)$ is an adapted stochastic process $X=(X_t)_{t\geq 0}$  of the form  \begin{equation}X_{t}=\sum_{k=1}^{N_{t}} Z_{k}, \quad  t\geq 0,\end{equation} 
	where $N=(N_t)_{t\geq 0}$ is a Poisson process of intensity $\lambda$ and $(Z_{k})_{k \geqslant 1}$ is an i.i.d. sequence  of real-valued random variables  with common law $\mu,$ all independent of $N$. We   additionally assume $\mu(\{0\})=0$ per definition. Note,  for $\mu=\delta_1$ we have that $X=N$ is an ordinary Poisson process. 
	We will call $\mu$ the \name{jump measure} of $X$ and $\lambda$ the \name{intensity} of $X$ (for reasons that will become apparent later).
	
	The characteristic function of $X_t$ can be found  by conditioning on $N$. Indeed, by the law of total expectation, we obtain
	\begin{align*}\phi_t(u)&=\mathbb E\left[e^{iuX_t}\right]= \mathbb{E}\Big[\mathbb{E}\left[ e^{iu X_t}\mid N_t\right]\Big]=\mathbb{E}\left[\left(\mathbb{E}\left[ e^{iu Z_1}\right]\right)^{N_t}\right]\\&=\sum_{n=0}^{\infty}e^{-\lambda t}\frac{(\lambda t)^n}{n!}\left(\mathbb{E}\left[e^{iu Z_1}\right]\right)^{n} 
	=e^{-\lambda t} \sum_{n=0}^{\infty}\frac{(\lambda t \mathbb{E}\left[e^{iu Z_1}\right])^n}{n!}=e^{\lambda t\big( \mathbb{E}\left[e^{iu Z_1}\right]-1\big)}.\end{align*}
Notice that we have used all   assumed stochastic independence   in the third equality. Observe, we can write the characteristic function as $\phi_t(u)=e^{t\psi(u)}$, where
\begin{equation}
	\psi(u)=\lambda t\big( \mathbb{E}\left[e^{iu Z_1}\right]-1\big)=\lambda t\int_\R \left(e^{iux}-1\right)\mu(\mathrm dx).
\end{equation}
Rewriting a bit  shows us that $\psi$ is of the form \eqref{eq:repres}, hence by Theorem \ref{thm:levykhint} we can conclude that $X$ is a Lévy process with characteristic triplet
	\begin{equation}\left(\lambda \int_{\{|x| \leqslant 1\}} x \, \mu(\mathrm dx), 0, \lambda \mu\right).\end{equation}
	Actually, it  is neater to show that $X$ is a Lévy process  by definition. In essence, the process $X$ simply inherits (once again) all Lévy properties of the Poisson process $N$ \citeb{p. 7}{book:kyp}. Moreover, due to the fact $N$ is a \cadlag  finite variation process, we obtain that $X$ is   \cadlag and of finite variation as well. Note that $X$ is of finite intensity, since $\nu(\R)=\lambda \mu(\R)=\lambda <\infty.$

\begin{figure}[!t]
	\begin{center}
		\includegraphics[width=0.675\textwidth]{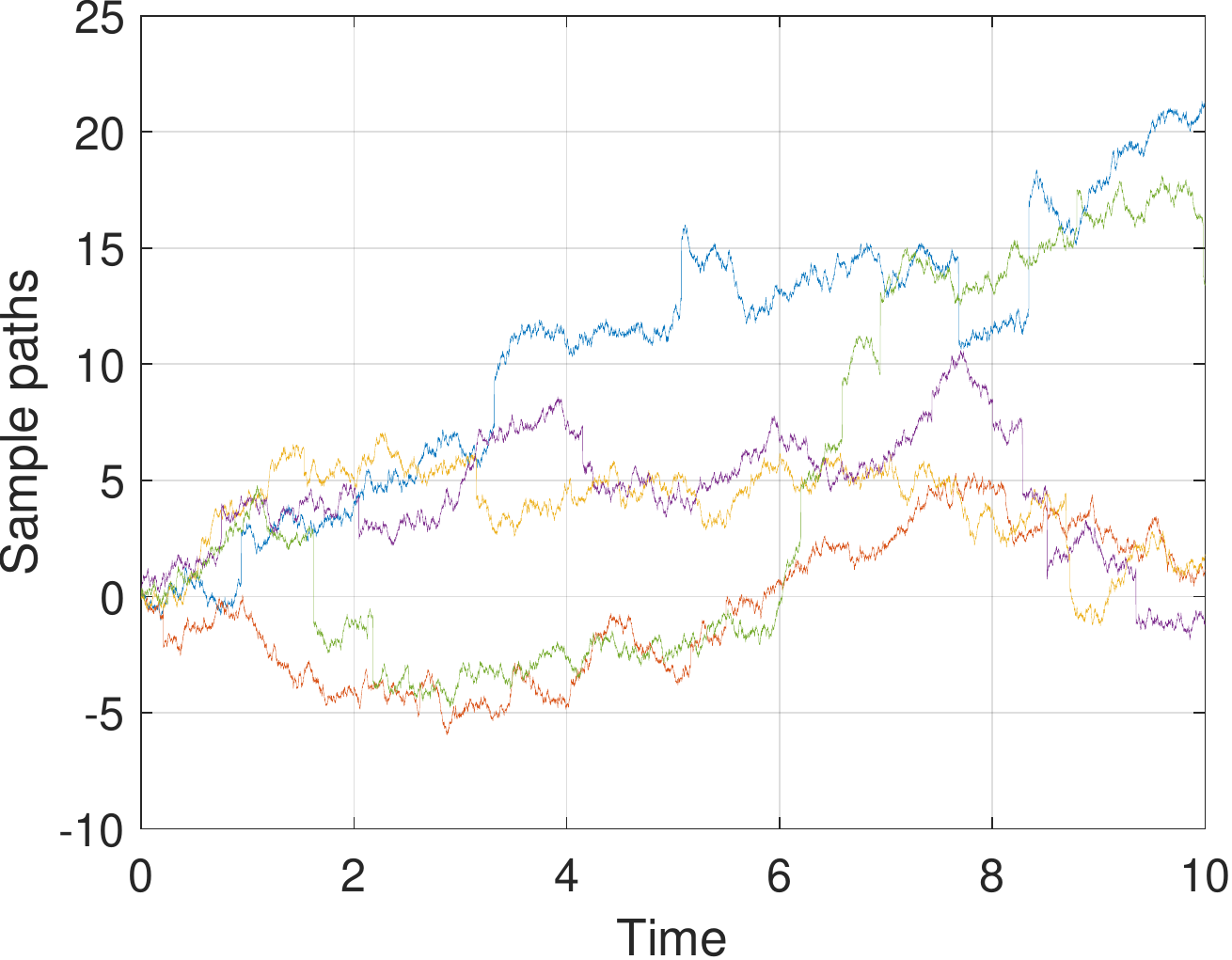}
	\end{center}
	\caption{Five sample path relations of a linear combination of (in fact, the sum of)  the processes $L$ and $C$ from Figures \ref{fig:drift} and \ref{fig:compound} respectively.}
	 
\end{figure}

	Note that we have assumed nothing about the integrability of $Z_1$. It is worth pointing out that whenever $Z_1$ has a finite $k$-th moment, the process $X$ also has a finite $k$-th moment.  In particular, if we presume $\mathbb E|Z_1|<\infty $ and $\mathbb EZ_1^2<\infty$, then we obtain
	\begin{equation}
		\mathbb{E}\left[X_{t}\right] =-i\frac{\partial}{\partial u} \mathbb{E}\left[e^{iu X_{t}}\right]\Bigr|_{u=0} \\
		=\lambda t \int_{\R} x\, \mu(\mathrm d x)=\lambda t \mathbb{E}\left[Z_1\right]=\mathbb E[{N_t}]\mathbb E [Z_1],
	\end{equation}
and 
\begin{equation}
	\mathbb{E} [X_{t}^{2} ] =-\frac{\partial^{2}}{\partial u^{2}} \mathbb{E}\left[\mathrm{e}^{iu X_{t}}\right]\Bigr|_{\alpha=0} \\
	= \mathbb E[{N_t}]\mathbb{E}[Z^{2}_1]+(\mathbb E[{N_t}] \mathbb{E}[Z_1])^{2}.
\end{equation}
		Ultimately, if $\mathbb E|Z_1|<\infty $ holds,  we define the  \name{compensated compound Poisson process}{compensated!compound Poisson process}   by
	\begin{equation}
	M_{t}:=X_{t}-\lambda t \mathbb{E}[Z_1], \quad t \geq 0,
	\end{equation}
	which is a martingale with respect to $\mathbb F$. For any $t\geq s\geq 0$, we have
	\begin{equation}
		\mathbb E[M_t-M_s\mid \mathcal F_s]=\mathbb E[X_t-X_s\mid \mathcal F_s]-\lambda(t-s)\mathbb E[Z_1]=\mathbb E[X_t-X_s]-\lambda(t-s)\mathbb E[Z_1]=0,
	\end{equation}
since $X$ is a Lévy process and, therefore, has increments independent of the past.

This gives us  the following immediate corollary: a compound Poisson process $X$ itself is a martingale if and only if    the  process is symmetric, i.e.,  $\mathbb E Z_1=0$. Hence, whenever we satisfy \begin{equation}\mathbb EZ_1=0\quad\text{and}\quad\mathbb E Z_1^2<\infty,\end{equation} the compound Poisson process  $X=(X_t)_{t\geq 0}$ is an $L^2$-martingale. 

The numerical scheme in Appendix \ref{matlab-compound} for compound Poisson processes extends   \ref{matlab-poisson} in a straightforward manner; see also \citeb{p. 181}{book:tankov}, \citeb{p. 175}{book:kroese} and \citeb{p. 6}{article:manuge}. 
%
\end{example}

We   know from Example \ref{ex:compound} that any compound Poisson process is a Lévy process whose sample paths are piecewise constant functions. It   turns out that they are actually the only possible Lévy processes with piecewise constant sample paths. 
\begin{proposition}[Proposition 3.3 of \cite{book:tankov}]\label{prop:tankov}
	A stochastic   process $(X_t)_{t\geq 0}$ is a compound Poisson process if and only if it is a Lévy process and its sample paths are piecewise constant functions.
\end{proposition}

Importantly, note that a linear combination of \textit{independent} Lévy processes is again a Lévy process. This, for instance, implies that $X=(X_t)_{t\geq 0}$, given by
\begin{equation}
	X_t=b t + \sigma W_t+\sum_{k=1}^{N_{t}} Z_{k},\quad t\geq 0, \label{eq:jumpdif}
\end{equation}
is again a Lévy process  (presuming that $W$ is independent of $N$ and $(Z_k)_{k\geq 0}$).  Stochastic processes of the form \eqref{eq:jumpdif} are widely known as \name{Lévy jump  diffusion processes}{Lévy process!jump diffusions}.
One may wonder   whether these linear combinations are all possible Lévy processes. The answer is obviously no, because taking linear combinations result into processes of finite intensity only, i.e., $\nu(\R)<\infty$; recall Proposition \ref{prop:intensity}. There are many interesting examples of Lévy processes with infinite activity, such as: {Gamma processes} \citeb{p. 33}{book:protter}, {$\alpha$-stable processes} \citeb{p. 51}{book:applebaum}, and numerous others \cite{article:manuge}. An extensive  review of  examples of L\'evy processes is also available in \citeb{p. 12}{book:kyp}. 

Nevertheless, as mentioned at the beginning of this section, each Lévy process can be decomposed into terms   which we are familiar with now.


\begin{theorem}\label{thm:LevyIto} \namethrm{Lévy--Itô decomposition}
 Given an arbitrary Lévy process $X=(X_t)_{t\geq 0}$ on   $(\Omega,\mathcal F,\mathbb F,\mathbb P)$ with characteristic triplet $(b,\sigma^2,\nu)$,   there exist three independent Lévy processes living on  the same filtered probability space, say $X^{(1)},$ $X^{(2)}$ and $X^{(3)}$, such that
 \begin{equation}X=X^{(1)}+X^{(2)}+X^{(3)}\quad \mathbb P\textnormal{-a.s.},\end{equation}
  where
 \begin{itemize}
\item  $ {X}^{(1)}$ is a Brownian motion with drift, i.e.,
 \begin{equation}
 	X_{t}^{(1)}=b t+\sigma W_{t}, \quad t \geqslant 0;
 \end{equation}
\item  $X^{(2)}$ is a square integrable martingale with a  \textnormal{(}$\mathbb P$-almost surely\textnormal{)} countable number of jumps on each
finite time interval, which are of magnitude less than one, and with characteristic exponent
 \begin{equation}
 	\psi^{(2)}(u)=\int_{\mathbb{R}}\left(e^{i u x}-1-i u x\right) \mathbf{1}_{\{|x| \leqslant 1\}} \mathrm{d} \nu(x);
 \end{equation}
\item  ${X}^{(3)}$ is a compound Poisson process, i.e.,
 \begin{equation}
 	X_{t}^{(3)}=\sum_{k=1}^{N_{t}} Z_{k}, \quad t \geqslant 0,
 \end{equation}
 with $N=\left(N_{t}\right)_{t \geqslant 0}$   a Poisson process  with intensity $\lambda:=\nu(\mathbb{R} \backslash[-1,1])$, independent of the i.i.d. sequence $ (Z_{k} )_{k \geqslant 1}$ with a distribution concentrated on the set $\{x\in\R: |x|>1\}$ and given by   $ \frac1\lambda   \nu|_{\{x\in\R: |x|>1\}}$ \textnormal{(}unless $\lambda=0$ holds, because then  $X^{(3)}$ is identically zero\textnormal{)}.
 \end{itemize}
%
\end{theorem}
\begin{proof}
One may consult many textbooks and other   literature for a proof. For instance, consider  \citeb{Ch. 4}{book:sato}, \citeb{Sec. 2.4}{book:applebaum},   \citeb{Sec. I.4}{book:protter},  and   \citeb{Thm. 2.1}{book:kyp}. We would like to point out that the comments in the proof of Theorem \ref{thm:levykhint}  can also be quite useful.
\end{proof}
\begin{corollary}\label{cor:semi!}
	Every Lévy process $X$ is a semimartingale. Decompose $X^{(1)} $ into \begin{equation}A^{(1)}_t=bt\quad\text{and}\quad M^{(1)}_t=\sigma W_t , \quad t\geq 0,\end{equation} 
	then $ A^{(1)} +X^{(3)} $ is the finite variation part and $M^{(1)} +X^{(2)} $ the martingale part of $X$.
\end{corollary}

More technically, the Lévy--Itô decomposition tells us that $X=(X_t)_{t\geq 0}$ is a Lévy process with characteristic triplet $ (b, \sigma^{2}, \nu )$ if and only if it can be written as the sum of the following three independent Lévy
processes:
\begin{equation}
	X_{t}=\Big(b t+\sigma W_{t}\Big)+\lim _{\eta \rightarrow 0}\left(\sum_{s \leqslant t} \Delta X_{s} \mathds{1}_{\{\eta<\left|\Delta X_{s}\right| \leqslant 1\}}-t \int_{\{\eta<|x| \leqslant 1\}} x \, \nu(\mathrm{d}x)\right)+\sum_{s \leqslant t} \Delta X_{s} \mathds{1}_{\{\left|\Delta X_{s}\right|>1\}}.\label{eq:decomposition}
\end{equation}
The limit is in the $L^2$-sense and converges to a square integrable martingale. Moreover, notice that both sums in \eqref{eq:decomposition} are over a countable set and even $\mathbb P$-a.s. finite, because the jumps live in measurable sets bounded away from zero (recall {\S}\ref{Sec1.2.1}).

The Lévy--It\^o decomposition separates the large jumps $(\sum_{s \leqslant t} \Delta X_{s} \mathds{1}_{\{\left|\Delta X_{s}\right|>1\}})_{t \geqslant 0}$ from the small jumps, since the (countable) infinite sum
\begin{equation}
	\sum_{s \leqslant t} \Delta X_{s} \mathds{1}_{\{\left|\Delta X_{s}\right|> 0\}},\quad t\geq 0, \label{eq:sprongen}
\end{equation}
is not always $\mathbb P$-almost surely finite. Of course, the separation does not necessarily need to happen about the point 1. Any separation point suffices (consider a different truncation function). 

In more detail, the finite variation   of the process in equation \eqref{eq:sprongen}   equals
\begin{equation}
\left|\sum_{s \leqslant t} \Delta X_{s} \mathds{1}_{\{\left|\Delta X_{s}\right|> 0\}}\right|=	\sum_{s \leqslant t} |\Delta X_{s} |\mathds{1}_{\{\left|\Delta X_{s}\right|> 0\}},\quad t\geq 0,
\end{equation}
which can be shown to be $\mathbb P$-a.s. finite if and only if $\int_{[-1,1]}|x|  \,\nu(\mathrm{d}x)<\infty $ holds (in essence, this follows ``immediately'' from the  Lévy--Itô decomposition; see \citeb{p. 129}{book:applebaum}, \citeb{p. 141}{book:sato}, or \citeb{p. 56}{book:kyp} for instance). We summarise these observations in the next proposition. In particular, recall that Brownian motion is not of finite variation (see Example \ref{ex:Brownian}).
%
%
\begin{proposition} \label{prop:fact1} A Lévy process $X$ with characteristic triplet $(b, \sigma^2, \nu)$ is a finite variation
	process if and only if 
	\begin{equation}
	\sigma^2=0\quad\text{and}\quad	\int_\R (|x|\wedge 1) \,\nu(\mathrm dx)<\infty.
	\end{equation}
In any case, the decomposition as in \eqref{eq:decomposition} can be written as
	\begin{equation}
	X_{t}=\left(b- \int_{\{|x|\leq 1\}} x \, \nu(\mathrm{d}x)\right) t +  \sum_{s \leqslant t} \Delta X_{s} \mathds{1}_{\{\left|\Delta X_{s}\right|> 0\}},\quad t\geq 0. 
\end{equation}
\end{proposition}

In  a similar fashion, we obtain the following result.

\begin{proposition}[Proposition II.2.29 of \cite{book:jacod}]\label{cor:finitefirst}  
	Any Lévy process $X$ with characteristic triplet $(b, \sigma^2, \nu)$ is a special semimartingale if and only if
	\begin{equation}
		\int_\R (|x|\wedge x^2) \,\nu(\mathrm dx)<\infty.\label{eq:finite}
	\end{equation}
The   canonical decomposition is then given by $X=A+M$, with predictable finite variation  
	\begin{equation}
	A_{t}=\left(b+ \int_{\{|x|> 1\}} x \, \nu(\mathrm{d}x)\right) t,\quad t\geq 0,\label{eq:old}
\end{equation}
and (local) martingale part 
\begin{equation}
	M_t= \sigma W_t+\lim _{\eta \rightarrow 0}\left(\sum_{s \leqslant t} \Delta X_{s} \mathds{1}_{\{\left|\Delta X_{s}\right| >\eta\}}-t \int_{\{ |x| > \eta\}} x \, \nu(\mathrm{d}x)\right),\quad t\geq 0. 
\end{equation}
\end{proposition}

See also  \citeb{p. 518}{book:kallenberg} or \citeb{p. 133}{book:protter} for other sources of reference.  Moreover, note that     \eqref{eq:finite} holds if and only if $X_t$,   for all $t\geq 0$, has finite first moment; see Proposition \ref{thm:expec}

Let us now consider  the finite intensity case. The propositions above then tell us  that any such Lévy process  is   of finite variation, but not necessarily a special semimartingale.
Recall the linear combination in equation \eqref{eq:jumpdif}; the Lévy jump diffusion processes. Clearly, these stochastic processes have finite activity. The Lévy--Itô decomposition in Theorem \ref{thm:LevyIto} yields that the converse  is also true.
\begin{proposition}
	Let $X=(X_t)_{t\geq 0}$ be\label{prop:thecompound} a Lévy process with characteristic triplet $(b,\sigma^2,\nu)$. Whenever $X$ is  of finite intensity, i.e., $\nu(\R)<\infty$,   the process   can be written $\mathbb P$-a.s. as
	\begin{equation}
		X_t=\tilde b t + \sigma W_t+\sum_{k=1}^{N_t}Z_k,\quad t\geq 0.\label{eq:finin}
	\end{equation}
  That is, as  sum of a Brownian motion with drift and a compound Poisson process, independent from one  another.
 Specifically,  $\tilde b:=b-\int_{\{|x|\leq 1\}} x\,\nu(\mathrm dx)$ and $\sum_{k=1}^{N_t}Z_k$ has   jump measure $\frac1{\nu(\R)}  \nu$.
\end{proposition}
\begin{proof} A very similar statement  can be found in  \citeb{p. 57}{book:kyp}. Its proof follows  immediately from   \citeb{Lem. 2.8}{book:kyp}; an alternative proof---in fact, a much more comprehensible proof in the author's opinion---of that particular lemma can be found in  \citeb{Thm. 2.3.9}{book:applebaum}.  Nevertheless, we will  achieve the result above via a  different (but then again, not really that different) approach.
	
	Since $X$ is of finite intensity, we obviously have that  $X$ is of finite variation. Consequently, by Proposition \ref{prop:fact1}  we obtain that the decomposition as in equation \eqref{eq:decomposition} simplifies into
	\begin{equation}
		X_{t}=\left(b- \int_{\{|x|\leq 1\}} x \, \nu(\mathrm{d}x)\right) t+\sigma W_{t}+   \sum_{s \leqslant t} \Delta X_{s} \mathds{1}_{\{\left|\Delta X_{s}\right|> 0\}},\quad t\geq 0. 
	\end{equation}
Observe that the part of $X$ describing all its jumps, namely
\begin{equation}
	\sum_{s \leqslant t} \Delta X_{s} \mathds{1}_{\{\Delta X_{s} \neq 0\}},\quad t\geq 0,\label{eq:compound}
\end{equation}
is a Lévy process. By  the Lévy--Khintchine representation in Theorem \ref{thm:levykhint} we know that the latter equals in distribution to a compound Poisson process with triplet $(\int_{\{|x|\leq 1\}} x \, \nu(\mathrm{d}x),0,\nu)$.

We are left with showing that the process in equation \eqref{eq:compound} must be a compound Poisson process (thus, not only   equality in distribution).
Thanks to Proposition \ref{prop:intensity} part (i), we can conclude that  the process in \eqref{eq:compound} is a Lévy process whose sample paths are piecewise constant functions---or at least, on a (proper) subset of the  sample space with probability one. In conclusion, this process is indeed a compound Poisson process according to Proposition \ref{prop:tankov}. The assertion now follows.
\end{proof}

We can conclude from the previous result in combination with Proposition \ref{cor:finitefirst} that any Lévy process with finite intensity is a special semimartingale if and only if $\mathbb E|Z_1|<\infty$. Notice, however,   the decomposition as  in \eqref{eq:finin} is not the canonical decomposition (i.e., the decomposition with a predictable finite variation process) because a compound Poisson process is not predictable (recall Lemma \ref{lem:pred}). The canonical decomposition  reads
\begin{equation}
	X_t={\hat b} t + \sigma W_t+\left[\sum_{k=1}^{N_t}Z_k-\lambda t \mathbb EZ_1 \right],\quad t\geq 0,\label{eq:finin2}
\end{equation}
where  $ {\hat b}:=b+\int_{\{|x|>1\}} x\,\nu(\mathrm dx)$. In particular, we have $\lambda \mathbb  EZ_1=\int_{\{|x|>0\}}x\,\nu(\mathrm dx)$.


Let us proceed by providing some extra details to gain more insight into the Lévy--Itô decomposition; see Theorem \ref{thm:LevyIto}. We   actually show how the concepts and results are  ``usually''   introduced in the literature.

\begin{definition}
	 A \name{(transition) kernel}{kernel (transition)}\index{transition kernel} from the measurable space  $(T, \mathcal{T})$ into the measurable space $(S, \mathcal{S})$ is a function $\mu: T \times \mathcal{S} \rightarrow [0,\infty]$ such that \begin{enumerate}[\normalfont(i)]
	 	\item the mapping $t\mapsto \mu(t, B)$ is $\mathcal{T}$-measurable  for any fixed $B \in \mathcal{S}$;
	 	\item the mapping $B\mapsto \mu(t, B)$ is a measure for any fixed $t \in T$.
	 \end{enumerate}
\end{definition}

\begin{definition} 
		 A kernel  on the   probability space $(\Omega,\mathcal F,\mathbb P)$ into a measurable space $(S,\mathcal S)$, that is $\mu:\Omega\times \mathcal S\to[0,\infty]$, is called a \name{random measure}.
\end{definition}

This definition of a random measure can be found in, e.g., \citeb{p. 106}{book:kallenberg}. What is meant by a random measure  differs per author, though; see  \citeb{p. 103}{book:applebaum} and \citeb{p. 65}{book:jacod} for other non-equivalent definitions. Despite this ambiguity, the notion of a Poisson random measure is more likely universal.
\begin{definition}
Let $(\Omega,\mathcal F,\mathbb P)$ be a probability space and $(S,\mathcal S,\nu)$	 a $\sigma$-finite measure space. The function $\mu:\Omega \times \mathcal S\to[0,\infty]$ is  called a \name{Poisson random measure} with \name{intensity measure} $\nu$ when
	\begin{enumerate}[\normalfont(i)]
		\item $\mu$ is a transition kernel;
\item for any pairwise disjoint sets $B_{1}, \ldots, B_{n} \in \mathcal{S}$, the random variables $\mu\left(\,\cdot\,,B_{1}\right), ..., \mu\left(\,\cdot\,,B_{n}\right)$ are independent;
	\item  if $\nu(B)<\infty$, then $\mu(\cdot, B)$ has a Poisson distribution with parameter $\nu(B)$.
	\end{enumerate}
\end{definition}
Obviously, a Poisson random measure is a random measure. For general theory on Poisson random measures, we refer to \citeb{p. 119}{book:sato} and \citeb{p. 40}{book:kyp}. In both references, one has the additional assumption that $\mu$ is integer-valued (including infinity). 

 The following example and theorem demonstrate  a typical application of (Poisson) random measures.
\begin{example}
	Let $X=(X_t)_{t\geq 0}$ be an adapted $\R^d$-valued \cadlag process. Then, according to Remark \ref{remark:measurability},  the set of jumps  $\left\{t \geqslant0: \Delta X_{t} \neq 0\right\}$ can be described by a countable sequences of stopping times $(T_{n})_{n\in\N}$ (in general, the $T_n$ cannot be in increasing order). Observe,
\begin{equation}
	Y_{n}:=\Delta X_{T_{n}}=X_{T_{n}}-X_{T_{n-}} \in \mathbb{R}^{d} \backslash\{0\}, \quad n \in\N,
	\end{equation}
	are $\mathcal{F}_{T_{n}}$-measurable. The sequence $(T_{n}, Y_{n})_{n \in\N}$  completely describes the jump nature of $X$.
	
	Subsequently, we define the \name{jump random measure} by the random measure $\mu_X$ associated with the jumps of $X$. That is,
 	\begin{equation}
		\mu_X(\omega,A)=\#\left\{\left(t, \Delta X_{t}(\omega)\right) \in A\right\}=\sum_{t \geq 0} \delta_{\left(t, \Delta X_{t}(\omega)\right)}(A)=\sum_{n \in\N} \delta_{\left(T_{n}(\omega), Y_{n}(\omega)\right)}(A), 
	\end{equation}
	for all $\omega\in\Omega$ and $A\in\mathcal B(\Rplus)\times \mathcal B(\R^d\backslash \{0\})$. 
	
	Indeed, the jump random measure $\mu_X$ is a random measure into $\Rplus\times \R^d$. Completely conform to Remark \ref{remark:clarify}, we see that the set function $A\mapsto \mu_X(\omega,A)$ defines a counting measure for each $\omega\in\Omega$ fixed (presuming, without loss of generality, that $X$ is \cadlag everywhere). Moreover, we have $\mu_X(\,\cdot\,,[0,t]\times B)=N_t^B$ for any $B\in  \mathcal B(\R^d\backslash \{0\})$, which are measurable functions thanks to the observations in Remark \ref{remark:measurability}. 	By another approximation argument, we obtain that $\omega\mapsto \mu_X(\omega,A)$ are random variables for all $A\in\mathcal B(\Rplus)\times \mathcal B(\R^d)$, hence we conclude $\mu_X$ is indeed a well-defined random measure.
	%
%
\end{example}

In case $X$ is a Lévy process, the example above can be continued with the following result, which essentially follows from $N^B$ being  a Poisson process for all $B\in\mathcal B(\R^d)$, $0\notin \bar B$.

\begin{theorem}[Theorem 19.2 of \cite{book:sato}]
	If $X$ is a Lévy process with   triplet $(b,\Sigma,\nu),$ then    $\mu_X$ is a Poisson random measure with intensity measure $\nu$. 
\end{theorem}
As usual, we will often suppress the $\omega$ from the notation of a random measure $\mu$. By means of the random jump measure $\mu_X$, we can state the Lévy--Itô decomposition in its concise and  traditional form (with $d=1$, conform Theorem \ref{thm:LevyIto}, but there is no difference for $d>1$).

\begin{theorem}
Let $X$ be a Lévy process with characteristic triplet $(b,\sigma^2,\nu)$. The stochastic process $X$ then admits a decomposition\label{thm:levitocompact} 
\begin{equation}\begin{aligned}
			X_t&=b t+\sigma W_t+ \int_{[0,t]\times \R}x\mathds 1_{\{|x|\leq 1\}}\big(\mu_{X}(\mathrm ds,\mathrm dx)-t\nu (\mathrm dx)\big)\\
			&\hspace{5cm} +\int_{[0,t]\times \R}x\mathds 1_{\{|x|>1\}}\mu_{X}(\mathrm ds,\mathrm dx),\quad t\geq 0, \label{eq:compare}
\end{aligned}
\end{equation}
where $W=(W_t)_{t\geq 0}$ is a standard Brownian motion.
\end{theorem}
 Note that the decomposition in \eqref{eq:decomposition} is a translation of \eqref{eq:compare}. In particular, using the  shorthand notation $\mu_X(t,\,\mathrm dx)=\mu_X([0,t],\mathrm dx)$, we have that 
 \begin{equation}
 	\int_{A}f(x)\mu_X(t,\mathrm dx)=\sum_{0<s\leq t}f(\Delta X_s)\mathds 1_{\{\Delta X_s\in A\}},
 \end{equation}
is a compound Poisson process for every measurable function $f:\R\to\R$ and    Borel set $A\in\mathcal B(\R)$ bounded away from zero    \citeb{p. 108}{book:applebaum}.
 Also,  the product measure of the Lévy measure and Lebesgue measure (on $\Rplus$), sometimes denoted by
\begin{equation}\nu_X(\mathrm d t,\mathrm d x)= \nu(\mathrm d x)\,\mathrm d t,\end{equation}
is a random measure as well. It is
 known as the \name{compensator} of the random jump measure $\mu_X$ for a Lévy process $X$.  We refer to \citeb{p. 66}{book:jacod} and \citeb{p. 76}{book:jacod} for a proper definition suitable for all semimartingales. Lastly, the \name{compensated jump random measure}{compensated!jump random measure} is defined by
 \begin{equation}\tilde \mu_X(\mathrm d t, \mathrm d x):=\mu_X(\mathrm d t, \mathrm d x)-\nu(\mathrm d x)\, \mathrm d t.\end{equation} The stochastic process $(\tilde \mu_X(t,A))_{t\geq 0}$---in line with the previous shorthand notation---is a martingale for any Lévy process $X$ and   all $A\in\mathcal B(\R^d)$ bounded away from zero \citeb{p. 105}{book:applebaum}.
\begin{remark} The Lévy--Itô decomposition can be generalised for all semimartingales. This result can be found in  \citeb{Thm. II.2.34}{book:jacod}.
\end{remark}

Finally, observe that in case a Lévy process $X$ is of finite variation (see Proposition \ref{prop:fact1}), we can write
\begin{equation}\label{eq:compare2}
	X_t=\left(b- \int_{\{|x|\leq 1\}} x \, \nu(\mathrm{d}x)\right) t  +\int_{[0,t]\times \R}x \mu_{X}(\mathrm ds,\mathrm dx),\quad t\geq 0.
\end{equation}
Likewise, if a Lévy process $X$ is a special semimartingale  (see Proposition \ref{cor:finitefirst}), the canonical decomposition of $X=A+M$ is given by $A$ as in \eqref{eq:old} and
\begin{equation}
	M_t=\sigma W_t+\int_{[0,t]\times \R}x  \big(\mu_{X}(\mathrm ds,\mathrm dx)-t\nu (\mathrm dx)\big),\quad t\geq 0.
\end{equation}

Another important feature of the jump random measure, is that it enables us to extend the idea of  Itô processes. In line with \citeb{p. 5}{book:sulem}, we can write equation \eqref{eq:compare} as
\begin{equation} 	X_t =b t+\sigma W_t+ \int_{[0,t]\times \R}x  \bar \mu_X(\mathrm d s, \mathrm d x)\label{eq:shorth}\end{equation}
where  
\begin{equation}
		\bar \mu_X(\mathrm d t, \mathrm d x)=\begin{cases}
			\mu_X(\mathrm d t, \mathrm d x)-\nu(\mathrm d x) \,\mathrm d t, & \text { for }   |x|\leq 1, \\
			\mu_X(\mathrm d t, \mathrm d x), & \text { for }   |x| >1.
		\end{cases}
	\end{equation}
This is to be understood as a compact way of writing \eqref{eq:compare}.

Conversely, let $W=(W_t)_{t\geq 0}$ be a standard Brownian  and suppose $\mu$  is an independent Poisson  random measure with intensity measure $\nu$; we take $\nu$ to   be a Lévy measure. Then along the same references as above, see also \citeb{p. 46}{book:kyp} and \citeb{p. 53}{book:kyp}, we deduce that
\begin{equation} 	t\mapsto b t+\sigma W_t+ \int_0^t\int_{  \R}x  \bar \mu(\mathrm d s, \mathrm d x)\label{eq:shorth2},\end{equation}
where
\begin{equation}
	\bar \mu(\mathrm d t, \mathrm d x)=\begin{cases}
		\mu(\mathrm d t, \mathrm d x)-\nu(\mathrm d x) \,\mathrm d t, & \text { for }   |x|\leq 1, \\
		\mu(\mathrm d t, \mathrm d x), & \text { for }   |x| >1,
	\end{cases}
\end{equation}
defines a Lévy process (with respect to its natural filtration). This inherently suggests  the more general stochastic integrals of the form
\begin{equation}
X_t=X_0+\int_{0}^{t} \alpha(s, \omega)\,\mathrm d s+\int_{0}^{t} \beta(s, \omega)\,\mathrm d W_s+\int_{0}^{t} \int_{\mathbb{R}} \gamma(s, x, \omega)\, \bar{\mu}(\omega;\mathrm d s,\mathrm  d x),\label{eq:notindiff}
\end{equation}
with $X_0\in\mathcal F_0$.
For $\alpha(s,\omega)=b$, $\beta(s,\omega)=\sigma$ and $\gamma(s,x,\omega)=x$, we get equation \eqref{eq:shorth2} back. If we set the Lévy measure to zero  ($\nu=0$), then $\bar \mu=0$ and we obtain the famous Itô processes.

These stochastic integrals in \eqref{eq:notindiff} are referred to as \name{Itô--Lévy processes} in \citeb{p. 5}{book:sulem}, where the integrands are satisfying   appropriate conditions in order to make the integrals   well-defined. We observe that  Itô--Lévy processes are a subclass of the so-called \name{Lévy-type stochastic integrals}; see \citeb{p. 233}{book:applebaum}. In there, the collection of appropriate integrands is discussed extensively. We are not going into more detail, but it has to be noted that  the largest class  of suitable integrands will  be given by some sort of predictable measurability.

 In differential form, equation \eqref{eq:notindiff} would read as
 \begin{equation}
\mathrm d X_t=\alpha(t )\, \mathrm d t+\beta(t)\, \mathrm d W_t+\int_{\mathbb{R}} \gamma(t, x) \bar{\mu}(\mathrm d t, \mathrm d x).
 \end{equation}
In line with  Corollary \ref{cor:semi!}, we  observe that Itô--Lévy processes are semimartingales (see also the discussion in \citeb{p. 234}{book:applebaum}).
This enables us to invoke all    theory  in  {\S}\ref{Sec1.1}. Specifically, we can use Itô's formula in Theorem \ref{thm:ito} and the discussion succeeding it. 
 \begin{theorem}\namethrm{Itô's formula}\label{thm:ito2} Suppose $X=(X_t)_{t\geq 0}$ is a real-valued Itô--Lévy process;
 \begin{equation}
	\mathrm d X_t=\alpha(t )\, \mathrm d t+\beta(t)\, \mathrm d W_t+\int_{\mathbb{R}} \gamma(t, x) \bar{\mu}(\mathrm d t, \mathrm d x).\label{eq:difficult}
\end{equation}
 Let $f:\R^2\to\R$ be  a sufficiently smooth function and define the stochastic process $Y=(Y_t)_{t\geq 0}$ by $Y_t=f( X_t,t),t\geq 0.$ Then $Y$ is again a real-valued Itô--Lévy process, where
\begin{align}  
	\displaystyle \mathrm d Y_t&= \displaystyle\frac{\partial f}{\partial t}(X_t,t)\,\mathrm   d t+\frac{\partial f}{\partial x}(X_t,t)\big[\alpha(t)\,\mathrm  d t+\beta(t) \,\mathrm d W_t\big]  +\frac{1}{2} \beta^{2}(t) \frac{\partial^{2} f}{\partial x^{2}}(X_t,t)\,\mathrm  d t \nonumber\\[.3cm]
	&\label{eq:ito3}\displaystyle\quad+\int_{\{|z|\leq 1\}}\left\{f\left(X_{t-}+\gamma(t, z),t \right)-f\left(X_{t-},t\right) -\frac{\partial f}{\partial x}\left(X_{t-},t\right) \gamma(t, z)\right\} \nu(\mathrm d z)\mathrm d t \\[.3cm]
	&\displaystyle\quad+\int_{\mathbb{R}}\big\{f\left(X_{t-}+\gamma(t, z),t\right)-f\left(X_{t-},t\right)\big\} \bar{\mu}(\mathrm d t, \mathrm d z).\nonumber
\end{align}
 \end{theorem}
\begin{proof}
	This exact statement can be found in   \citeb{Thm. 1.14}{book:sulem}. There they refer to several works, but we point out that  \citeb{Sec. 4.4}{book:applebaum} is (one of) the appropriate source(s). With a bit of work, it can also be seen as a corollary of Theorem \ref{thm:ito}. 
\end{proof}
 In case $\nu(\R)<\infty$ holds, we note that the Itô--Lévy procces in  \eqref{eq:ito3} can be rewritten as 
  \begin{equation}
 	\mathrm d X_t=\tilde \alpha(t )\, \mathrm d t+\beta(t)\, \mathrm d W_t+\int_{\mathbb{R}} \gamma(t, x)  {\mu}(\mathrm d t, \mathrm d x)\label{eq:simplified},
 \end{equation}
whence the Itô formula in \eqref{eq:ito3}   then reads
 \begin{equation}
 	\begin{aligned}  
 	\displaystyle \mathrm d Y_t&= \displaystyle\frac{\partial f}{\partial t}(X_t,t)\,\mathrm   d t+\frac{\partial f}{\partial x}(X_t,t)\big[\tilde \alpha(t)\,\mathrm  d t+\beta(t) \,\mathrm d W_t\big]  +\frac{1}{2} \beta^{2}(t) \frac{\partial^{2} f}{\partial x^{2}}(X_t,t)\,\mathrm  d t  \\[.3cm]
 	&\displaystyle\quad+\int_{\mathbb{R}}\big\{f\left(X_{t-}+\gamma(t, z),t\right)-f\left(X_{t-},t\right)\big\}  {\mu}(\mathrm d t, \mathrm d z). \label{eq:ito4}
 \end{aligned}
 \end{equation}
In here, we have set $\tilde \alpha(t):=\alpha(t)-\int_{\{|x|\leq 1\}}\gamma (t,x)\,\nu(\mathrm dx),t\geq 0$ (as usual; compare with the result in equation \eqref{eq:compare2}.)
 Observe   this is   in accordance with  the discussion in   \eqref{eq:short1}--\eqref{eq:short3} in {\S}\ref{Sec1.1.5}. See also \citeb{p. 249}{book:applebaum} (and \citeb{p. 383}{book:applebaum}), where we  claim that our finite intensity result above follows from the latter.
%
%
Making use of Theorem \ref{thm:ito2} instead of Theorem \ref{thm:import_form} turns out favourable  with regard  to the following remark.
\begin{remark}
	Itô's formula as in Theorem \ref{thm:ito2} significantly tells us that any   function (sufficiently smooth) of an Itô--Lévy process is not only a semimartingale, but   again an Itô--Lévy process. 
    Particularly, one could consider \textit{L\'evy stochastic differential equations} instead of stochastic differential equations driven by semimartingales; see  {\S}\ref{Sec1.3} and \citeb{p. 10}{book:sulem}.
\end{remark}

\subsection{A special class: regulated Lévy martingale processes}\label{Sec1.2.2}
\noindent Let us introduce a specific class of Lévy processes, namely those of finite intensity with symmetric and bounded  jumps. 
%
%
%
%

\begin{definition}\label{def:regulated}
	Let $\mathscr L$ denote the class of $\mathbb R$-valued Lévy   processes $L=(L_t)_{t\geq 0}$ that are 
	\begin{enumerate}[\quad (\text{A}1)]
		\item  integrable and of zero mean, i.e. $\mathbb E |L_t|<\infty$ and $\mathbb E L_t =0$ for all $t\geq 0$.
\end{enumerate}
 Processes in  $\mathscr L$ are  called \name{Lévy martingales}. We denote by $\mathscr L^2$   the subclass of $\mathscr L$  consisting of all square integrable Lévy martingales.
 
 We subsequently write  $\mathscr{L}_{\textnormal{reg}}$ for the subclass of $\mathscr L^2$ of square integrable Lévy martingales satisfying  the additional properties that
\begin{enumerate}[\quad (\text{A}1)]\setcounter{enumi}{1}
		\item  they are of finite intensity, i.e., $\nu(\R)<\infty$ where $\nu$ is the associated Lévy measure, and
	\item  their jumps are $\mathbb P$-a.s. uniformly bounded, i.e., $|L_{t-}-L_t|\leq \zeta$ $\mathbb P$-a.s., for all $t\geq 0$, with some certain jump height $\zeta\geq 0.$
\end{enumerate}
Lévy martingales in \regulated\index{\regulated} are said to be \name{regulated Lévy martingales}.
\end{definition}

 Proposition \ref{cor:finitefirst} tells us that  a Lévy process is a martingale if and only if (A1) is satisfied. This agrees with our   definition of  $\mathscr L$. The previous proposition even yields that the  characteristic triplet of processes in $\mathscr L$ must be of the form
\begin{equation}\left(  \int_{\{|x| \leqslant 1\}} x \, \nu(\mathrm dx), \sigma^2,   \nu\right).\end{equation}
 In addition to   (A1),   assumption (A2) gives us    that a Lévy martingale $L=(L_t)_{t\geq 0}$ can be written as a    sum of two independent processes: a Brownian motion with dispersion coefficient $\sigma^2$ and a compound Poisson process with jump measure $\frac1{\nu(\R)}\nu$, i.e.,
 \begin{equation}
 	L_t=\sigma W_t+\sum_{k=1}^{N_t}Z_k,\quad t\geq 0.\label{eq:form}
 \end{equation}
 This is simply a consequence of Proposition \ref{prop:thecompound}.  
 Assumption (A3)   makes it possible to determine pathwise estimates of the integral. Typically, this assumption gives us
\begin{equation}
	\mathbb E[Z_1]=0\quad\text{and}\quad \mathbb E[Z_1^2]<\infty. \label{eq:square-cpp}
\end{equation}
 We conclude that assumptions (A2)  and  (A3) together yield that   Lévy martingales are square integrable, so assuming square integrability is not restrictive at all (because Brownian motions are square integrable, obviously).

Let us proceed by computing the quadratic variation of   regulated Lévy martingales, i.e., martingales $L=(L_t)_{t\geq 0}$ of the form \eqref{eq:form}. In the  below, we  write $C=(C_t)_{t\geq 0}$ for the compound Poisson part of $L$. Since quadratic covariations are symmetric bilinear forms, we obtain  the following:
\begin{equation}
	[L,L]_t=\sigma^2[ W,  W]_t+2\sigma[ W,C]_t+[C,C]_t=\sigma^2t+[C,C]_t,\quad t\geq 0,
\end{equation}
where it follows from Proposition \ref{prop:quadvarFV} that $[W,C]=0$ holds. 
By appealing to Theorem \ref{thm:ucpquad}  and Example \ref{ex:compound}, one  easily deduces that the quadratic variation process of $C$ and its compensator---that is, the predictable quadratic variation---are given by
	\begin{equation}
		[C,C] _t=\sum_{s \leq t}\left(\Delta C_{s}\right)^{2}=\sum_{k=1}^{N_{t}} Z_{k}^{2}\quad\text{and}\quad \langle C,C\rangle_t=\lambda t \mathbb E Z_1^2,\quad t\geq0.
	\end{equation}
	Hence, a regulated Lévy martingale has  quadratic variations
	\begin{equation}
		[L]_t=\sigma^2 t +\sum_{k=1}^{N_{t}} Z_{k}^{2}\quad\text{and}\quad \langle L\rangle _t=\mu t,
	\end{equation}
where 
\begin{equation}
	\mu:=\sigma^2+\lambda \mathbb E Z_1^2\leq \sigma^2+\lambda \zeta^2.
\end{equation}
It is interesting to note that the quadratic variation $[L]$ of a regulated Lévy martingale is again a Lévy process, namely, it is a compound Poisson process with drift.

	\begin{remark}\label{remark:doleansabs}
		Thanks to the pleasant nature of regulated Lévy martingales $L$, it is possible to consider stochastic integrals 
			\begin{equation}
		 \left(\int_0^tH_s\,\mathrm dL_s\right)_{t\geq 0}\
		\end{equation}
	with  predictable   and even progressively measurable processes $H=(H_t)_{t\geq 0}$ as integrands, recall {\S}\ref{Sec1.1.4}. In particular, the extension to stochastic integrals  with progressively measurable integrands is feasible---in lines with {\S}\ref{Sec1.1.4}---because the Doléans measure $\mu_L$ is absolutely continuous with respect to $\mathrm ds \times \mathbb P$, i.e., $\mu_L\ll \mathrm ds \times \mathbb P$. Indeed, we have \begin{align*}\mu_L((s,t]\times F_s)&=\mathbb E\big[\mathds 1_{F_s}([L]_t-[L]_s)\big]\\
		&=\mathbb E[\mathds 1_{F_s} ]\mathbb E\big[ [L]_t-[L]_s \big]\\
		&=\mathbb P(F_s)(\sigma^2+\lambda \mathbb E Z_1^2)(t-s)\\
		&=(\sigma^2+\lambda \mathbb E Z_1^2)(\mathrm ds \times \mathbb P)((s,t]\times F_s),\end{align*}
	for all predictable rectangles $(s,t]\times F_s \in \mathcal R,$ thus $F_s\in\mathcal F_s.$ Note, 	we used in the second equality the fact that the quadratic variation $[L]$ has increments independent of the past (because it is a Lévy process as well). We conclude
	\begin{equation}
		\mu_L= (\sigma^2+\lambda \mathbb E Z_1^2)(\mathrm ds \times \mathbb P),
	\end{equation}
	because     the Doléans measure on $\mathcal P$  is fully determined by the set of  predictable rectangles $\mathcal R$. This yields the absolute continuity with respect to $\mathrm ds\times \mathbb P$.
	Notice that  for  a standard Brownian motion $W$, the Doléans  measure $\mu_W$ equals the product measure $\mathrm ds\times \mathbb P.$ Consequently, from the perspective of stochastic integration, we see that a Brownian motion and regulated Lévy martingales are not really that different.

	Finally, conform   to the stochastic integration theory with respect to Brownian motion (see {\S}\ref{Sec1.1.4} for more details), we need the integrands   to  satisfy an $L^2$-integrability condition.  More specifically, for  expanding to both the predictable and progressive measrable processes, we require
		\begin{equation}
			\mathbb{E} \int_{0}^{t}|H_s|^{2}\, \mathrm d [L] _s=\mathbb{E} \int_{0}^{t}|H_s|^{2}\,\mathrm  d \langle L\rangle _s=\frac1\mu\mathbb{E} \int_{0}^{t}|H_s|^{2}\,\mathrm  ds<\infty,\quad t\geq 0.
		\end{equation}
We observe,   the $L^2$-integrability condition which we are supposed to impose on our integrands is independent of the   regulated Lévy martingale. For all $L\in$\,\regulated, we need   $\mathbb{E} \int_{0}^{t}|H_s|^{2}\,\mathrm  ds<\infty$ to hold for all $t\geq 0$.
%
%
%
	\end{remark}


We conclude that regulated Lévy martingales are not that  different from a Brownian motion, with regard to stochastic integration. 
It needs to be pointed out, once again, that the extension beyond predictably processes is non-trivial and may feel slightly counter-intuitive  (recall Remark \ref{remark:counter} and Example \ref{ex:generalise}). In the specific setting above, we can write
\begin{equation}
	\int_0^tH_s\,\mathrm dL_s=\sigma \int_0^tH_s\,\mathrm dW_s+\int_0^tH_s\,\mathrm dC_s,\quad t\geq 0,\label{eq:integral}
\end{equation}
where   no confusion can arise regarding the first integral in \eqref{eq:integral} on the right hand side. The second integral, on the other hand, can be either understood as a Lebesgue--Stieltjes integral or as a stochastic integral (since the integrator is also of finite variation, unlike Brownian motion). 
Unfortunately, they no longer coincide; see Example \ref{ex:generalise}.
Despite the   possible confusion, recall that restricting ourselves to predictable (or even merely adapted \caglad) processes gives us unambiguity; see Proposition \ref{prop:stieltjes} and Proposition \ref{prop:stieltjes2}.

\section{Stochastic  differential equations and  simulations}\label{Sec1.3}

\noindent  In the theory of (non-random) ordinary differential equations, coefficients are typically assumed to be Lipschitz continuous, since this ensures us the existence and   uniqueness of a solution. Why do we even bother?
Differential equations can model many important physical situations. Therefore,  it is relevant to know when these equations have unique solutions (and thus when solutions exist in the first place). Additionally, the last couple of decades we obtained lots of numerical  tools at our disposal. Whenever we want to solve a differential equation numerically, we should always
check for existence first, because otherwise we just end up with ``a
bunch of garbage''.

The Lipschitz continuity condition can, in fact, be relaxed.
Indeed, we still achieve existence and uniqueness  for locally Lipschitz continuous coefficients. Nevertheless, solutions  are then only  guaranteed to  exist up to a
possible explosion time. Consider the initial value problem
\begin{equation}
\begin{cases}
	\dfrac{\mathrm d x}{\mathrm dt}=x^2,\quad t\geq 0,\\[.1cm]
	x(0)=x_0,
\end{cases}
\end{equation}
with $x_0>0$, for instance. This equation has the unique solution $x({t})=\big(x^{-1}_0-t\big)^{-1}, t\geq 0,$ which explodes at time
$t=x^{-1}_0.$ We refer to \cite{book:hale} for the basic theory on existence and uniqueness of solutions to ordinary  and functional differential equations (e.g., delay differential equations).  

In the field of stochastic differential equations, the statements above   remain valid (possibly under    extra technical conditions). The reader is assumed to be familiar with \name{It\^o stochastic differential equations}{stochastic differential equation!of Itô type}, i.e.,
\begin{equation}
	\mathrm d X_t=a(X_t,t)\,\mathrm dt+b(X_t,t)\,\mathrm dW_t,\label{eq:brw}
\end{equation}
where $W=(W_t)_{t\geq 0}$ is a standard Brownian motion. Suitable references are, e.g., \cite{book:karatzas,book:mao,book:revuz}. An extensive review on Itô stochastic differential equations is also available in \cite{article:platen}.

Subsequently, we now want to allow jumps in equation \eqref{eq:brw}. There is only no universal way of achieving this. In line with \cite{book:protter}, for example, we can consider
\begin{equation}
	\mathrm d X_t=a(X_t,t)\,\mathrm dt+b(X_{t-},t)\,\mathrm dZ_t,\label{eq:sm}
\end{equation}
where we replace Brownian motion with a general semimartingale $Z=(Z_t)_{t\geq 0}$. For sufficient (regularity) conditions on $a,b$, and the initial value $X_0$, guaranteeing  existence and uniqueness of solutions, we refer to \citeb{Thm. V.6}{book:protter}.   
In case one wants to take for example delay into account, we refer to   \citeb{Thm. V.7}{book:protter} (of which \citeb{Thm. V.6}{book:protter} is a trivial corollary). 

Alternatively, in accordance with \cite{book:sulem} and \cite{book:applebaum}, one can consider \name{Lévy stochastic differential equations}{stochastic differential equation!of Lévy type}, i.e.,
\begin{equation}\mathrm d X_t=\alpha(X_t,t) \, \mathrm d t+\beta( X_t,t) \, \mathrm d W_t+\int_{\mathbb{R}} \gamma\left( X_{t{-}}, t, z\right) \bar \mu(\mathrm d t, \mathrm d z).\label{eq:levyeq}\end{equation}
Recall the discussion at the end of {\S}\ref{Sec1.2.1}. Sufficient (regularity) conditions on $\alpha,\beta,$ and $\gamma$ which ensure  existence and uniqueness of solutions to   equation \eqref{eq:levyeq}, can be found  \citeb{Sec. 6.2}{book:applebaum}; more specifically, see     \citeb{Thm. 6.2.9}{book:applebaum} and  \citeb{Thm. 6.2.11}{book:applebaum}, where they  deal with Lipschitz and local Lipschitz continuity, respectively.


In the remainder of this section, we   briefly discuss   numerical schemes for stochastic differential equations as in \eqref{eq:sm}. At the end, we then motivate a numerical scheme for stochastic delay differential equations. Recall that the well-posedness (i.e., existence and uniqueness of solutions) can be assured by means of \cite[Thm. V.7]{book:protter}. Moreover, from now on we  will consistently write any process $X$ as  $(X(t))_{t\geq 0}$ instead of  $ (X_t)_{t\geq 0}$, for notational purposes.

 Before we address general numerical schemes, we test our approach with a stochastic differential equation that has a known explicit solution (analogous to Example \ref{ex:interm}). For instance, let us consider  the famous \name{Black--Scholes equation} \citeb{p. 329}{book:applebaum}:
 \begin{equation}
 \mathrm d Y(t)=\alpha Y(t) \,\mathrm d t+\beta Y(t-)\,\mathrm  d L(t) = Y(t-)\,\mathrm d\tilde L(t),\quad Y(0)=Y_0,\label{eq:BS}
 \end{equation}
for $\alpha,\beta\in\R$ and
   $L$ a Lévy process with triplet $(b,\sigma^2,\nu)$. Note that $\tilde L=\alpha t+\beta L$ is a Lévy process as well; its triplet equals $(\beta(b+\alpha),\beta^2\sigma^2,\beta\nu)$. By means of Theorem \ref{thm:dade}, we obtain that the explicit solution of equation \eqref{eq:BS} is given by\footnote{The solution $Y=(Y(t))_{t\geq 0}$ to the Black--Scholes equation   \eqref{eq:BS} remains to be non-negative  if  the Lévy  measure satisfies $\text{supp}(\nu)\subset [-\beta^{-1},\infty)$. If $\sup_{t\geq 0}|\Delta L(t)|\leq \beta^{-1}$ holds, e.g., then $\text{supp}(\nu)\subset [-\beta^{-1},\beta^{-1}]$.  }
 \begin{equation}
 	Y(t)=Y_0\exp\Big(\big(\alpha-\tfrac12\beta^{2}\sigma^2\big) t+\beta L(t)\Big)\prod_{0<s\leq t}(1+\beta\Delta L(s))\exp\big(-\beta\Delta L(s)\big),\quad t\geq 0.\label{eq:MC}
 \end{equation}
 Indeed, by means of Proposition \ref{prop:quadvarFV} and the Lévy--Itô decomposition in Theorem \ref{thm:LevyIto}, we deduce $[\tilde L]^c_t=\beta^2[L]^c_t=\beta^2\sigma^2$ for all $t\geq 0.$ In particular, we observe   the solution  exists for all  time, thus there is an infinite explosion time. No explosions are actually already guaranteed in advance, because the coefficients in equation \eqref{eq:BS} are Lipschitz continuous.
 
 Recall that \eqref{eq:BS} is shorthand notation for
 \begin{equation}
 	Y(t)=Y_0+\int_0^t\alpha Y(s)\,\mathrm ds+\int_0^t\beta Y(s-)\,\mathrm dL(s),\quad t\geq 0.
 \end{equation}
For practical purposes, we take $Y_0=y_0\in \R$ and consider the stochastic process $Y$ on a finite time horizon $[0,T]$ only, which we partition equidistantly: 
\begin{equation}
	0=\tau _{{0}}<\tau _{{1}}<\cdots <\tau _{{N}}=T\quad\text{and}\quad\Delta t=T/N,\quad N\in\N.
\end{equation} 
 Note $\tau_k=k\Delta t.$ Theorem \ref{thm:approx} then suggests the  recurrence relation
 \begin{equation}
 Y_{n+1}=Y_{n}+    \alpha Y_{n}\Delta t  +\beta Y_{n} \Delta L_{n}, \quad Y_0=y_0,\label{eq:ns1}
 \end{equation}
for all $0\leq n< N$, where $\Delta L_n=L(\tau_{n+1})-L(\tau_n)\sim L(\tau_1).$ As we will discuss after Figure \ref{fig:stability}  and Figure \ref{fig:stability2}, this approach  corresponds to the {Euler-Maruyama method}. More generally, one could introduce the  general discretisation 
  \begin{equation}
 	Y_{n+1}=Y_{n}+(1-\theta)   \alpha Y_{n}\Delta t+\theta   \alpha Y_{n+1}\Delta t+\beta Y_{n} \Delta L_{n},\quad Y_0=y_0,\label{eq:ns2}
 \end{equation}
with $ \theta\in[0,1], $ which resembles the $\theta$-method. For this typical example, an explicit expression for $Y_{n+1}$ is possible and  reads  
 \begin{equation}
	 	Y_{n+1} =
 	(1-\theta   \alpha \Delta t)^{-1}(1+(1-\theta) \alpha \Delta t+ \beta \Delta L_n)Y_n,
 	 \end{equation}
  hence
  \begin{equation}
  	Y_{n+1}=y_0(1-\theta  \alpha \Delta t)^{-n}\prod_{k=1}^n(1+(1-\theta)\alpha\Delta t + \beta \Delta L_k).\label{compare1}
  \end{equation}
 	On the other hand, if we restrict to Lévy jump diffusion process, i.e.,
 	\begin{equation}
 		L(t)=b t + \sigma W(t)+Z(t),\quad Z(t)=\sum_{k=1}^{N({t})} Z_{k},\quad t\geq 0,
 	\end{equation}
 	  a direct Monte Carlo simulation of the explicit solution in \eqref{eq:MC} yields
 	\begin{equation}
 		Y_n^*=y_0\exp\Big(\big(\alpha-\tfrac12\beta^{2}\sigma^2\big) n\Delta t+\beta L(n\Delta t)\Big)\prod_{k=1}^n(1+\beta\Delta Z_k)\exp\big(-\beta\Delta Z_k\big),\label{compare2}
 	\end{equation}
 where $\Delta Z_n=Z(\tau_{n+1})-Z(\tau_n)$. This is conform the numerics in Appendix \ref{matlab-BS}.

 Based on the comments in {\S}\ref{Sec1.2.1}, it is fine to presume  $Y_n^*\approx Y(\tau_n)$ whenever $\Delta t$ is sufficiently small (but then again not too small,  in order to prevent rounding errors).  Also, for a fixed Lévy process $L$, we observe that $\beta$ cannot be too large either, since this would affect the numerical stability in \eqref{compare1}. It is  to be expected  that too many (large) jumps will cause instabilities as well.
   Figures \ref{fig:stability} and  \ref{fig:stability2} compare  the $Y_n$ in \eqref{compare1} and $Y_n^*$ in \eqref{compare2} for $\theta=0$ (Euler  method), for $\theta=1$ (Euler backwards method) and $\theta=\frac12$, respectively, in which we deliberately  take $\Delta t$ somewhat large.

 \begin{figure}[!t]
 	\centering
 	\includegraphics[width=0.525\linewidth]{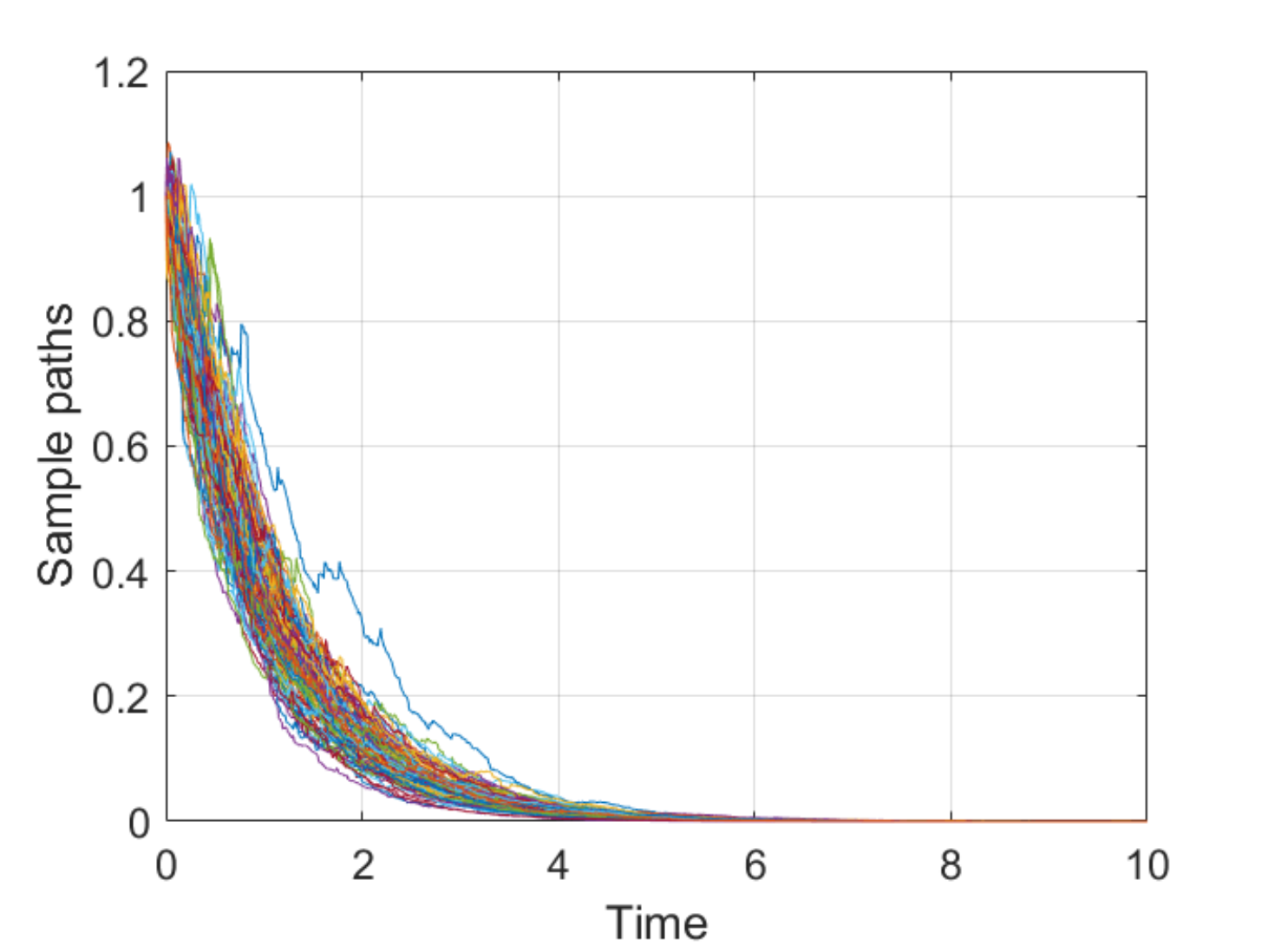}
 	 	\includegraphics[width=0.46\linewidth]{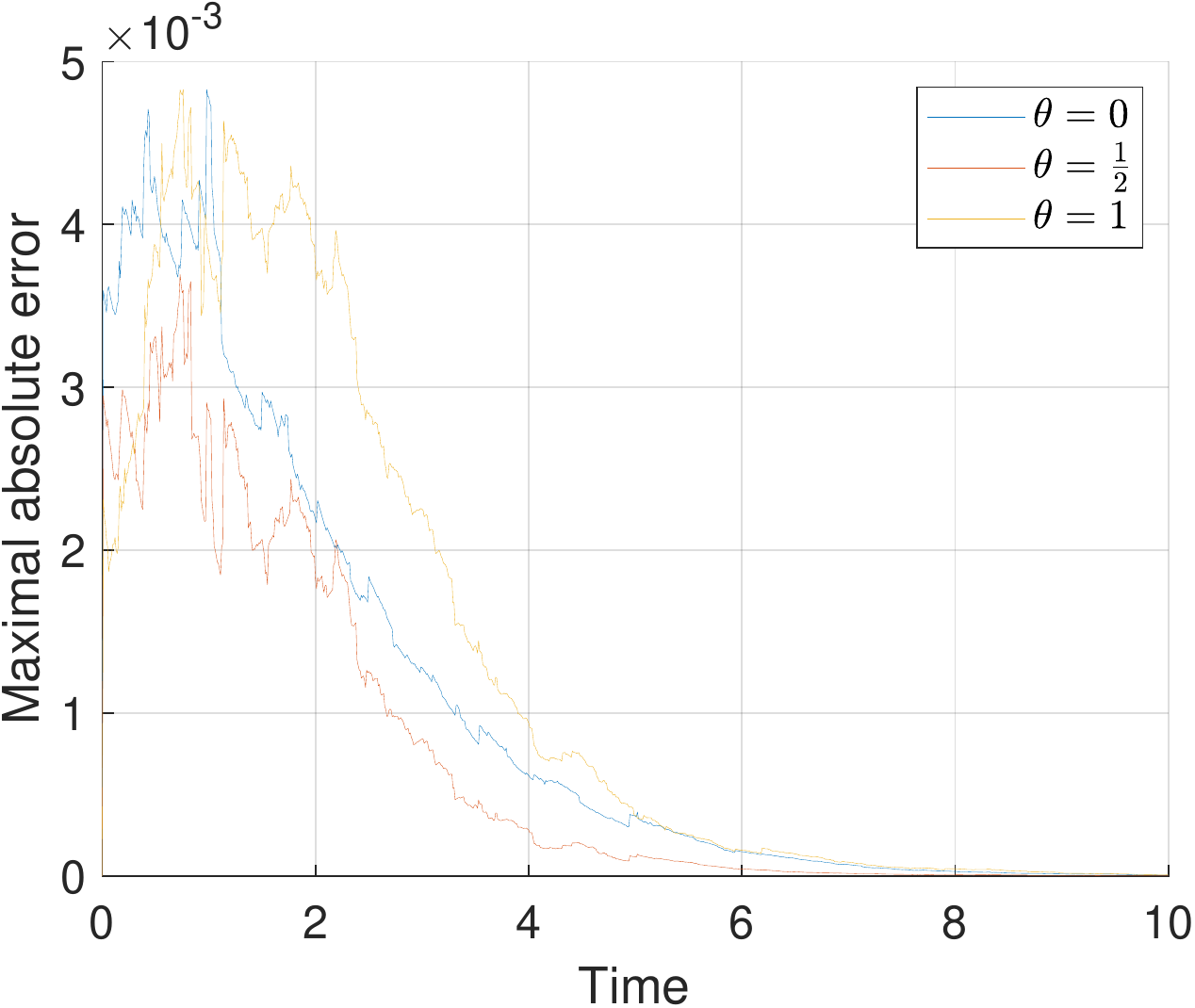}
 	\caption{On the left, a plot of 100 sample paths of the Black-Scholes equation \eqref{eq:BS} with $\alpha=-1$, $\beta=10^{-1},$ and $L=(L(t))_{t\geq 0}$ a Lévy jump diffusion process where $b=0$, $\sigma^2=1$, $Z_1\sim $ Unif $[-1,1]$, and intensity parameter $\lambda=\nu(\R)=\nu([-1,1])=10$ (obtained via the Monte Carlo scheme in \eqref{compare2}). In here, the step size is taken to be $\Delta t=10^{-2}.$ On the right,   the absolute difference $n\mapsto |Y_n-Y_n^*|$  is plotted of the particular realisation attaining the highest error (assuming $Y_n^*\approx Y(\tau_n)$).}
 	\label{fig:stability}  
 \end{figure}
 \begin{figure}[!t]
	\centering
	 	\includegraphics[width=0.52\linewidth]{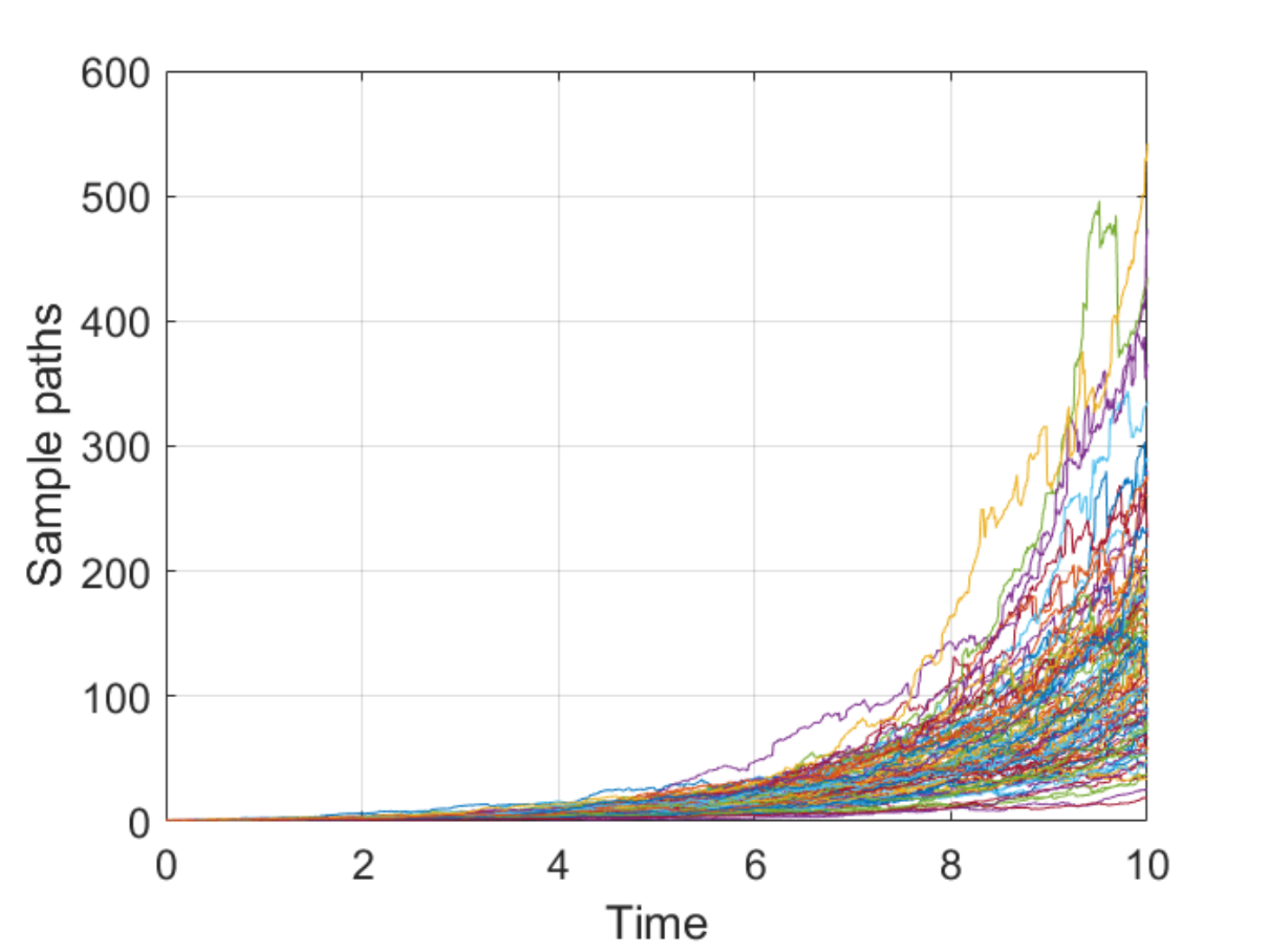}
	\includegraphics[width=0.4675\linewidth]{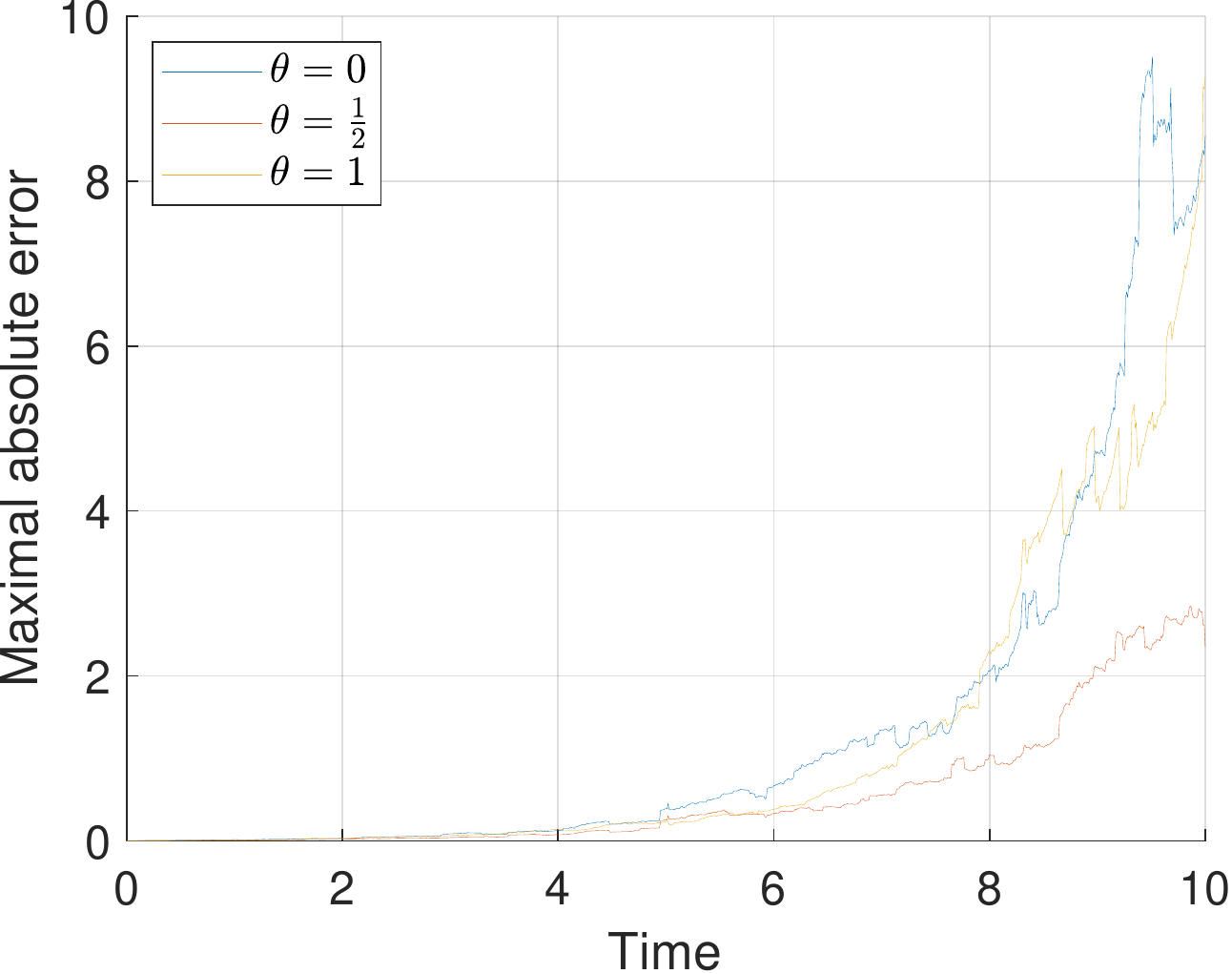}	\caption{The same as in Figure \ref{fig:stability}, but with $\alpha=0.5$ instead of $\alpha=-1.$}
	\label{fig:stability2}
\end{figure}

\pagebreak
 Based on the simulations from above, we are confident that the numerical schemes in  \eqref{eq:ns1} and \eqref{eq:ns2} are reliable.  Further numerical investigation (not presented in here) shows us that when $\Delta t$ decreases, the maximal absolute error decreases also  (to a certain extend of course). However, for $\Delta t$ sufficiently small, there appears to be no  real significant difference  between the $\theta$-methods; already for $\Delta t=10^{-2}$ it is debatable whether $\theta=\frac12$ is ``much better'' or not. We claim that this is due to the presence of noise (with too large volatility).
  In the sequel, it is therefore reasonable to restrict ourselves to $\theta=0$.

Let us now return to the  more general stochastic differential equation
\begin{equation}
	\mathrm d X(t)=a(X(t),t)\,\mathrm dt+b(X(t-),t)\,\mathrm dL(t),
\end{equation}
where $L=(L(t))_{t\geq 0}$ is a Lévy process. Then, completely conform to the specific example  above,  the \name{Euler--Maruyuma method} yields the following recurrence relation:
\begin{equation}
	X_{n+1}=X_n+a(X_n,\tau_n)\Delta t+b(X_n,\tau_n)\Delta L_n.
\end{equation} 
Numerous papers and textbooks have been written on convergence/error analysis of numerical schemes for stochastic differential equations. Within the continuous setting, i.e., where $L$ is a standard Brownian motion, we refer to \cite{book:kloeden}, \cite{article:Giles}, and \cite{article:platen}---which are not limited to Euler discretisation only---to mention only a few references. The latter is a concise review and provides neat lists of references. For   stochastic  differential equations driven by (a certain class of) Lévy processes, one may consult, e.g., \cite{article:talay}, \cite{article:jacod}, \cite{article:Kuhn}, and  \cite{article:FCKS}.  
Partitioning equidistantly is of course not necessary;  in \cite{article:FCKS} one   considers  a Poisson  distributed partition of $[0,T]$.

Ultimately, one may be interested in stochastic \textit{functional} differential equations such as, but not limited to, stochastic delay differential equations. Loosely speaking, a typical   stochastic delay equation would be given by
\begin{equation}
	\mathrm d X(t)=a(X(t),X(t-1),t)\,\mathrm dt+b(X(t-), X((t-1)-),t)\,\mathrm dL(t).
\end{equation}
Suppose there is an $M\in\N$ such that $M\Delta t=1$, then the above suggests the following numerical method:
\begin{equation}
	X_{n+1}=X_n+a(X_n,X_{n-M},\tau_n)\Delta t+b(X_n,X_{n-M},\tau_n)\Delta L_n.
\end{equation}
The corresponding code can be found in  Appendix \ref{matlab-SDDE-Main}.
We state   an exhaustive list with interesting references  
\cite{article:agrawal,book:asymptotic,article:buckwar,article:Kuske,article:jacob,article:kloeden,phdthesis:kumar,phdthesis:norton,article:tan,article:zhao,article:Hu,article:zong}.
The main focus differs per reference though, where it may be  on locally Lipschitz coefficients,   non-autonomous systems, and/or the presence of jumps.


\printindex

\appendix
\section{MATLAB codes}\label{B}
\noindent In this appendix we provide several MATLAB codes. All  figures in these notes are replicable  with the help of these codes below. 
%

\subsection{Main programme (driver.m)}\label{matlab-driver}
\noindent This code delivers numerical  approximations of realisations of many possible Lévy jump diffusions. It does this  by making use of the elementary functions   in Appendix \ref{elementary}.  

\begin{lstlisting}
clear all; clc; close all;
%rng('default'); %All figures are created under this specific seed

size_h   = 0.01;                  %Time step size (usually 0.001 in plots)
length_I = 50;                     %Length of observed time interval
amount_n = 5;                    %Amount of realisations
steps_T  = round(length_I/size_h); %Total amount of steps

t = 0:size_h:length_I;             %Grid on [0,length_I]; 
t_delay = -1:size_h:length_I;      %Grid on [-1,length_I];    

%% Visual preferences %%
set(0,'DefaultLineLineWidth',.1);
set(0,'DefaultAxesLineStyleOrder',{'-'})
set(0,'DefaultAxesFontSize',14)

%% Initiation of Brownian motion with drift on [0,length_I] %%
var_mu = 0;    %Drift
var_sigma = 1; %Variance

[dW,W]=drifted_Brownian_motion(size_h,steps_T,amount_n,var_mu,var_sigma);

%% Initiation of various compound Poisson processes on [0,length_I] %%
p = 1;
dist1 = p*randn(amount_n,steps_T);              %N(0,p^2), EZ_1=0
dist2 = -p + 2*p*rand(amount_n,steps_T);        %Unif[-p,p], EZ_1=0
dist3 = -p + 2*p*randi([0,1],amount_n,steps_T); %-p or p 50%, EZ_1=0
dist4 = p*ones(amount_n,steps_T);               %Scaled Poisson process
dist5 = 0;                                      %Make your own distribution

var_lambda = 1;                                 %Poisson intensity
[dC,C]=compound_Poisson_process(size_h,steps_T,amount_n,var_lambda,dist1);
%Change the last argument into the preferred distribution

%% The considered Levy process on [0,length_I] %%
L = W + C; dL = dW + dC;

figure(1) %Plot of the Levy process L
for i=1:amount_n
    plot(t,L(i,:))
    xlabel('Time'); ylabel('Sample paths'); title('Levy process')
    hold on; grid on
end
hold off
\end{lstlisting}

\subsection{Elementary functions}\label{elementary}
\noindent In this section,  we state  numerical schemes for (drifted) Brownian motions, Poisson processes,  and compound Poisson processes, respectively. Note that a few comments on  these numerical schemes  can be found throughout {\S}\ref{Sec1.2}, including various relevant references.
\subsubsection{Brownian motion with drift}\label{matlab-drift}

\begin{lstlisting}
function [dW,W] = drifted_Brownian_motion (h,T,n,mu,sigma)
    dW = mu*h*ones(n,T) + sqrt(h)*sigma*randn(n,T); 
    %randn(A,B) returns an AxB matrix with N(0,1)-samples
    W  = [zeros(n,1) cumsum(dW,2)];
    %cumsum(A,2) returns the cumulative sum of each row of matrix A
end
\end{lstlisting}

\subsubsection{Poisson process}\label{matlab-poisson}
\noindent Just for the sake of completeness, we provide a MATLAB code for   Poisson processes.  One  can namely  use the more general code on compound Poisson processes; see Appendix \ref{matlab-compound} below.

\begin{lstlisting}
function [dN,N] = Poisson_process (h,T,n,lambda)
    dN=zeros(n,T); 
    for i=1:T
        next_jump = zeros(n,1);
        final = 0;
        while (final == 0) %Allows multiple jumps in one single time step
            next_jump = next_jump - log(rand(n,1))/lambda;
            %U ~ Unif[0,1] then -log(U)/k ~ Exp(k), for k>0
            %rand(A,B) returns an AxB matrix with  Unif[0,1]-samples
            %First jump occurence of a Poisson process ~ Exp(lambda)
            final = 1;
            for j=1:n
                if (next_jump(j) < h) %Checks jump happens in [t_i,t_{i+1}]
                    dN(j,i) = dN(j,i) + 1; 
                    final = 0; 
                end                
            end
            %If ''final = 0'', perhaps another jump within [t_i,t_{i+1}]
        end        
    end
    N  = [zeros(n,1) cumsum(dN,2)]; 
    %cumsum(A,2) returns the cumulative sum of each row of matrix A
end
\end{lstlisting}


\subsubsection{Compound Poisson process}\label{matlab-compound}
\begin{lstlisting}
function [dC,C] = compound_Poisson_process (h,T,n,lambda,distribution_Z)
    dC=zeros(n,T);  
    for i=1:T
        total = zeros(n,1);
        final = 0;
        while (final == 0) %See Poisson_Process.m for more information
            total = total - log(rand(n,1))/lambda; 
            final = 1;
            for j=1:n
                if (total(j) < h) 
                    dC(j,i) = dC(j,i) + distribution_Z(j,i); 
                    %distribution_Z is identically 1 for a Poisson process
                    final = 0;
                end                
            end
        end        
    end
    C  = [zeros(n,1) cumsum(dC,2)]; 
    %cumsum(A,2) returns the cumulative sum of each row of matrix A
end
\end{lstlisting}

\textbf{Disclaimer:} It must   be noted that this code   allows multiple jump occurrences in one time step, as it also should. However, it may   cause  implausible results. For example, there is a possi-\linebreak bility that  realisations of a stochastic   differential equation can become negative while actually the solution must remain positive everywhere.  We felt no need to contribute a fix, because such\linebreak numerical inaccuracies were simply avoided. 

\

\subsection{Simulating stochastic  delay  differential equations and invariant measures}\label{matlab-SDDE-Main}
\noindent With the help of this code, one  can obtain  path simulations  of the solution to a stochastic delay\linebreak differential equation  of the following type:
\begin{equation}
\label{B1}	\mathrm d X(t)=f(X(t),X(t-1))\,\mathrm dt+g(X(t-), X((t-1)-))\,\mathrm dL(t),
\end{equation}
under sufficient conditions on $f$ and $g$. The function SDDE.m is to be found in Appendix \ref{matlab-SDDE}. In addition, the code below can illustrate a numerical approximation of the invariant measure obtained from   Krylov--Bogoliubov's method (if it exists).

\begin{lstlisting}
%We consider an autonomous SDDE dY(t)=f(y_t)dt+g(y_t-)dL(t). Changing the
%functions f and g can be done at the end of this code. Notice that we only
%allow input values x(t) and x(t-1) for f and g. Not using x_delay, i.e.,
%the value of x(t-1), results into a numerical scheme for ordinary SDEs.

%% Simulating stochastic delay differential equations %%
driver      %Retrieving the input of driver.m

art = true; %Applying it to Mackey-Glass type equations, one may require 
            %the artificial drift term condition, see subsection 4.1.3.
            %Note that 'true' yields a=-sigma^2*b^2/2, 'false' yields a=0.
            %For a more general a, one simply needs to change func_f.
            
Y=SDDE(size_h,steps_T,amount_n,dL,@int_cond,@func_f,@func_g,var_sigma,art);

figure(2) %Plot of the solution Y
for i=1:amount_n
    plot(t_delay,Y(i,:))
    xlabel('Time'); xlim([-1,length_I]); ylabel('Solution Y(t)'); title('SDDE')
    hold on; grid on
end
hold off

figure(3) %Check whether k(t)>0
length_one = 1/size_h;
t_diff = 0:size_h:length_I;
Y_diff = -(Y(:,length_one+1:end)-Y(:,1:end-length_one));

for i=1:amount_n
    plot(t_diff,Y_diff(i,:))
    xlabel('Time'); xlim([0,length_I]); ylabel('Solution Y(t)'); title('k(t)>0?')
    hold on; grid on
end

gamma = 1;
plot(t_diff,-gamma*ones(1,length(t_diff)));

%% Numerical estimation of invariant measure (if it exists) %% 
% T = 4630;                       %Most accurate if T equals 'period' of sol.
% A = Y(:,length(Y)-T:end);
% A = A'; A = reshape(A,1,[]);
% number_of_bins_histogram = 100; %Total amount of visible bins of histogram
% kernel_width = 0.01;            %Making this number smaller causes the red
%                                 %density curve to (over)fit the data
%                                 
% figure(3) %Plot of the invariant measure obtained from Krylov-Bogoliubov
% histfitnew(A,number_of_bins_histogram,'kernel',kernel_width)
% xlabel('Values of Y(t) over [50-T,50]'); ylabel('Frequency')
% title('Invariant measure')
% hold off
%We used the 'histfitnew' function here. This is simply the built-in MATLAB
%function 'histfit', where we have taken the following slight modification 
%in line 62: pd = fitdist(data,dist,'width',w);. We have changed the first 
%line of the MATLAB function 'histfit' accordingly. 

%% The main functions f and g and the initial condition %%
function [f] = func_f(x,x_delay) %x_delay = x(t-1)
    gamma=1; r=10; p=6; q=1;
    f=r.*exp(q.*x_delay-x)./(1+exp(p.*x_delay))-gamma; %A typical example
end

function [g] = func_g(x,x_delay) %x_delay = x(t-1)
    g=0.01; %For simplicity, another example: min(x_delay.^2,sin(x));
end

function [initial] = int_cond(x)
    initial=-3; %For simplicity, another example: exp(x)*sin(2*pi.*x);
end
\end{lstlisting}

\subsubsection{Numerical integration scheme for a SDDE}\label{matlab-SDDE}
\noindent Inspired by the Euler--Maruyama method,  this code   approximates the solution of \eqref{B1} via 
\begin{equation}
	X_{n+1}=X_n+f(X_n,X_{n-M})\Delta t+g(X_n,X_{n-M})\Delta L_n,
\end{equation}
where $M\in\N$ such that $M\Delta t=1$. A justification of this numerical integration scheme, which includes many references and comments on the accuracy,  can be found in {\S}\ref{Sec1.3}.

\begin{lstlisting}
function [Y] = SDDE (h,T,n,dL,initial,f,g,sigma,artificial)
    %We consider the autonomous SDDE dY(t)=f(y_t)dt+g(y_t-)dL(t). Notice 
    %that we only consider the input values x(t) and x(t-1) for f and g.
    N = round(1/h); dL = [zeros(n,N) dL]; %Expanding to [-1,length_I]
    Y = zeros(n,N+1+T);
    
    %Giving Y the preferred initial condition
    for i = 1:N+1
        Y(:,i)=initial(-(N+1-i)*h);
    end
    
    %Approximating realisations of the solution of the SDDE in question
    for i = N+1:N+T
        drift = f(Y(:,i),Y(:,i-N))*h;
        if (artificial == true) %Relevant to our Mackey-Glass equations  
                                %only, see subsection 4.1.3. for more info
            compensation = drift + -sigma^2/2*g(Y(:,i),Y(:,i-N)).^2*h;
            drift = compensation;
        end
        volatility = g(Y(:,i),Y(:,i-N)).*dL(:,i);
        Y(:,i+1)=Y(:,i) + drift + volatility;
    end

\end{lstlisting}

\subsection{Additional programmes}\label{add}
\subsubsection{Stochastic integrals with a deterministic integrand}\label{matlab-fdW}
\noindent This code gives us a numerical approximation of the stochastic integral   $\int_0^t f(s) \,\mathrm dL(s) $, where $f$ can be any (sufficiently regular) deterministic process and where $L$ is a Lévy process.

\begin{lstlisting}
driver %Retrieving the input of driver.m

%int_0^t f(s)dL(s) approximate solution equals sum_i f(ti)dL_{ti}, where
%dL_{ti}=L_{ti+1}-L_{ti}. Changing f can be done at the end of this code.
f_ito_integral=zeros(amount_n,steps_T);
for i = 1:steps_T
    f_ito_integral(:,i)=func(i*size_h).*ones(amount_n,1);
end
SumfdL=cumsum(f_ito_integral.*dL,2); intfdL_approx=[SumfdL SumfdL(:,end)];
%cumsum(A,2) returns the cumulative sum of each row of matrix A

figure(2) %Approximation of int f(s)dL(s) via partitioning, see
          %Theorem 2.1.20 in particular
for i=1:amount_n-1
    plot(t,intfdL_approx(i,:))
    xlabel('Time'); ylabel('Sample paths')
    title('Direct numerical approximation of $\int_0^t f(s)dL(s)$', ...
          'Interpreter','Latex')
    hold on; grid on
end
hold off

function [f] = func(t)
    f=.5*exp(-sin(t)); %Specific example
end
\end{lstlisting}

\subsubsection{Integrals   with a stochastic integrand: a typical example}\label{matlab-WdW}
\noindent In this code, we approximate the typical stochastic integral $\int_0^t W \,\mathrm dW $ by means of two different approaches, where $W$ is a standard Brownian motion. See also Example \ref{ex:interm}.
\begin{lstlisting}
driver %Retrieving the input of driver.m
intWdW=1/2*(W.^2-t); %Exact solution

figure(2) %An approximation of int_0^t W_tdW_t using the exact solution
for i=1:amount_n
    plot(t,intWdW(i,:))
    xlabel('Time'); ylabel('Sample paths')
    title('Using exact solution $\int_0^t WdW=\frac12(W_t^2-t)$', ...
          'Interpreter','Latex')
    hold on; grid on
end
hold off

%int_0^t W_tdW_t approximate solution equals sum_i W_{ti}dW_{ti}, where
%dW_{ti}=W_{ti+1}-W_{ti}
W_ito_integral=W(:,1:end-1); SumWdW=cumsum(W_ito_integral.*dW,2);
%cumsum(A,2) returns the cumulative sum of each row of matrix A
intWdW_approx=[SumWdW SumWdW(:,end)];

figure(3) %An approximation of int_0^t W_tdW_t via partitioning, see
          %Theorem 2.1.20 in particular
for i=1:amount_n-1
    plot(t,intWdW_approx(i,:))
    xlabel('Time'); ylabel('Sample paths')
    title('Direct numerical approximation of $\int_0^t WdW$', ...
          'Interpreter','Latex')
    hold on; grid on
end
hold off
\end{lstlisting}

\subsubsection{On a more general numerical integration scheme}\label{matlab-BS}
\noindent This particular code was used to obtain the figures in {\S}\ref{Sec1.3}. Observe that one can also use the numerical integration scheme in Appendix  \ref{matlab-SDDE-Main} for the  $\theta=0$ scenario.
\begin{lstlisting}
driver %Retrieving the input of driver.m

theta00 = 0.0; %Euler forward (or Euler-Maruyuma) method 
theta05 = 0.5; %Trapezoidal method
theta10 = 1.0; %Euler backwards method

Y0 = 1; alpha = -1; beta = .1; %Initial condition and parameters for BS

%Direct approximation of real solution to Black-Scholes equation
Inf_conv_prod = [ones(amount_n,1) cumprod((1+beta*dC).*exp(-beta*dC),2)]; 
%Note that Delta L(s)=dC(s) nummerically
Yreal=Y0*exp((alpha-1/2*beta^2*var_sigma^2)*t+beta*L).*Inf_conv_prod;

figure(2)
for i=1:amount_n 
    plot(t,Yreal(i,:))
    xlabel('Time'); ylabel('Sample paths'); title('Black-Scholes')
    hold on; grid on
end
hold off

%Approximate Black-Scholes solution via several different theta-methods
Temp00=1/(1-theta00*size_h*alpha)*(1+(1-theta00)*size_h*alpha+beta*dL);
Y00 = [Y0*ones(amount_n,1) cumprod(Temp00,2)]; %theta=0.0
Temp05=1/(1-theta05*size_h*alpha)*(1+(1-theta05)*size_h*alpha+beta*dL);
Y05 = [Y0*ones(amount_n,1) cumprod(Temp05,2)]; %theta=0.5
Temp10=1/(1-theta10*size_h*alpha)*(1+(1-theta10)*size_h*alpha+beta*dL);
Y10 = [Y0*ones(amount_n,1) cumprod(Temp10,2)]; %theta=1.0

figure(3) %Difference between theta-method approximation and direct approx.
hold on; grid on
plot(t,max(abs(Y00-Yreal)))
plot(t,max(abs(Y05-Yreal)))
plot(t,max(abs(Y10-Yreal)))
xlabel('Time'); ylabel('Maximal absolute error')
title('The $\theta$-methods','Interpreter','latex')
legend('$\theta=0$','$\theta=\frac12$','$\theta=1$','Interpreter','latex')
hold off
\end{lstlisting}


\printbibliography
\end{document}